\documentclass[12pt,reqno]{amsart}
\usepackage{latexsym}
\usepackage{amssymb}
\usepackage{amscd}

\catcode`\@=11

\thinlines
\newskip\Einheit \Einheit=0.6cm
\newcount\xcoord \newcount\ycoord
\newdimen\xdim \newdimen\ydim \newdimen\PfadD@cke \newdimen\Pfadd@cke
\def\PfadDicke#1{\PfadD@cke#1 \divide\PfadD@cke by2 \Pfadd@cke\PfadD@cke \multiply\PfadD@cke by2}
\long\def\LOOP#1\REPEAT{\def\BODY{#1}\ITERATE}
\def\ITERATE{\BODY \let\next\ITERATE \else\let\next\relax\fi \next}
\let\REPEAT=\fi
\def\Punkt{\hbox{\raise-2pt\hbox to0pt{\hss\scriptsize$\bullet$\hss}}}
\def\DuennPunkt(#1,#2){\unskip
  \raise#2 \Einheit\hbox to0pt{\hskip#1 \Einheit
          \raise-2.5pt\hbox to0pt{\hss\normalsize$\bullet$\hss}\hss}}
\def\NormalPunkt(#1,#2){\unskip
  \raise#2 \Einheit\hbox to0pt{\hskip#1 \Einheit
          \raise-3pt\hbox to0pt{\hss\large$\bullet$\hss}\hss}}
\def\DickPunkt(#1,#2){\unskip
  \raise#2 \Einheit\hbox to0pt{\hskip#1 \Einheit
          \raise-4pt\hbox to0pt{\hss\Large$\bullet$\hss}\hss}}
\def\Kreis(#1,#2){\unskip
  \raise#2 \Einheit\hbox to0pt{\hskip#1 \Einheit
          \raise-4pt\hbox to0pt{\hss\Large$\circ$\hss}\hss}}
\def\Diagonale(#1,#2)#3{\unskip\leavevmode
  \xcoord#1\relax \ycoord#2\relax
      \raise\ycoord \Einheit\hbox to0pt{\hskip\xcoord \Einheit
         \unitlength\Einheit
         \line(1,1){#3}\hss}}
\def\AntiDiagonale(#1,#2)#3{\unskip\leavevmode
  \xcoord#1\relax \ycoord#2\relax %\advance\xcoord by -0.05\relax
      \raise\ycoord \Einheit\hbox to0pt{\hskip\xcoord \Einheit
         \unitlength\Einheit
         \line(1,-1){#3}\hss}}
\def\Pfad(#1,#2),#3\endPfad{\unskip\leavevmode
  \xcoord#1 \ycoord#2 \thicklines\ZeichnePfad#3\endPfad\thinlines}
\def\ZeichnePfad#1{\ifx#1\endPfad\let\next\relax
  \else\let\next\ZeichnePfad
    \ifnum#1=1
      \raise\ycoord \Einheit\hbox to0pt{\hskip\xcoord \Einheit
         \vrule height\Pfadd@cke width1 \Einheit depth\Pfadd@cke\hss}%
      \advance\xcoord by 1
    \else\ifnum#1=2
      \raise\ycoord \Einheit\hbox to0pt{\hskip\xcoord \Einheit
        \hbox{\hskip-\PfadD@cke\vrule height1 \Einheit width\PfadD@cke depth0pt}\hss}%
      \advance\ycoord by 1
    \else\ifnum#1=3
      \raise\ycoord \Einheit\hbox to0pt{\hskip\xcoord \Einheit
         \unitlength\Einheit
         \line(1,1){1}\hss}
      \advance\xcoord by 1
      \advance\ycoord by 1
    \else\ifnum#1=4
      \raise\ycoord \Einheit\hbox to0pt{\hskip\xcoord \Einheit
         \unitlength\Einheit
         \line(1,-1){1}\hss}
      \advance\xcoord by 1
      \advance\ycoord by -1
    \else\ifnum#1=5
      \advance\xcoord by -1
      \raise\ycoord \Einheit\hbox to0pt{\hskip\xcoord \Einheit
         \vrule height\Pfadd@cke width1 \Einheit depth\Pfadd@cke\hss}%
    \else\ifnum#1=6
      \advance\ycoord by -1
      \raise\ycoord \Einheit\hbox to0pt{\hskip\xcoord \Einheit
        \hbox{\hskip-\PfadD@cke\vrule height1 \Einheit width\PfadD@cke depth0pt}\hss}%
    \else\ifnum#1=7
      \advance\xcoord by -1
      \advance\ycoord by -1
      \raise\ycoord \Einheit\hbox to0pt{\hskip\xcoord \Einheit
         \unitlength\Einheit
         \line(1,1){1}\hss}
    \else\ifnum#1=8
      \advance\xcoord by -1
      \advance\ycoord by +1
      \raise\ycoord \Einheit\hbox to0pt{\hskip\xcoord \Einheit
         \unitlength\Einheit
         \line(1,-1){1}\hss}
    \fi\fi\fi\fi
    \fi\fi\fi\fi
  \fi\next}
\def\hSSchritt{\leavevmode\raise-.4pt\hbox to0pt{\hss.\hss}\hskip.2\Einheit
  \raise-.4pt\hbox to0pt{\hss.\hss}\hskip.2\Einheit
  \raise-.4pt\hbox to0pt{\hss.\hss}\hskip.2\Einheit
  \raise-.4pt\hbox to0pt{\hss.\hss}\hskip.2\Einheit
  \raise-.4pt\hbox to0pt{\hss.\hss}\hskip.2\Einheit}
\def\vSSchritt{\vbox{\baselineskip.2\Einheit\lineskiplimit0pt
\hbox{.}\hbox{.}\hbox{.}\hbox{.}\hbox{.}}}
\def\DSSchritt{\leavevmode\raise-.4pt\hbox to0pt{%
  \hbox to0pt{\hss.\hss}\hskip.2\Einheit
  \raise.2\Einheit\hbox to0pt{\hss.\hss}\hskip.2\Einheit
  \raise.4\Einheit\hbox to0pt{\hss.\hss}\hskip.2\Einheit
  \raise.6\Einheit\hbox to0pt{\hss.\hss}\hskip.2\Einheit
  \raise.8\Einheit\hbox to0pt{\hss.\hss}\hss}}
\def\dSSchritt{\leavevmode\raise-.4pt\hbox to0pt{%
  \hbox to0pt{\hss.\hss}\hskip.2\Einheit
  \raise-.2\Einheit\hbox to0pt{\hss.\hss}\hskip.2\Einheit
  \raise-.4\Einheit\hbox to0pt{\hss.\hss}\hskip.2\Einheit
  \raise-.6\Einheit\hbox to0pt{\hss.\hss}\hskip.2\Einheit
  \raise-.8\Einheit\hbox to0pt{\hss.\hss}\hss}}
\def\SPfad(#1,#2),#3\endSPfad{\unskip\leavevmode
  \xcoord#1 \ycoord#2 \ZeichneSPfad#3\endSPfad}
\def\ZeichneSPfad#1{\ifx#1\endSPfad\let\next\relax
  \else\let\next\ZeichneSPfad
    \ifnum#1=1
      \raise\ycoord \Einheit\hbox to0pt{\hskip\xcoord \Einheit
         \hSSchritt\hss}%
      \advance\xcoord by 1
    \else\ifnum#1=2
      \raise\ycoord \Einheit\hbox to0pt{\hskip\xcoord \Einheit
        \hbox{\hskip-2pt \vSSchritt}\hss}%
      \advance\ycoord by 1
    \else\ifnum#1=3
      \raise\ycoord \Einheit\hbox to0pt{\hskip\xcoord \Einheit
         \DSSchritt\hss}
      \advance\xcoord by 1
      \advance\ycoord by 1
    \else\ifnum#1=4
      \raise\ycoord \Einheit\hbox to0pt{\hskip\xcoord \Einheit
         \dSSchritt\hss}
      \advance\xcoord by 1
      \advance\ycoord by -1
    \else\ifnum#1=5
      \advance\xcoord by -1
      \raise\ycoord \Einheit\hbox to0pt{\hskip\xcoord \Einheit
         \hSSchritt\hss}%
    \else\ifnum#1=6
      \advance\ycoord by -1
      \raise\ycoord \Einheit\hbox to0pt{\hskip\xcoord \Einheit
        \hbox{\hskip-2pt \vSSchritt}\hss}%
    \else\ifnum#1=7
      \advance\xcoord by -1
      \advance\ycoord by -1
      \raise\ycoord \Einheit\hbox to0pt{\hskip\xcoord \Einheit
         \DSSchritt\hss}
    \else\ifnum#1=8
      \advance\xcoord by -1
      \advance\ycoord by 1
      \raise\ycoord \Einheit\hbox to0pt{\hskip\xcoord \Einheit
         \dSSchritt\hss}
    \fi\fi\fi\fi
    \fi\fi\fi\fi
  \fi\next}
\def\Koordinatenachsen(#1,#2){\unskip
 \hbox to0pt{\hskip-.5pt\vrule height#2 \Einheit width.5pt depth1 \Einheit}%
 \hbox to0pt{\hskip-1 \Einheit \xcoord#1 \advance\xcoord by1
    \vrule height0.25pt width\xcoord \Einheit depth0.25pt\hss}}
\def\Koordinatenachsen(#1,#2)(#3,#4){\unskip
 \hbox to0pt{\hskip-.5pt \ycoord-#4 \advance\ycoord by1
    \vrule height#2 \Einheit width.5pt depth\ycoord \Einheit}%
 \hbox to0pt{\hskip-1 \Einheit \hskip#3\Einheit 
    \xcoord#1 \advance\xcoord by1 \advance\xcoord by-#3 
    \vrule height0.25pt width\xcoord \Einheit depth0.25pt\hss}}
\def\Gitter(#1,#2){\unskip \xcoord0 \ycoord0 \leavevmode
  \LOOP\ifnum\ycoord<#2
    \loop\ifnum\xcoord<#1
      \raise\ycoord \Einheit\hbox to0pt{\hskip\xcoord \Einheit\Punkt\hss}%
      \advance\xcoord by1
    \repeat
    \xcoord0
    \advance\ycoord by1
  \REPEAT}
\def\Gitter(#1,#2)(#3,#4){\unskip \xcoord#3 \ycoord#4 \leavevmode
  \LOOP\ifnum\ycoord<#2
    \loop\ifnum\xcoord<#1
      \raise\ycoord \Einheit\hbox to0pt{\hskip\xcoord \Einheit\Punkt\hss}%
      \advance\xcoord by1
    \repeat
    \xcoord#3
    \advance\ycoord by1
  \REPEAT}
\def\Label#1#2(#3,#4){\unskip \xdim#3 \Einheit \ydim#4 \Einheit
  \def\lo{\advance\xdim by-.5 \Einheit \advance\ydim by.5 \Einheit}%
  \def\llo{\advance\xdim by-.25cm \advance\ydim by.5 \Einheit}%
  \def\loo{\advance\xdim by-.5 \Einheit \advance\ydim by.25cm}%
  \def\o{\advance\ydim by.25cm}%
  \def\ro{\advance\xdim by.5 \Einheit \advance\ydim by.5 \Einheit}%
  \def\rro{\advance\xdim by.25cm \advance\ydim by.5 \Einheit}%
  \def\roo{\advance\xdim by.5 \Einheit \advance\ydim by.25cm}%
  \def\l{\advance\xdim by-.30cm}%
  \def\r{\advance\xdim by.30cm}%
  \def\lu{\advance\xdim by-.5 \Einheit \advance\ydim by-.6 \Einheit}%
  \def\llu{\advance\xdim by-.25cm \advance\ydim by-.6 \Einheit}%
  \def\luu{\advance\xdim by-.5 \Einheit \advance\ydim by-.30cm}%
  \def\u{\advance\ydim by-.30cm}%
  \def\ru{\advance\xdim by.5 \Einheit \advance\ydim by-.6 \Einheit}%
  \def\rru{\advance\xdim by.25cm \advance\ydim by-.6 \Einheit}%
  \def\ruu{\advance\xdim by.5 \Einheit \advance\ydim by-.30cm}%
  #1\raise\ydim\hbox to0pt{\hskip\xdim
     \vbox to0pt{\vss\hbox to0pt{\hss$#2$\hss}\vss}\hss}%
}
\catcode`\@=12

\numberwithin{equation}{section}

\newtheorem{theorem}{Theorem}
\newtheorem{proposition}[theorem]{Proposition}
\newtheorem{lemma}[theorem]{Lemma}
\newtheorem{corollary}[theorem]{Corollary}

\theoremstyle{remark}
\newtheorem*{remark}{Remark}

\def\fl#1{\left\lfloor#1\right\rfloor}

\def\Hom{\operatorname{Hom}}

\def\al{\alpha}
\def\si{\sigma}
\def\rh{\rho}
\def\ga{\gamma}
\def\Ga{\Gamma}
\def\la{\lambda}
\def\De{\Delta}

\newcommand{\Z}{\mathbb{Z}}

\newcommand{\N}{\mathbb{N}}
\newcommand{\Q}{\mathbb{Q}}

\begin{document}
\title[Parity patterns associated with lifts of Hecke groups]{Parity
patterns associated with lifts of Hecke groups} 
\author[C. Krattenthaler and T.\,W. M\"uller]{Christian Krattenthaler$^\dagger$
and Thomas W. M\"uller}

\address{Fakult\"at f\"ur Mathematik, Universit\"at Wien, Nordbergstra{\ss}e
15, A-1090 Vienna, Austria}

\address{School of Mathematical Sciences, Queen Mary \& Westfield College, 
University of London, Mile End Road, London E1\,4NS, 
United Kingdom}

\thanks{$^\dagger$Research partially supported 
by the Austrian Science Foundation FWF, grant S9607-N13,
in the framework of the National Research Network
``Analytic Combinatorics and Probabilistic Number Theory"}

\date{}

\begin{abstract}
Let $q$ be an odd prime, $m$ a positive integer, and let $\Ga_m(q)$ be
the group generated by two elements $x$ and $y$ subject to the
relations $x^{2m}=y^{qm}=1$ and $x^2=y^q$; that is, $\Ga_m(q)$ is the
free product of two cyclic groups of orders $2m$ respectively $qm$,
amalgamated along their subgroups of order $m$. Our main result
determines the parity behaviour of the generalized subgroup numbers of
$\Ga_m(q)$ which were defined in [T.~W.~M\"uller, {\it Adv.\ in Math.\
{\bf 153}} (2000), 118--154], and which count all the
homomorphisms of index $n$ subgroups of $\Ga_m(q)$ into a given finite
group $H$, in the case when $\gcd(m,\vert H\vert)=1$. This computation
depends upon the solution of three counting problems in the Hecke group
$\mathfrak H(q)=C_2*C_q$: (i)~determination of the parity of the subgroup
numbers of $\mathfrak H(q)$; 
(ii)~determination of the parity of the number of index $n$
subgroups of $\mathfrak H(q)$ which are isomorphic to a free product of
copies of $C_2$ and of $C_\infty$;
(iii)~determination of the parity of 
the number of index $n$
subgroups in $\mathfrak H(q)$ which are isomorphic to a free product of
copies of $C_q$. The first problem has already been
solved in [T.~W.~M\"uller, in: {\it Groups: Topological, Combinatorial and
Arithmetic Aspects}, (T.~W.~M\"uller ed.), LMS Lecture Notes Series
311, Cambridge University Press, Cambridge, 2004, pp.~327--374]. 
The bulk of our paper deals with the solution of
Problems (ii) and (iii). 
\end{abstract}
\maketitle
\section{Introduction}
\label{Sec:Intro}
\subsection{}
The study of congruences for subgroup numbers and related numerical
quantities of groups is almost as old as group theory itself. The
reader might think for instance of Frobenius' refinement in
\cite{Frob1} of one of Sylow's fundamental theorems:     

{\leftskip1cm\rightskip1cm
(I) \textit{the number $N_{p,r}(G)$ of subgroups of order $p^r$ in a
finite group $G,$ whose order is divisible by $p^r,$ satisfies
$N_{p,r}(G)\equiv 1\,\,\mathrm{mod}\,\,p$.}  \par}

Related results concern the number of solutions of equations in finite
groups. The model result here is again due to Frobenius, \cite{Frob2}: 

{\leftskip1cm\rightskip1cm
(II) \textit{the number of solutions of the equation $x^m=1$ in a
finite group $G$ is a multiple of $\gcd(m, \vert G\vert)$.} \par}

The last result was considerably extended and sharpened by P.~Hall
\cite{PHall2}, building heavily on his groundbreaking work
\cite{PHall1} concerning the structure of finite $p$-groups.  

From a somewhat more abstract point of view, Frobenius' theorem (II)
provides congruences for the number of homomorphisms of a finite
cyclic group into a finite group $G$; and it is this point of view,
which explains the connection between results of types (I) and (II).
Certain sequences $\{G_n\}_{n\geq0}$ of finite groups, like the
sequence $G_n=S_n$ of symmetric groups or, more generally, a sequence
of full monomial groups $G_n=H\wr S_n$, have the power of relating the
enumeration of finite index subgroups by index (or similar counting
functions) in an arbitrary finitely generated group $\Gamma$ to the
counting of homomorphisms $\Gamma\rightarrow G_n$. On the level of
generating functions, this relationship is of exponential type; see
Formula (\ref{Eq:MonRepsGenFunc}) for a concrete example in this direction.

To some extent, divisibility properties of subgroup numbers of a
(finitely generated) infinite group may be viewed as a kind of
analogue to these classical results for finite groups. Further
motivation comes from \textit{subgroup growth theory}, which studies
growth, asymptotics, and more delicate number-theoretic properties of
subgroup counting functions; cf.\ the monograph \cite{LubSeg} for an
overview of results up to about 2002. As far as divisibility
properties are concerned, it was a, by now classical, result of
Stothers concerning the modular group $\Gamma=\mathrm{PSL}_2(\Z)$,
which prompted the quest that led to the delicate congruence patterns
exhibited by Hecke groups and other free products. Stothers
\cite{Stothers} showed that

{\leftskip1cm\rightskip1cm
(III) \textit{the number of index $n$ subgroups in the modular group
is odd if, and only if, $n$ is of the form $2$-power minus $3$ or
$2$-power minus $6$.}  \par}

The present paper continues this line of research by investigating the
parity of (generalized) subgroup numbers of certain free products with
amalgamation, which may be viewed as lifts of an underlying Hecke
group. We now turn to a more precise description of the contents of
this paper. 

\subsection{}
For an odd prime $q$, let
\begin{equation} \label{eq:Hecke} 
\mathfrak H(q)=\big\langle x,y\,\big\vert\, x^{2}=y^q=1\big\rangle\cong C_2*C_q
\end{equation}
be the standard Hecke group attached to $q$, and for a positive
integer $m$, let 
\[
\Gamma_m = \Gamma_m(q)=
  \big\langle x,y\,\big\vert\,x^{2m} = y^{qm} = 1,\, x^2=y^q\big\rangle
\cong C_{2m} \underset{C_m}{\ast} C_{qm}
\]
be the associated sequence of lifts in the sense of
\cite[Eq.~(40)]{MuInv}. Moreover, for a finitely generated group
$\Ga$, a finite group $H$, and a positive integer $n$, set 
$$
s_{\Gamma}^H(n):= {\sum_{(\Gamma:\Delta)=n}}\vert\Hom(\Delta,H)\vert.$$
The family $\{s_\Ga^H(n)\}_{n,H}$ is the collection of {\it
generalized subgroup numbers} of $\Ga$. For $H=\{1\}$, the quantity
$s_\Ga^H(n)$ is the number of index $n$ subgroups of
$\Ga$, 
which we denote by $s_\Ga(n)$.

The present paper focusses on the parity behaviour of the
numbers $s_\Ga^H(n)$ in the case when $\Ga=\Ga_m(q)$ for $m,q$ as
above.

The systematic investigation of divisibility properties of subgroup
counting functions begins with \cite{MuParheck} 
(which had circulated in the subgroup growth community for some years), 
where the parity of
$s_q(n):=s_{\mathfrak H(q)}(n)$ and of the number of free subgroups in $\mathfrak
H(q)$ of given finite index is determined. These results have been
generalized to larger classes of groups and primes not necessarily
equal to $2$ in \cite{MuFo1,MuFo2}. For a survey of these and related
developments up to 2001, the reader may consult \cite{MuSurv}.

The present paper is a first attempt to study 
divisibility properties of (generalized) subgroup
numbers for amalgamated products. Our main result (see
Theorem~\ref{thm:main} in Section~\ref{sec:main}) 
determines the parity behaviour of the generalized subgroup numbers 
$s_{\Ga_m(q)}^H(n)$
in the case when $\gcd(m,\vert H\vert)=1$. 
As Proposition~\ref{Prop:Reduction1} in the next section 
shows, this computation can be
reduced to the solution of three counting problems in the Hecke group
$\mathfrak H(q)=C_2*C_q$: 
\begin{enumerate}
\item[(i)]determination of the parity of the subgroup
numbers $s_q(n)$ of $\mathfrak H(q)$; 
\item[(ii)]determination of the parity of the numbers $M_q(n)$,
where, by definition, $M_q(n)$ is the number of index $n$
subgroups of $\mathfrak H(q)$ isomorphic to a free product of
copies of $C_2$ and of $C_\infty$;
\item[(iii)]determination of the parity of 
the numbers $N_q(n)$, where, by definition, $N_q(n)$ 
is the number of index $n$
subgroups of $\mathfrak H(q)$ isomorphic to a free product of
copies of $C_q$.
\end{enumerate}
The first problem has already been
solved in \cite[Cor.~4]{MuParheck}. The corresponding result is
recalled here as Theorem~\ref{Thm:sqn} in the next section. 
Problem~(ii) is solved in
Theorem~\ref{Thm:NqnParity} in Section~\ref{Sec:IsoType}, which, 
perhaps somewhat surprisingly, 
shows that the numbers $M_q(n)$ are always even, and Problem~(iii) is
solved in Theorem~\ref{Thm:SolvSmod2} in Section~\ref{Sec:GenParHecke}.

The purpose of Section~\ref{Sec:CosetDiagrams} is to develop in some detail
the language of coset diagrams for the groups $\mathfrak H(q)$, in order to
prepare for the proof of Theorem~\ref{Thm:NqnParity}.
This proof needs several partial results, which are the subject of
Section~\ref{Sec:IsoType}.
The first observation is that $M_q(n)=0$ if $q$ does not divide $n$,
see \eqref{eq:qnmidn}. 
It therefore suffices to investigate the parity of the
numbers $M_q(qk)$, where $k$ is some positive integer.
The enumerative relation between the coset diagrams and subgroup
numbers is made precise in Subsection~\ref{sec:4.1}. This is done in
greater generality than actually needed in our paper, in order to
record the corresponding facts for possible use elsewhere.
Lemma~\ref{Lem:Zus} expresses the numbers of index $n$ subgroups of
$\mathfrak H(q)$ which are isomorphic to 
$C_2^{\ast\lambda} \ast C_q^{\ast\mu} \ast F_{\nu}$ for fixed
$\lambda,\mu,\nu$ in terms of the number of certain mixed graphs. The
enumeration of these graphs is subsequently reduced in
Lemma~\ref{Lem:ZusN} to graphs where the number of vertices is
divisible by $q$. Then Corollary~\ref{Cor:SubgpMu=0Enum}
summarizes specifically the implications for the subgroup numbers
$M_q(qk)$ and $s_q(n)$. The main result of Subsection~\ref{sec:4.2}
is an explicit formula for the number of (mixed) graphs that appear
in the corollary, see Proposition~\ref{Prop:NqComp}. This formula
allows us to undertake a $2$-adic study of the numbers $M_q(qk)$ for
$q\ge5$ in Subsection~\ref{Subsec:NqnParity}, 
by using known results on the $2$-adic valuation of the
number of involutions in symmetric groups. 
The corresponding result, which we prove
in Proposition~\ref{Prop:Nq2Part}, is stronger than actually needed,
as it exhibits a rapidly growing lower bound for the $2$-adic valuation 
of the numbers $M_q(qk)$.
On the other hand, a similar result cannot be expected for $q=3$.
Moreover, our explicit formula from Subsection~\ref{sec:4.2} does
not seem to be suited for the $2$-adic analysis in this case.
Instead, we use results of Stothers \cite{Stothers} to prove in
Proposition~\ref{Prop:N3Parity}, also in Subsection~\ref{Subsec:NqnParity}, 
that $M_3(3k)$ is even as well.

Section~\ref{Sec:GenParHecke} puts together the ingredients for the 
proof of Theorem~\ref{Thm:SolvSmod2}. The starting point is the
observation that the parity of $N_q(n)$ is the same as the parity
of the generalized subgroup numbers $s_q^H(n):=s_{\mathfrak
H(q)}^H(n)$ with $\vert H\vert$ even, see \eqref{eq:s=N}.
Just as subgroup numbers are connected with the enumeration of
permutation representations, by \cite[Cor.~1]{MRepres}
the generalized subgroup numbers
$s_\Gamma^H(n)$ of a finitely generated group $\Gamma$ are related to
the function $\vert\Hom(\Gamma,H\wr S_n)\vert$, counting monomial
representations of $\Gamma$, via an identity of exponential type,
which is restated here in
\eqref{Eq:MonRepsGenFunc}. This identity is exploited 
in Subsections~\ref{sec:5.1}--\ref{sec:5.3} to
compute the generalized subgroup numbers $s_q^H(n)$ via a generating
function approach, and to determine their parity
in the case when $H$ has even order in Subsection~\ref{sec:5.5}, 
thereby also determining the parity of $N_q(n)$. 
As a side result, we actually obtain a formula leading
to an integral recurrence relation for the numbers $s_q^H(n)$,
$n=0,1,\dots$, for each fixed $q$; see \eqref{Eq:PrincEq}. This aspect
is illustrated in Subsection~\ref{sec:5.4}
with the modular group $\mathfrak H(3)$; 
see Proposition~\ref{Prop:ModularHRec}.

%Christian:
Since the proofs of some auxiliary results in Sections~\ref{Sec:IsoType}
and \ref{Sec:GenParHecke} are tedious, involving however relatively
straightforward calculations, we have put these proofs in an appendix,
so that they do not distract from or obscure the main line of arguments.

The final section, Section~\ref{sec:main}, contains
our main result, Theorem~\ref{thm:main}, determining the parity 
behaviour of the generalized subgroup numbers 
$s_{\Ga_m(q)}^H(n)$ in the case when\break 
$\gcd(m,\vert H\vert)=1$,
as well as its specialization to the case when $q$ is a Fermat prime,
in which case one can be much more precise (see
Corollary~\ref{cor:main}), together with further simplifications 
when $m$ and/or $n$ satisfy some more specific properties.

\section{A reduction result}
\label{Sec:Reduction}
%\subsection{The groups $\Gamma_m(q)$}
%\label{Subsec:Reduction1}

For an odd prime $q$, let $\Gamma_m = \Gamma_m(q)$ 
be as in the introduction. We note that \linebreak $\Ga_m(q)\not\cong
\Ga_{m'}(q')$ for $(m,q)\ne(m',q')$. 
In this section we will show that the parity behaviour of the
generalized subgroup numbers of these groups can be computed in terms
of certain subgroup numbers pertaining to the base group $\mathfrak
H(q)$; see Proposition~\ref{Prop:Reduction1}.

Set
\[
\zeta:=x^2=y^q,
\] 
so that $\zeta$ is a primitive element for the centre of
$\Gamma_m$. For a positive integer $n$, a divisor $d$ of $m$, and a
finite group $H$, define 
\[
s_{\Gamma_m}^H(d,n):=\underset{\Delta\cap\langle\zeta\rangle =
\langle\zeta^{m/d}\rangle}{ \underset{(\Gamma_m:\Delta)=n}
{\sum_{\Delta}}}\vert\Hom(\Delta,H)\vert,
\] 
so that
\begin{equation}
\label{Eq:GSNDecomp}
s_{\Gamma_m}^H(n) = \sum_{d\mid m}s_{\Gamma_m}^H(d,n).
\end{equation}
In what follows, we shall find it necessary to suppose that $\gcd(m,\vert
H\vert)=1$. Under the natural projection 
\[
\pi_d:
\Gamma_m\longrightarrow\Gamma_m\big/\big\langle\zeta^{m/d}\big\rangle
\cong \Gamma_{m/d} 
\]
the subgroups $\Delta$ of index $n$ in $\Gamma_m$ having the property
that $\Delta\supseteq\big\langle\zeta^{m/d}\big\rangle$ correspond
bijectively to the totality of index $n$ subgroups in
$\Gamma_{m/d}$. Moreover, since $\gcd(m,\vert H\vert)=1$, every
homomorphism $\varphi: \Delta\rightarrow H$ factors through the
canonical projection $\pi_d^\Delta: \Delta
\rightarrow\Delta/\langle\zeta^{m/d}\rangle$. Hence, for $d\mid m$, 
\begin{align*}
s_{\Gamma_{m/d}}^H(n) &=
\underset{(\Gamma_{m/d}:\bar{\Delta})=n} {\sum_{\bar{\Delta}}}
\vert\Hom(\bar{\Delta},H)\vert \\[2mm] 
&= \underset{\Delta\supseteq\langle\zeta^{m/d}\rangle}
{\underset {(\Gamma_m:\Delta)=n}
{\sum_{\Delta}}}\vert\Hom(\Delta,H)\vert \\[2mm]
&= \underset{d\mid d'\mid
m}{\sum_{d'}}\,\,\underset{\Delta\cap\langle\zeta\rangle=
\langle\zeta^{m/d'}\rangle}{\underset {(\Gamma_m:\Delta)=n}
{\sum_{\Delta}}}
\vert\Hom(\Delta,H)\vert, 
\end{align*}
and thus
\begin{equation}
\label{Eq:Moebius1}
s_{\Gamma_{m/d}}^H(n) = \underset{d\mid d'\mid m}{\sum_{d'}}
s_{\Gamma_m}^H(d',n),\quad d\mid m. 
\end{equation}
Fixing $q,m,n,$ and $H$, and setting
\[
g(x):= \left.\begin{cases}
           s_{\Gamma_{m/x}}^H(n),&x\mid m\\[2mm]
           0,&\mbox{otherwise}
           \end{cases}\right\}\quad(x\in\N),
\]
\[
f(x):= \left.\begin{cases}
           s_{\Gamma_m}^H(x,n),& x\mid m\\[2mm]
           0,&\mbox{otherwise}
           \end{cases}\right\}\quad(x\in\N),
\]
and letting $P$ be the poset with ground-set $\N$ and partial order
defined by $x\le y$ if, and only if, $y\mid x$,
Equation (\ref{Eq:Moebius1}) translates into
\[
g(x) = \sum_{y\leq x} f(y),\quad x\in P,
\]
and, by definition of the function $f$, we have $f(x)=0$ unless $x\geq
m$. By M\"obius inversion (cf.\ \cite[Propositions~2--3]{Rota}), 
we find that 
\begin{equation*}
%\label{Eq:Moebius2}
s_{\Gamma_m}^H(d,n) = \underset{d\mid d'\mid
m}{\sum_{d'}}\mu_P(d,d')\,s_{\Gamma_{m/d'}}^H(n),\quad d\mid m, 
\end{equation*}
where $\mu_P(x,y)$ is the M\"obius function in the poset $P$.
It is well-known that $\mu_P(x,y)=\mu(x/y)$ for $x\le y$ (that is,
$y\mid x$), where $\mu(\cdot)$ is the classical M\"obius function of
number theory. Hence, we have
\begin{equation}
\label{Eq:Moebius2}
s_{\Gamma_m}^H(d,n) = \underset{d\mid d'\mid
m}{\sum_{d'}}\mu(d'/d)\,s_{\Gamma_{m/d'}}^H(n),\quad d\mid m.
\end{equation}
On the other hand, if we set $d=1$ in \eqref{Eq:Moebius2} and 
apply M\"obius
inversion in the other direction, we obtain
\begin{equation}
\label{Eq:Reduction1}
s_{\Gamma_m}^H(n) = \sum_{d\mid m}\,s_{\Gamma_d}^H(1,n),\quad\gcd(m,\vert
H\vert)=1. 
\end{equation}

Our next task is to compute $s_{\Gamma_m}^H(1,n)$ modulo $2$. If there
exists a subgroup $\Delta\leq \Gamma_m$ with $(\Gamma_m:\Delta)=n$ and
$\Delta\cap\langle\zeta\rangle = 1$, then 
\[
\widetilde{\Delta}:= \Delta\cdot\langle\zeta\rangle =
\Delta\times\langle\zeta\rangle; 
\]
hence, $m$ must divide $n$. Thus,
\begin{equation}
\label{Eq:ZeroCrit}
s_{\Gamma_m}^H(1,n)=0,\quad m\nmid n.
\end{equation}
Suppose now that $m\mid n$. Given a subgroup $\Delta\leq\Gamma_m$ with
$(\Gamma_m:\Delta)=n$ and $\Delta\cap\langle\zeta\rangle = 1$,
consider the diagram 
\begin{equation}
\label{Eq:LiftDiagram}
\begin{CD}
\Delta @>{m}>> \widetilde{\Delta} = \Delta\cdot\langle\zeta\rangle
@>{n/m}>> \Gamma_m\\ 
@V{\pi_m}VV  @V{\pi_m}VV    @VV{\pi_m}V\\
\bar{\Delta} @>{\text{id}}>> \bar{\Delta} @>>{n/m}>
\hbox to 8pt{$\bar{\Gamma}\cong\Gamma_1=\mathfrak H(q)$\hss}
\end{CD}
\end{equation}
(a number next to an arrow indicates the index of the corresponding
embedding). Reading Diagram (\ref{Eq:LiftDiagram}) from bottom to top,
we can describe $\Delta$ as a complement to $\langle\zeta\rangle$ in
the lift $\widetilde{\Delta}=\pi_m^{-1}(\bar{\Delta})$ of a subgroup
$\bar{\Delta}$ in $\bar{\Gamma}\cong\mathfrak{H}(q)$ with
$(\bar{\Gamma}:\bar{\Delta})=n/m$. Hence, 
\begin{equation}
\label{Eq:Comp1}
s_{\Gamma_m}^H(1,n) =
\underset{(\bar{\Gamma}:\bar{\Delta})=n/m}
{\sum_{\bar{\Delta}}}\,\vert\Hom(\bar{\Delta},H)\vert\cdot
\vert\mathfrak{C}(\pi_m^{-1}(\bar{\Delta});\langle\zeta\rangle)\vert, 
\quad m\mid n,
\end{equation}
where
$\mathfrak{C}(\pi_m^{-1}(\bar{\Delta});\langle\zeta\rangle)$ is the
(possibly empty) set of complements of $\langle\zeta\rangle$ in
$\pi_m^{-1}(\bar{\Delta})$.

Given $\bar{\Delta}\leq\bar{\Gamma}$ of index $n/m$, when does
$\langle\zeta\rangle$ split in $\pi_m^{-1}(\bar{\Delta})$? To answer
this question, we make use of the long exact (Mayer--Vietoris)
sequence 
\begin{multline}
\label{Eq:MVSequence}
\cdots\rightarrow H^k(G;A)\underset{(res,res)}{\longrightarrow}
H^k(G_1;A)\oplus H^k(G_2;A)\\
\underset{(res,-res)}{\longrightarrow}
H^k(S;A)\underset{\delta}{\longrightarrow}
H^{k+1}(G;A)\rightarrow\cdots 
\end{multline} 
associated with an amalgamated product $G=G_1 \underset{S}{\ast} G_2$
and a left $RG$-module $A$; cf. \cite[Theorem~2.10]{Bieri} or
\cite[Theorem~2]{Chiswell}. By the Kurosh subgroup theorem,
$\bar{\Delta}$ is of the form 
\[
\bar{\Delta} \cong C_2^{\ast\lambda(\bar{\Delta})} \ast
C_q^{\ast\mu(\bar{\Delta})} \ast F_{\nu(\bar{\Delta})} 
\]
with non-negative integers
$\lambda(\bar{\Delta}),\mu(\bar{\Delta}),\nu(\bar{\Delta})$ (these
cardinal numbers are seen to be finite through a comparison of Euler
characteristics; cf. Equation~(\ref{Eq:Type/IndexRel}) below). 
Applying (\ref{Eq:MVSequence}) together with the fact that every
extension by a free group splits, we see that 
\begin{equation}
\label{Eq:H2Rel1}
H^2(\bar{\Delta};C_m) \cong \lambda(\bar{\Delta}) H^2(C_2;C_m) \oplus
\mu(\bar{\Delta}) H^2(C_q;C_m). 
\end{equation}
We now distinguish four cases.

(i) \textit{$2\nmid m$ and $q\mid m$.} Suppose that, for some subgroup
$\bar{\Delta}\leq\bar{\Gamma}$ of index $(\bar{\Gamma}:\bar{\Delta}) =
n/m$, $\langle\zeta\rangle$ splits in $\pi_m^{-1}(\bar{\Delta})$, and
that $\mu(\bar{\Delta})>0$. Then $C_m\times C_q$, an Abelian group of
rank $2$, would embed into $\Gamma_m$, which is impossible. Hence, if
$\langle\zeta\rangle$ splits in $\pi_m^{-1}(\bar{\Delta})$, then
$\mu(\bar{\Delta})=0$. Conversely, if $\mu(\bar{\Delta})=0$, then
every extension of $C_m$ by $\bar{\Delta}$ splits, since
$H^2(C_2;C_m)=0$ by the Schur--Zassenhaus theorem plus the fact that
$2\nmid m$; cf. for instance \cite[Chapter~3]{Wehr}. To summarize, we
have shown that 
\begin{equation}
\label{Eq:Splitting1}
\mbox{\textit{if $2\nmid m$ and $q\mid m,$ then $\langle\zeta\rangle$
splits in $\pi_m^{-1}(\bar{\Delta})$ if, and only if,
$\mu(\bar{\Delta})=0$.}} 
\end{equation}
(ii) \textit{$2\nmid m$ and $q\nmid m$.} In this case, again by the
Schur--Zassenhaus theorem and the isomorphism (\ref{Eq:H2Rel1}), we
find that $H^2(\bar{\Delta};C_m)=0$. Hence, 
\begin{multline}
\label{Eq:Splitting2}
\mbox{\textit{if $2\nmid m$ and $q\nmid m,$ then the lift
$\pi_m^{-1}(\bar{\Delta})$}}\\
\mbox{\textit {of every subgroup $\bar{\Delta}\leq_{n/m}
\bar{\Gamma}$ splits the centre $\langle\zeta\rangle$.}} 
\end{multline}
(iii) \textit{$2\mid m$ and $q\nmid m$.} Arguing as in (i), we find that
\begin{equation}
\label{Eq:Splitting3}
\mbox{\textit{if $2\mid m$ and $q\nmid m,$ then $\langle\zeta\rangle$
splits in $\pi_m^{-1}(\bar{\Delta})$ if, and only if,
$\lambda(\bar{\Delta})=0$.}} 
\end{equation}
(iv) $2q\mid m$. In this case, we find by similar arguments that
\begin{equation}
\label{Eq:Splitting4}
\mbox{\textit{if $2q\mid m,$ then $\langle\zeta\rangle$ splits in
$\pi_m^{-1}(\bar{\Delta})$ if, and only if, $\lambda(\bar{\Delta}) =
\mu(\bar{\Delta})=0$.}} 
\end{equation}
Given the information in (\ref{Eq:Splitting1})--(\ref{Eq:Splitting4}),
we can now evaluate the quantities $s_{\Gamma_m}^H(1,n)$ modulo $2$ in
terms of data associated with the base group $\bar{\Gamma}$
alone. Denote by $M_q(n)$ the number of index $n$ subgroups
$\bar{\Delta}$ in $\bar{\Gamma}=\mathfrak{H}(q)$ with the property that 
$\mu(\bar{\Delta})=0$, by 
$N_q(n)$ the number of subgroups of index $n$ in
$\bar{\Gamma}$ which are isomorphic to a free power of $C_q$, 
and recall that $s_q(n)$ denotes the number of index $n$ subgroups 
in $\mathfrak{H}(q)$. 
\begin{proposition}
\label{Prop:Reduction1}
For $m,n\geq1,$ an odd prime $q,$ and a finite group $H,$ we have 
\begin{equation*}
s_{\Gamma_m(q)}^H(1,n) \equiv \left\{\begin{matrix}
                             N_q(n/m);\hfill& 2\mid m\mid n     \mbox{ and
}q\nmid m \hfill\\[2mm]
                             s_q(n/m);\hfill& 2\nmid m, q\nmid m, m\mid
n,\mbox{ and }2\nmid\vert H\vert\hfill\\[2mm]
                             N_q(n/m);\hfill& 2\nmid m, q\nmid m, m\mid
n,\text{ and } 2\mid\vert H\vert\hfill\\[2mm]
                             M_q(n/m);\hfill& q\mid m\mid n, 2\nmid m,\mbox{
and }2\nmid\vert H\vert\hfill\\[2mm]
                             0;\hfill&\mbox{otherwise}\hfill
                             \end{matrix}\right\}\,\,\,\bmod{2}.
\end{equation*}
\end{proposition}
\begin{proof}
This is straightforward. For instance, in the first case ($2\mid m\mid
n\mbox{ and } q\nmid m$), the computation runs as follows: 
\begin{align*}
s_{\Gamma_m}^H(1,n) &=
\underset{(\bar{\Gamma}:\bar{\Delta})=n/m}{\sum_{\bar{\Delta}}}\,
\vert\Hom(\bar{\Delta},H)\vert\cdot
\vert\mathfrak{C}(\pi_m^{-1}(\bar{\Delta});
\langle\zeta\rangle)\vert\\[2mm] 
&=
\underset{\underset{\lambda(\bar{\Delta})=0}
{(\bar{\Gamma}:\bar{\Delta})=n/m}}{\sum_{\bar{\Delta}}}\,
\vert\Hom(\bar{\Delta},H)\vert\cdot\vert\Hom(\bar{\Delta},C_m)\vert\\[2mm]
&=
\underset{\underset{\lambda(\bar{\Delta})=0}
{(\bar{\Gamma}:\bar{\Delta})=n/m}}{\sum_{\bar{\Delta}}}\,
m^{\nu(\bar{\Delta})}\,\vert\Hom(\bar{\Delta},H)\vert\\[2mm]
&\equiv \underset{\underset{\lambda(\bar{\Delta})=\nu(\bar{\Delta})=0}
{(\bar{\Gamma}:\bar{\Delta})=n/m}}{\sum_{\bar{\Delta}}}\,\vert\Hom(\bar{\Delta},H)\vert\\[2mm]
&\equiv \underset{\underset{\lambda(\bar{\Delta})=\nu(\bar{\Delta})=0}
{(\bar{\Gamma}:\bar{\Delta})=n/m}}{\sum_{\bar{\Delta}}}\,1\\[2mm]
&= N_q(n/m),
\end{align*}
where $\bar{\Gamma}=\mathfrak{H}(q)$, and where we have made use of
(\ref{Eq:Comp1}), (\ref{Eq:Splitting3}), and the fact that 
\[
\vert\Hom(C_q,H)\vert = 1 \,+\, (q-1)\cdot\big\vert\big\{U\leq H:\,
U\cong C_q\big\}\big\vert 
\]
is always odd. Computations for the other cases are similar, and are
left to the reader. 
\end{proof}
Combining Equation (\ref{Eq:Reduction1}) with
Proposition~\ref{Prop:Reduction1}, the calculation of the generalized
subgroup numbers $s_{\Gamma_m(q)}^H(n)$ modulo $2$ has been reduced to
the mod $2$ calculation of the numbers $s_q(n)$, $M_q(n)$, and
$N_q(n)$. The first problem has already been solved in
\cite[Cor.~4]{MuParheck}. The result reads as follows.

\begin{theorem} \label{Thm:sqn}
Let $q$ be an odd prime. Then
\begin{gather} \label{eq:sqn}
s_q(n)\equiv1\ (\text{\em mod }2)\quad \text{if, and only if,}\quad
n=1+2(q-1)\eta\text{ or }n=2+4(q-1)\eta,
\text{ where }
\notag
\\
  \mathfrak{s}_2((q-1)\eta+1) = \mathfrak{s}_2(\eta)
+ \mathfrak{s}_2((q-2)\eta+1).
\end{gather}
Here, $\mathfrak{s}_2(x)$ denotes the sum
of digits in the binary expansion of the positive integer $x$.
In particular, if $q$ is a Fermat prime, then 
\begin{multline} \label{eq:sqnFermat}
s_q(n)\equiv1\ (\text{\em mod }2)\quad \text{if, and only if,}\\
n=\frac {2(q-1)^\si-q} {q-2}\text{ or }
n=\frac {4(q-1)^\si-2q} {q-2} 
\text{ for some $\si\ge1$.}
\end{multline}

\end{theorem}

\section{Coset diagrams associated with Hecke groups}
\label{Sec:CosetDiagrams}

There is a well-known connection between the subgroups of index $n$ in
a group $\Gamma$ and the transitive permutation representations of 
$\Gamma$ of degree $n$, which is sometimes useful to obtain
information concerning various subgroup numbers of $\Gamma$. As the
details of this relationship are somewhat hard to find in the
literature, we provide a brief discussion of the general facts 
in Subsection~\ref{sec:3.1} for the 
convenience of the reader, after which we focus on the special case of 
Hecke groups and their diagrams 
in Subsection~\ref{sec:3.2}. 
This discussion will provide the
basis for the mod~$2$ determination of the numbers $M_q(n)$ in 
Section~\ref{Sec:IsoType}.

\subsection{Transitive permutation representations versus
finite-index subgroups} \label{sec:3.1}
Here, and in the rest of the paper,
for a non-negative integer $n$, the symbol
$[n]$ will always denote the standard set $\{1,2,\ldots,n\}$ of size $n$.
Let $S_n=\mathrm{Sym}([n])$ be the symmetric group of degree $n$ acting 
on the standard set $[n]$,
endowed with geometric multiplication, let
$U_n=\mathrm{stab}_{S_n}(1)$ be the stabilizer of the letter $1$ in
$S_n$, and, for a group $\Gamma$,  denote by $T_n(\Gamma)$ the set of
all transitive permutation representations $\varphi: \Gamma\rightarrow
S_n$. For a permutation $\sigma\in S_n$, let  $\iota_\sigma$ be the
inner automorphism of $S_n$ induced by $\sigma$, 
\[
\iota_\sigma (\pi) = \sigma\circ \pi\circ \sigma^{-1},\quad \pi\in S_n, 
\]
and define a left action of $U_n$ on $T_n(\Gamma)$ via the commutative 
diagram
\[
\begin{CD}
\Gamma @>\varphi>> S_n\\
@V{\mathrm{id}_\Gamma}VV  @VV{\iota_\sigma}V\\
\Gamma @>>{\sigma\cdot\varphi}> S_n
\end{CD}\qquad(\sigma\in U_n,\, \varphi\in T_n(\Gamma)),
\]
where $\mathrm{id}_\Gamma$ is the identity map on $\Gamma$. The basic
observation is now the following. 
\begin{proposition}
\label{Prop:TransReps}
{\em(i)} The action of $U_n$ on $T_n(\Gamma)$ is free; in particular,
each equivalence class of $T_n(\Gamma)$ modulo $U_n$ contains
precisely $(n-1)!$ elements.

{\em(ii)} There is a one-to-one correspondence between the subgroups
of index $n$ in $\Gamma$ and the equivalence classes of $T_n(\Gamma)$
under the action of $U_n$. 
\end{proposition}
\begin{proof}
(i) Suppose that $\sigma\cdot\varphi=\varphi$ for some $\sigma\in U_n$
and some $\varphi\in T_n(\Gamma)$. By transitivity, given $j\in[n]$,
there exists $\gamma_j\in\Gamma$ such that $\varphi(\gamma_j)(1)=j$,
and hence 
\begin{align*}
\sigma(j) &= \sigma(\varphi(\gamma_j)(\sigma^{-1}(1)))\\
&= (\sigma\cdot\varphi)(\gamma_j)(1)\\
&= \varphi(\gamma_j)(1)\\
&= j.
\end{align*}
Since $j$ was arbitrary, we conclude that $\sigma=\mathrm{id}_{[n]}$; thus, the
action of $U_n$ on $T_n(\Gamma)$ is free as claimed. The particular
statement follows from this and the fact that $U_n\cong S_{n-1}$.

(ii) Given $\varphi\in T_n(\Gamma)$, the subgroup
\[
\Delta_\varphi := \mathrm{stab}_\varphi(1) = \varphi^{-1}(U_n)
\]
has index $n$ in $\Gamma$, since a set of elements
$\{\gamma_1,\gamma_2,\ldots,\gamma_n\}$ chosen as in part~(i) forms a
system of (left) coset representatives for $\Gamma$ modulo
$\Delta_\varphi$. Further, if $\varphi_2=\sigma\cdot\varphi_1$ for
some $\sigma\in U_n$, then 
\begin{align*}
\Delta_{\varphi_2} &= \varphi_2^{-1}(U_n)\\
&= (\iota_\sigma\circ\varphi_1)^{-1}(U_n)\\
&= \varphi_1^{-1}(\iota_{\sigma^{-1}}(U_n))\\
&= \varphi_1^{-1}(U_n)\\
&= \Delta_{\varphi_1}.
\end{align*}
Hence, the assignment $\varphi\mapsto\Delta_\varphi$ induces a 
well-defined map
\[
\Phi: {U_n}\backslash{T_n(\Gamma)} \longrightarrow\big\{\Delta:\, 
\Delta\leq\Gamma,\, (\Gamma:\Delta)=n\big\}.
\]
A subgroup $\Delta\leq\Gamma$ of index $n$ in $\Gamma$ induces a
transitive $\Gamma$-action by left multiplication on the $n$-set
$\Gamma/\Delta$ of left cosets which, after suitable renaming, becomes
a transitive $\Gamma$-action on $[n]$ with the property that 
$\mathrm{stab}(1)=\Delta$. This shows that $\Phi$ is surjective.

To prove injectivity, suppose that $\varphi_1,\varphi_2\in
T_n(\Gamma)$ are two transitive permutation representations of
$\Gamma$ on $[n]$, such that $\Delta_{\varphi_1}=\Delta_{\varphi_2}$.
As before, choose elements $\gamma_1,\gamma_2,\ldots,\gamma_n\in
\Gamma$ such that 
\[
\varphi_1(\gamma_j)(1)=j,\quad 1\leq j\leq n,
\] 
and define a permutation $\sigma_0\in U_n$ via
\[
\sigma_0(j):= \varphi_2(\gamma_j)(1),\quad 1\leq j\leq n.
\]
Then, for $\gamma\in\Gamma$ and $j\in[n]$, and with $\gamma\gamma_j
\underset{\Delta_{\varphi_1}}{\sim} \gamma_k$, we have 
\begin{align*}
\sigma_0(\varphi_1(\gamma)(j)) &=
\sigma_0(\varphi_1(\gamma)(\varphi_1(\gamma_j)(1)))\\
&= \sigma_0(\varphi_1(\gamma\gamma_j)(1))\\
&= \sigma_0(\varphi_1(\gamma_k)(1))\\
&= \sigma_0(k)\\
&= \varphi_2(\gamma_k)(1)\\
&= \varphi_2(\gamma\gamma_j)(1)\\
&= \varphi_2(\gamma)(\varphi_2(\gamma_j)(1))\\
&= \varphi_2(\gamma)(\sigma_0(j)).
\end{align*}
Since $j$ and $\gamma$ were arbitrary, this shows that
$\sigma_0\cdot\varphi_1=\varphi_2$; that is, $\varphi_1$ and
$\varphi_2$ are equivalent under the action of $U_n$, as required. 
\end{proof}
\subsection{The case of Hecke groups}\label{sec:3.2}
Now let $q$ be an odd prime number, and let
\begin{equation}
\label{Eq:HeckeGenerators}
\mathfrak{H}(q) = \big\langle x,y\,\big\vert\, x^2=y^q=1\big\rangle
\end{equation}
be the standard Hecke group attached to $q$. By Kurosh's subgroup theorem, a
subgroup $\Delta\leq\mathfrak{H}(q)$ is of the form 
\begin{equation}
\label{Eq:IsoType}
\Delta\,\cong\,
C_2^{\ast\lambda(\Delta)}\,\ast\,C_q^{\ast\mu(\Delta)}\,\ast\,F_{\nu(\Delta)} 
\end{equation}
with cardinal numbers $\lambda(\Delta), \mu(\Delta)$, and $\nu(\Delta)$.
Moreover, if $\Delta$ has finite index in $\mathfrak{H}(q)$, comparing
the Euler characteristic of $\Delta$ with that of $\mathfrak{H}(q)$
shows that $\lambda(\Delta), 
\mu(\Delta),\nu(\Delta)$ are finite, and are connected to the index
$(\mathfrak{H}(q):\Delta)$ via the relation 
\begin{equation}
\label{Eq:Type/IndexRel}
q\lambda(\Delta) + 2(q-1)\mu(\Delta) + 2q(\nu(\Delta)-1) =
(q-2)(\mathfrak{H}(q):\Delta). 
\end{equation} 
Our next result which, in the case of the modular group, goes back to
Millington \cite[Theorem~1]{Millington}, provides a refinement of the
bijection $\Phi$ in the proof of Proposition~\ref{Prop:TransReps}(ii)
by taking into account the isomorphism type 
\[
{\bf t}(\Delta) = (\lambda(\Delta), \mu(\Delta),\nu(\Delta))
\]
of a finite-index subgroup $\Delta$ in $\mathfrak{H}(q)$.
\begin{proposition}
\label{Prop:TransRepsRefined}
Let $q$ be an odd prime, let $\Delta$ be a subgroup
of finite index $n$ in $\mathfrak{H}(q),$ and let
$\varphi:\mathfrak{H}(q)\rightarrow S_n$ be a transitive permutation
representation of $\mathfrak{H}(q)$ such that $\Delta_\varphi=\Delta$.
Then $\varphi(x)$ has precisely $\lambda(\Delta)$ fixed points, and 
$\varphi(y)$ has exactly $\mu(\Delta)$ fixed points, where $x,y$ are
as in {\em (\ref{Eq:HeckeGenerators})}. 
\end{proposition}
\begin{proof}
As in the proof of Proposition~\ref{Prop:TransReps}, choose elements
$\gamma_1,\gamma_2,\ldots,\gamma_n\in\mathfrak{H}(q)$ such that 
\[
\varphi(\gamma_j)(1) = j,\quad 1\leq j\leq n;
\]
i.e., $\{\gamma_1,\gamma_2,\ldots,\gamma_n\}$ is a left transversal
for $\mathfrak{H}(q)$ modulo $\Delta$. The equation $\varphi(x)(j)=j$
is equivalent to the condition that $x^{\gamma_j}\in\Delta$; in
particular, the assignment 
\[
j\mapsto\mathcal{C}_j:=\big\{\delta^{-1} x^{\gamma_j} \delta:\,
\delta\in\Delta\big\} 
\]
defines a mapping $\Phi_x$ from the set of fixed points of
$\varphi(x)$ to the set of conjugacy classes of elements of order $2$
in $\Delta$. If 
\[
\mathcal{C} = \big\{\delta^{-1} x' \delta:\,
\delta\in\Delta\big\}\subseteq\Delta 
\]
is such a conjugacy class, then $x'$ is in $\Delta$ and has order $2$;
thus, by the 
torsion theorem for free products 
(cf., for instance,
Theorem~1.6 in Chapter~IV of \cite{LS}), there exists an element
$\gamma=\gamma_j\delta$ such that 
\[
x' = \gamma^{-1} x \gamma = \delta^{-1} x^{\gamma_j} \delta.
\]
It follows that $x^{\gamma_j}\in\Delta$, i.e., $\varphi(x)(j)=j$, and
$\mathcal{C}=\mathcal{C}_j$; that is, $\Phi_x$ is surjective. To prove
injectivity, let $j$ and $k$ be fixed points of $\varphi(x)$, and
suppose that $\mathcal{C}_j=\mathcal{C}_k$; that is, 
\[
\delta^{-1} x^{\gamma_j} \delta = x^{\gamma_k}.
\]
Consequently, the element $\gamma_j\delta\gamma_k^{-1}$ centralizes
the generator $x$, implying $\gamma_j\delta\gamma_k^{-1}=x^\epsilon$
with $\epsilon\in\{0,1\}$ by \cite[Cor.~4.1.6]{MKS}. Hence, 
\[
j = \varphi(\gamma_j\delta\gamma_k^{-1})(k) = \varphi(x^\epsilon)(k) = k,
\]
as desired. We conclude that the fixed points of $\varphi(x)$ are in
one-to-one correspondence with the $\lambda(\Delta)$ conjugacy classes of
elements of order $2$ in $\Delta$. By a similar argument, the fixed
points of $\varphi(y)$ are in one-to-one correspondence with the
$\mu(\Delta)$ conjugacy classes of cyclic subgroups of order $q$ in
$\Delta$, completing the proof. 
\end{proof}
It is customary to translate these facts into a geometric language. To
every transitive permutation representation
$\varphi:\mathfrak{H}(q)\rightarrow S_n$ there corresponds a diagram
$D_\varphi$ consisting of $n$ labelled vertices, red undirected loops,
blue undirected loops, red undirected edges, and blue directed edges,
constructed as follows: the vertices of $D_\varphi$ are labelled with
the elements of the standard set $[n]$; for $i,j\in[n]$ such that
$\varphi(x)(i)=j$, the vertices labelled $i$ and $j$ are joined by an
undirected red edge (a loop if $i=j$); for $i,j\in[n]$ with $i\neq j$
and $\varphi(y)(i)=j$, we draw a directed blue edge from vertex $i$ to
vertex $j$, while, for $i=j$, we attach an undirected blue loop to
vertex $i$. In this way, the set $T_n(\mathfrak{H}(q))$ is in
bijective correspondence with the set $\mathcal{D}_{n,q}$ of diagrams
on $n$ vertices labelled with the elements of the standard set $[n]$,
such that 
\begin{enumerate}
\item[(a)] each vertex has a red loop, or is incident with exactly one 
red edge,
\item[(b)] each vertex has a blue loop, or is contained in precisely
one oriented blue $q$-gon,
\item[(c)] the red and blue edges together give a connected figure.
\end{enumerate}
The elements of $\mathcal{D}_{n,q}$ are the \textit{coset diagrams of
order $n$} associated with the Hecke group $\mathfrak{H}(q)$. A
permutation $\sigma\in S_n$ induces a map $\sigma:
\mathcal{D}_{n,q}\rightarrow \mathcal{D}_{n,q}$; two diagrams $D_1,
D_2$ are called equivalent, if $\sigma(D_1)=D_2$ for some permutation
$\sigma\in U_n$. The set of orbits ${U_n}\backslash{\mathcal{D}_{n,q}}$ is
in bijective correspondence with the subgroups of index $n$ in
$\mathfrak{H}(q)$, and each equivalence class of diagrams has $(n-1)!$
elements. Further, if $D$ is a diagram corresponding to the index $n$
subgroup $\Delta_D\leq\mathfrak{H}(q)$, then $D$ contains precisely
$\lambda(\Delta_D)$ red loops and $\mu(\Delta_D)$ blue loops. Our next
result, while spelling out certain numerical constraints, also
provides a geometric interpretation for the quantity $\nu(\Delta_D)$.

\begin{proposition}
\label{Prop:DiagramConstraints}
If $D\in\mathcal{D}_{n,q}$ is a diagram containing $k$ blue $q$-gons
and $e$ 
%Christian: Was soll "straight" hier bedeuten?
%   Graphen sind doch "topologische" Objekte. Wie ich mit die Kanten
%   gezeichnet vorstelle, ob gerade oder kurvig, ist ja egal?
%(straight) 
red edges, and with associated subgroup $\Delta_D,$ then 
\begin{enumerate}
\item $\frac{\displaystyle n}{\displaystyle q} \geq k \geq
\frac{\displaystyle n-e-1}{\displaystyle q-1},$ 
\vspace{2mm}
\item $\frac{\displaystyle n}{\displaystyle 2} \geq e\geq
n-(q-1)k-1,$
\vspace{2mm}
\item $e+(q-1)k-n+1 = \nu(\Delta_D)$.
\end{enumerate}
\end{proposition}
\begin{proof}
If $n=1$, then $D$ consists of one vertex labelled $1$, with a red and
a blue loop attached to it. Thus, $k=e=0$ and, by
(\ref{Eq:Type/IndexRel}), $\nu(\Delta_D)=0$. Similarly, for $n=2$, the
diagram $D$ consists of two vertices labelled $1$ and $2$,
respectively, each having a blue loop attached to it, and joined by a
red edge. Thus, $k=0$, $e=1$, and, by (\ref{Eq:Type/IndexRel}),
$\nu(\Delta_D)=0$. One checks that in both cases assertions (i)--(iii)
hold true. Hence, for the rest of the proof, we may assume that $n>2$.

Since, by (b), the $q$-gons are disjoint, we must have $kq\leq n$; the
$n-qk$ vertices not involved in a $q$-gon must carry blue loops. By
(c), there exist at least $k-1$ red edges joining vertices of distinct
blue $q$-gons, and a further $n-qk$ red edges, each joining a vertex
with a blue loop to a vertex of a $q$-gon (here we need that $n>2$).
Thus, there are at least $n-(q-1)k-1$ red edges. Since
\[
\mu(\Delta_D) = n-qk
\]
and, in view of (a),
$$
\lambda(\Delta_D) = n - 2e.
$$
Equation (\ref{Eq:Type/IndexRel}) shows that indeed
$e+(q-1)k-n+1=\nu(\Delta_D)$. The remaining inequalities follow since
$\lambda(\Delta_D), \nu(\Delta_D)\geq0$. 
\end{proof}

It is not hard to see that, conversely, whenever inequalities (i) and
(ii) are satisfied, a diagram $D$ with specifications as described in
Proposition~\ref{Prop:DiagramConstraints} does indeed exist.

\section{Counting finite-index subgroups in Hecke groups via coset diagrams}
\label{Sec:IsoType}

The main purpose of this section is to establish
the fact that, for every odd prime $q$ and each integer
$n\geq1$, we have 
\begin{equation}
\label{Eq:NqnParity}
M_q(n) \equiv 0\ \mathrm{mod}\ 2,
\end{equation}
so that these numbers do in fact not enter into the mod~$2$
calculation of $s_{\Gamma_m(q)}^H(n)$, despite their appearance in
Proposition~\ref{Prop:Reduction1}; cf.\ Theorem~\ref{Thm:NqnParity}.

If $\Delta$ is a subgroup of index $n$ in $\mathfrak{H}(q)$ with
$\mu(\Delta)=0$, then we must have $q\mid n$ by
Equation~(\ref{Eq:Type/IndexRel}) plus the fact that $q$ is a
prime. Hence, 
\begin{equation} \label{eq:qnmidn}
M_q(n) = 0,\quad q\nmid n;
\end{equation}
in particular, $M_q(n)\equiv0$~(mod~$2$) in this case. We are thus reduced
to checking the case where $n=qk$ with $k\geq1$. Using coset
enumeration techniques, the numbers $M_q(qk)$ are expressed as 
a sum of certain combinatorially defined quantities $M_q(qk;e,k)$ divided 
by $q^{k-1}(k-1)!$, see
Corollary~\ref{Cor:SubgpMu=0Enum}(ii). As a next step, an
explicit formula for $M_q(qk;e,k)$ is found in
Proposition~\ref{Prop:NqComp}. Subsequently, this explicit formula is used
to show that $M_q(qk)$ is always even for $q\ge5$; see
Proposition~\ref{Prop:Nq2Part}. Finally, the corresponding fact for $q=3$ 
is established in Proposition~\ref{Prop:N3Parity}, building on the work
of Stothers \cite{Stothers}. A curious side result is
Proposition~\ref{Prop:GJ}, which provides a closed form product formula
for $M_q(qk;e,k)$ in the minimal case $e=k-1$.
As already mentioned in the introduction, the next subsection
addresses, in complete generality, the problem of enumerating the
index $n$ subgroups in $\mathfrak H(q)$ which are isomorphic to
$C_2^{\ast\lambda} \ast C_q^{\ast\mu} \ast F_{\nu}$ for fixed
$\lambda,\mu,\nu$.

\subsection{Enumerating finite-index subgroups of given type}
\label{sec:4.1}
Let $n,m_1,m_2$ be integers with $n>0$ and $m_1,m_2\ge0$. 
Suppose that we are given a set of $m_2$ disjoint oriented blue
$q$-gons, the vertices of the $i$-th one being labelled
$q(i-1)+1,\dots,qi-1,qi$ in order, as well as $n-qm_2$ additional
vertices labelled $qm_2+1,qm_2+2,\dots,n$. 
We consider (mixed) graphs resulting from the previously described
$q$-gons and additional vertices
by drawing $m_1$ undirected red edges, such that each vertex is
incident with at most one red edge, and such that a connected graph
is obtained. (Here, for connectivity, {\it both\/} the blue and red
edges are taken into account.) Let $\mathcal M_q(n;m_1,m_2)$ be
the set of all these graphs, and let $M_q(n;m_1,m_2)$ be the
cardinality of $\mathcal M_q(n;m_1,m_2)$.

Denote by $s_q(n;m_1,m_2)$ the number of index $n$ subgroups
$\Delta$ in $\mathfrak H(q)$ of representation type
\begin{equation} \label{Eq:rept}
{\bf m}(\Delta) = \begin{pmatrix}
              m_1&n-2m_1\\
              m_2&n-qm_2
              \end{pmatrix};
\end{equation}
that is, the generator $x$ in (\ref{Eq:HeckeGenerators}) 
acts as a product of $m_1$ $2$-cycles and
$n-2m_1$ fixed points on the coset space
$\mathfrak{H}(q)/\Delta$, while the generator $y$ in
(\ref{Eq:HeckeGenerators}) acts as a 
product of $m_2$ $q$-cycles and $n-qm_2$ fixed points on
$\mathfrak{H}(q)/\Delta$.
We note that in our situation the representation type 
${\bf m}(\Delta)$ and the isomorphism type ${\bf t}(\De)$ of a finite
index subgroup $\De$ determine each other. More precisely, 
if $\De$ is of index $n$ in $\mathfrak H(q)$, then ${\bf m}(\Delta)$
is given by \eqref{Eq:rept} if and only if ${\bf t}(\De)$ is given by
\begin{equation} \label{Eq:isot}
{\bf t}(\De)=(n-2m_1,n-qm_2,m_1+(q-1)m_2-n+1).
\end{equation}
The numbers $M_q(n;m_1,m_2)$ and $s_q(n;m_1,m_2)$ are related
in the following way.
\begin{lemma} \label{Lem:Zus}
For an odd prime $q$ and integers $n,m_1,m_2$ with $n>0$ and
$m_1,m_2\ge0,$ we have
\begin{equation} \label{Eq:Zus}
s_q(n;m_1,m_2)
=\frac {n M_q(n;m_1,m_2)} {q^{m_2}m_2!\,(n-qm_2)!}.
\end{equation}
\end{lemma}
\begin{proof} 
Clearly, a graph $G\in \mathcal M_q(n;m_1,m_2)$ can be made into
a coset diagam $D\in \mathcal D_{n,q}$ by adding a red loop at each
vertex not incident with a red edge and a blue loop at each vertex
not incident with a blue $q$-gon. The subgroup $\Delta_D$ corresponding
to such a diagram has index $n$ and representation type 
$\left(\begin{smallmatrix}
m_1&n-2m_1\\m_2&n-qm_2\end{smallmatrix}\right)$.
Conversely, such a subgroup leads to a graph 
$G\in \mathcal M_q(n;m_1,m_2)$ for {\it some} labelling of the
vertices by deleting all red and blue loops in a corresponding coset diagram.

In order to construct all possible labellings, first choose labels
for the vertices involved in some $q$-gon in $\binom n{qm_2}$ ways.
These chosen labels can be used to label the $m_2$ $q$-gons in
$$\frac {1} {m_2!}\binom {qm_2}{q,\dots,q}(q-1)!^{m_2}=
\frac {(qm_2)!} {m_2!\,q^{m_2}}$$
ways. On the other hand, there is only one way to label the
additional $(n-qm_2)$ vertices by the remaining labels, these being
completely indistinguishable. Hence, the
number of coset diagrams corresponding to subgroups of
index $n$ and representation type 
$\left(\begin{smallmatrix}
m_1&n-2m_1\\m_2&n-qm_2\end{smallmatrix}\right)$ is
$$
\binom n{qm_2}\frac {(qm_2)!} {m_2!\,q^{m_2}}
M_q(n;m_1,m_2).
$$
Our claim follows now from Proposition~\ref{Prop:TransReps} upon
little simplification.
\end{proof}
Relation (\ref{Eq:Zus}) allows us to compute the subgroup numbers
$s_q(n;m_1,m_2)$ in terms of the geometrically defined quantities
$M_q(n;m_1,m_2)$. As our next result shows, it is enough to
consider the latter numbers in the case when $q\mid n$.

\begin{lemma} \label{Lem:ZusN}
For an odd prime $q$ and integers $n,m_1,m_2$ with $n>0$ and
$m_1,m_2\ge0,$ we have
\begin{equation} \label{Eq:ZusN}
M_q(n;m_1,m_2)
=\frac {(2n-2m_1-qm_2)!} {(n-2m_1)!} M_q(qm_2;m_1+qm_2-n,m_2) .
\end{equation}
\end{lemma}
\begin{proof}
Given a graph in $\mathcal M_q(n;m_1,m_2)$, removal of all
vertices not involved in some $q$-gon together with all incident (red)
edges leaves a graph in $\mathcal M_q(qm_2;m_1+qm_2-n,m_2)$.
Conversely, starting from a graph in $\mathcal
M_q(qm_2;m_1+qm_2-n,m_2)$, among the vertices involved in a $q$-gon
there are $qm_2-2(m_1+qm_2-n)=2n-2m_1-qm_2$ vertices not incident
with a red edge. From these vertices we choose $n-qm_2$ vertices in $\binom
{2n-2m_1-qm_2}{n-qm_2}$ ways and, having chosen them, we connect each
of them by means of a red edge to exactly one of a new set of vertices
labelled $\{qm_2+1,qm_2+2,\dots,n\}$. This last step can be done in
$(n-qm_2)!$ different ways. Hence, in total we obtain the relation
\eqref{Eq:ZusN}.
\end{proof}

Since we shall make use of it later on, we point out that
a combination of Lemmas~\ref{Lem:Zus} and \ref{Lem:ZusN} yields the
relation
\begin{equation} \label{eq:ZussN}
s_q(n;m_1,m_2)
=\frac {n\,(2n-2m_1-qm_2)! } {q^{m_2}m_2!\,(n-qm_2)!\,(n-2m_1)!}
M_q(qm_2;m_1+qm_2-n,m_2) .
\end{equation}
\begin{corollary}
\label{Cor:SubgpMu=0Enum}
{\em(i)} Let $k$ and $e$ be integers with $k\ge1$ and $0\le e\leq qk/2$.
Then the number of subgroups $\Delta$ in $\mathfrak{H}(q)$ of index
$qk$ and type 
\begin{equation} \label{eq:typet}
{\bf t}(\Delta) = \big(qk-2e, 0, e-k+1\big)
\end{equation} 
is 
\[
\frac{M_q(qk;e,k)}{q^{k-1} (k-1)!}.
\]
In  particular, we have
\begin{multline} \label{Eq:Free}
M_q\Big(2q(e-k)/(q-2);e,\frac{2(e-k)}{q-2}\Big) 
= q^{\frac{2(e-k+1)-q}{q-2}}\,
\Big(\frac{2(e-k+1)-q}{q-2}\Big)!\,
f_q\Big(\frac{2q(e-k)}{q-2}\Big),\\
e\geq k+1,\ e\equiv k\ (\mathrm{mod}\ q-2), 
\end{multline}
where $f_q(n)$ is the number of free subgroups of index $n$ in
$\mathfrak{H}(q)$.

{\em(ii)} For $k\geq1,$ the number $M_q(qk)$ of subgroups
$\Delta\leq\mathfrak{H}(q)$ of index $qk$ and with $\mu(\Delta)=0$ is
given by 
\begin{equation}
\label{Eq:SubgpMu=0Enum}
M_q(qk) = \frac {1} {q^{k-1} (k-1)!}\sum_{0\leq e\leq\frac{qk}{2}}
{M_q(qk;e,k)}. 
\end{equation}

{\em(iii)} We have $s_q(1)=s_q(2)=1,$ and, for $n>2,$
\begin{equation}
\label{Eq:sqnComp}
s_q(n) = \sum_{1\leq k\leq\frac{n}{q}}
\,\sum_{0\leq e\leq\frac{n}{2}}\frac{n}{q^k k! }
\,\binom {qk-2e}{n-qk}\,M_q(qk;e,k). 
\end{equation}
\end{corollary}
\begin{proof}
Setting $(\mathfrak H(q):\Delta)=qk$, $\la(\De)=qk-2e$, and
$\mu(\De)=0$ in Equation~\eqref{Eq:Type/IndexRel}, we find that
$\nu(\De)=e-k+1$, so that ${\bf t}(\De)$ agrees with \eqref{eq:typet}.
Hence, setting $n=qk$, $m_1=e$, and $m_2=k$ in \eqref{eq:ZussN}, the
first part of (i) follows. The particular statement in (i) as well as
Assertion (ii) are immediate consequences. Finally, Assertion~(iii)
results upon setting $m_1=e+n-qk$ and $m_2=k$ in \eqref{eq:ZussN} and
summing over all $e$ and $k$.
\end{proof}

\subsection{Calculation of the numbers $M_q(qk;e,k)$}
\label{sec:4.2}
Our next result provides a rather complicated looking but explicit
formula for the quantities mentioned in the title. This formula in
turn will enable us to determine the parity of the numbers $M_q(n)$
for $q\ge5$.
\begin{proposition}
\label{Prop:NqComp}
For $k\geq1$ and $k-1\leq e\leq\frac{qk}{2},$ we have
\begin{equation}
\label{Eq:NqComp}
M_q(qk;e,k) =
\sum_{\gamma=1}^{k}\,\frac{(-1)^{\gamma-1}}{\gamma}
\underset{\alpha_1+\cdots+\alpha_\gamma=e}
{\sum_{\alpha_1,\ldots,\alpha_\gamma\geq0}}\,
\underset{\rho_1+\cdots+\rho_\gamma=k}
{\sum_{\rho_1,\dots,\rho_\gamma\geq1}}
\binom{k}{\rho_1,\ldots,\rho_\gamma}
\prod_{i=1}^\gamma
\frac {(q\rho_i)!}{2^{\al_i}\alpha_i!\,(q\rh_i-2\al_i)!}.
\end{equation}
\end{proposition}
\begin{proof}
According to the definition of $M_q(qk;e,k)$, we want to
enumerate the elements of the set $\mathcal M_q(qk;e,k)$; that
is, connected graphs consisting of $k$ oriented blue $q$-gons and $e$
(unoriented) red edges connecting vertices of the $q$-gons in 
such a way that no two red edges share a vertex.

Let the $q$-gons be denoted by $P_1,P_2,\dots,P_k$ in order (according
to their smallest vertex label), and 
let $\Pi(k)$ be the set of all partitions of the standard set 
$[k]$.

For a partition $\pi\in \Pi(k)$ with $\ga$ blocks,
we write $M_{=\pi}(k;e_1,e_2,\dots, e_\ga)$ for the number of ways to
draw $e=e_1+e_2+\dots+e_\ga$ red edges among the vertices of the
$q$-gons in such a way that no two edges share a vertex, and 
such that the connectivity
structure of the resulting graph corresponds to the partition $\pi$;
that is, if $\{i_1,i_2,\dots,i_j\}$ is the $l$-th block of
$\pi$ (the blocks can be canonically ordered with respect to their
smallest elements), the $q$-gons
$P_{i_1},P_{i_2},\dots,P_{i_j}$ form a connected component of the
graph and, moreover, there are exactly $e_l$ red edges in this
component. Analogously, let $M_{\le\pi}(k;e_1,e_2,\dots,e_\ga)$ be the
number of ways to draw $e=e_1+e_2+\dots+e_\ga$ red edges 
such that no two edges share a vertex and
such that the
connectivity structure of the resulting graph is described by a
partition which is equal to or finer than $\pi$
(in the usual partial order on
set partitions; cf., for instance, \cite[Example~3.1.1(d)]{StanleyI}), 
and such that the same condition is satisfied
with regard to the distribution of the red edges.

Clearly, for fixed $e$ and $\pi\in\Pi(k)$, and denoting by
$\vert\pi\vert$ the number of blocks of $\pi$, we have 
$$
\sum _{e_1+\dots+e_{\vert\pi\vert}=e}
M_{\le\pi}(k;e_1,e_2,\dots,e_{\vert\pi\vert})=
\sum _{\si\le \pi}
\sum _{f_1+\dots+f_{\vert\sigma\vert}=e} 
M_{=\si}(k;f_1,f_2,\dots,f_{\vert\sigma\vert}).$$
M\"obius inversion (cf.\ \cite[Propositions~2--3]{Rota}) yields
$$
\sum _{e_1+\dots+e_{\vert\pi\vert}=e} 
M_{=\pi}(k;e_1,e_2,\dots,e_{\vert\pi\vert})=
\sum _{\si\le \pi} 
\sum _{f_1+\dots+f_{\vert\sigma\vert}=e} ^{}
\mu(\si,\pi)\,M_{\le\si}(k;f_1,f_2,\dots,f_{\vert\sigma\vert}),$$
where $\mu$ denotes the M\"obius function of the partition lattice
$\Pi(k)$.
In particular, for
$$\pi=\{[k]\}=:\hat 1$$ 
(the maximum element in the partition lattice $\Pi(k)$), we have
\begin{equation} 
M_q(qk;e,k)=M_{=\hat1}(k;e)
=\sum _{\si \in\Pi(k)} ^{}
\sum _{f_1+\dots+f_{\vert\sigma\vert}=e}
  \mu(\si,\hat 1)\,
M_{\le\si}(k;f_1,f_2,\dots,f_{\vert\sigma\vert}).
\label{Eq:max}
\end{equation}
The numbers $M_{\le\si}(k;f_1,f_2,\dots,f_{\vert\si\vert})$ are easily
determined: if $\{i_1,i_2,\dots,i_j\}$ is the $l$-th block of
$\si$, then this means that $f_l$ edges are to be drawn {\it
arbitrarily} among the vertices of the $q$-gons 
$P_{i_1},P_{i_2},\dots,P_{i_j}$, subject only to the restriction that
no two edges are allowed to share a vertex. The number of 
ways to do this is
$$\frac {(jq)(jq-1)\cdots(jq-2f_l+1)} {2^{f_l}f_l!}.$$
Hence, if $\rh_1,\rh_2,\dots,\rh_{\vert\sigma\vert}$ are the block sizes of
$\si$, we have
$$M_{\le\si}(k;f_1,f_2,\dots,f_{\vert\sigma\vert})=
\prod_{i=1}^{\vert\sigma\vert}
\frac {(q\rho_i)!}{2^{f_i}f_i!\,(q\rh_i-2f_i)!}.
$$
The M\"obius function $\mu(\si,\hat 1)$ is known as well, 
namely one has
$$\mu(\si,\hat 1)=(-1)^{\vert\sigma\vert-1}(\vert\sigma\vert-1)!;$$
cf., for instance, \cite[Example~3.10.4]{StanleyI}. 
Finally, for fixed $\rh_1,\rh_2,\dots,\rh_\ga$ with
$\rh_1+\rh_2+\dots+\rh_\ga=k$, the number of {\it ordered\/} partitions 
of $\{1,2,\dots,k\}$ (here, ``ordered" means that the order of the
blocks matters), 
the block sizes of which are $\rh_1,\rh_2,\dots,\rh_\ga$, is given by
$\binom k{\rh_1,\rh_2,\dots,\rh_\ga}$. Every partition in $\Pi(k)$
with $\ga$ blocks giving rise to exactly $\ga!$ {\it ordered\/}
partitions of $\{1,2,\dots,k\}$ by permuting the blocks, we must in
the end divide by $\ga!$ in order to get rid of the overcounting. If
everything is put together, \eqref{Eq:max} transforms into
\eqref{Eq:NqComp}.
\end{proof}
\begin{remark}
A combination of \eqref{eq:ZussN} and Proposition~\ref{Prop:NqComp} 
yields an explicit formula for the number
$s_q(n;m_1,m_2)$ of index $n$ subgroups in $\mathfrak H(q)$ of
representation type 
$\left(\begin{smallmatrix}
m_1&n-2m_1\\m_2&n-qm_2\end{smallmatrix}\right)$.
\end{remark}
Although this is not apparent from (\ref{Eq:NqComp}), the number
$M_q(qk;e;k)$ admits a simple product formula representation in
the case when $e=k-1$; that is, when the underlying graph formed by
the red edges and the $q$-gons (collapsed to vertices) is a
tree. Rather embarrassingly, we have not been able to deduce
Formula~(\ref{Eq:GJ}) below directly from the formula of
Proposition~\ref{Prop:NqComp}. 
\begin{proposition}
\label{Prop:GJ}
We have $M_q(q;0,1)=1$, and for $k\geq2$,
\begin{equation}
\label{Eq:GJ}
M_q(qk;k-1,k) = q^k ((q-1)k)((q-1)k-1)\cdots (qk-2k+3),
\end{equation}
where an empty product must be interpreted as $1$.
\end{proposition}

\begin{proof} 
The proof consists in converting the problem of counting the
elements of\break $\mathcal M_q(qk;k-1,k)$ into a counting problem for
certain planar maps. The latter problem has already been solved by
Goulden and Jackson in \cite{GoJaAS} in connection with the
computation of connection coefficients for the symmetric group.

Indeed, recall that the set $\mathcal M_q(qk;k-1,k)$ consists
only of connected graphs.
Thus, the only way to generate an 
element of $\mathcal M_q(qk;k-1,k)$ out of 
$k$ blue $q$-gons and $k-1$ red edges is by starting from a tree with
vertices labelled $v_1,v_2,\dots,v_k$ (and, hence, $k-1$ edges, which we
assume to be red), blowing up the vertices of
the tree to $q$-gons, and gluing one end of a red edge originally
connecting $v_i$ and $v_j$ to a vertex of the polygon corresponding
to $v_i$, the other to a vertex of the polygon corresponding
to $v_j$, in such a way that no two red edges share a vertex.
Finally, we label the vertices of the polygon corresponding to $v_i$
by $qi-q+1,\dots,qi-1,qi$ in circular order.

Such an object can be embedded canonically in the plane 
without crossings of edges by
requiring that all polygons are embedded with clockwise circular
labelling. Deleting all labels 
and marking the vertex originally labelled by $1$,
we obtain a certain set $\widetilde{\mathcal M}(q,k)$ of graphs in
which one vertex is marked. 
Figure~\ref{fig:1}.a shows such a graph in $\widetilde{\mathcal
M}(3,5)$; there, the marked vertex is indicated by a black square, 
red edges are indicated as undirected edges,
while the blue edges are the directed edges.
We observe that
$$M_q(qk;k-1,k)=q^{k-1}(k-1)!\,\big\vert\widetilde{\mathcal M}(q,k)\big\vert,$$
since, starting with an object from $\widetilde{\mathcal M}(q,k)$, we have
$(k-1)!$ possibilities to decide from which set of the form
$\{qi-q+1,\dots,qi-1,qi\}$ to take the labels for a given unmarked
$q$-gon, and subsequently, for each of the $k-1$ unmarked $q$-gons, 
we have $q$ possibilities where to start the clockwise labelling.
The problem of counting the elements of $\mathcal M_q(qk;k-1,k)$
has thus been reduced to the problem of finding the cardinality
of the set $\widetilde{\mathcal M}(q,k)$.

\begin{figure}
$$
\Einheit=.35cm
\setlength{\unitlength}{.35cm}
\raise10 \Einheit\hbox to0pt{\hskip0 \Einheit
           \raise-4pt\hbox to0pt{\hss\Large$\blacksquare$\hss}\hss}
\DickPunkt(4,0)
\DickPunkt(2,6)
\DickPunkt(8,4)
%\DickPunkt(0,10)
\DickPunkt(6,10)
\DickPunkt(2,14)
\DickPunkt(6,16)
\DickPunkt(12,14)
\DickPunkt(10,18)
\DickPunkt(12,0)
\DickPunkt(18,4)
\DickPunkt(14,6)
\DickPunkt(18,10)
\DickPunkt(14,12)
\DickPunkt(16,16)
\Pfad(18,4),222222\endPfad
\unskip\leavevmode
\raise0 \Einheit\hbox to0pt{\hskip4 \Einheit \thicklines
          \vector(-1,3){1}\hss}
\raise3 \Einheit\hbox to0pt{\hskip3 \Einheit \thicklines
          \line(-1,3){1}\hss}
\raise6 \Einheit\hbox to0pt{\hskip2 \Einheit \thicklines
          \vector(3,-1){3}\hss}
\raise5 \Einheit\hbox to0pt{\hskip5 \Einheit \thicklines
          \line(3,-1){3}\hss}
\raise4 \Einheit\hbox to0pt{\hskip8 \Einheit \thicklines
          \vector(-1,-1){2}\hss}
\raise2 \Einheit\hbox to0pt{\hskip6 \Einheit \thicklines
          \line(-1,-1){2}\hss}
\raise10 \Einheit\hbox to0pt{\hskip0 \Einheit \thicklines
          \vector(1,2){1}\hss}
\raise12 \Einheit\hbox to0pt{\hskip1 \Einheit \thicklines
          \line(1,2){1}\hss}
\raise14 \Einheit\hbox to0pt{\hskip2 \Einheit \thicklines
          \vector(1,-1){2}\hss}
\raise12 \Einheit\hbox to0pt{\hskip4 \Einheit \thicklines
          \line(1,-1){2}\hss}
\raise10 \Einheit\hbox to0pt{\hskip6 \Einheit \thicklines
          \vector(-1,0){3}\hss}
\raise10 \Einheit\hbox to0pt{\hskip3 \Einheit \thicklines
          \line(-1,0){3}\hss}
\raise16 \Einheit\hbox to0pt{\hskip6 \Einheit \thicklines
          \vector(2,1){2}\hss}
\raise17 \Einheit\hbox to0pt{\hskip8 \Einheit \thicklines
          \line(2,1){2}\hss}
\raise18 \Einheit\hbox to0pt{\hskip10 \Einheit \thicklines
          \vector(1,-2){1}\hss}
\raise16 \Einheit\hbox to0pt{\hskip11 \Einheit \thicklines
          \line(1,-2){1}\hss}
\raise14 \Einheit\hbox to0pt{\hskip12 \Einheit \thicklines
          \vector(-3,1){3}\hss}
\raise15 \Einheit\hbox to0pt{\hskip9 \Einheit \thicklines
          \line(-3,1){3}\hss}
\raise12 \Einheit\hbox to0pt{\hskip14 \Einheit \thicklines
          \vector(1,2){1}\hss}
\raise14 \Einheit\hbox to0pt{\hskip15 \Einheit \thicklines
          \line(1,2){1}\hss}
\raise16 \Einheit\hbox to0pt{\hskip16 \Einheit \thicklines
          \vector(1,-3){1}\hss}
\raise13 \Einheit\hbox to0pt{\hskip17 \Einheit \thicklines
          \line(1,-3){1}\hss}
\raise10 \Einheit\hbox to0pt{\hskip18 \Einheit \thicklines
          \vector(-2,1){2}\hss}
\raise11 \Einheit\hbox to0pt{\hskip16 \Einheit \thicklines
          \line(-2,1){2}\hss}
\raise0 \Einheit\hbox to0pt{\hskip12 \Einheit \thicklines
          \vector(1,3){1}\hss}
\raise3 \Einheit\hbox to0pt{\hskip13 \Einheit \thicklines
          \line(1,3){1}\hss}
\raise6 \Einheit\hbox to0pt{\hskip14 \Einheit \thicklines
          \vector(2,-1){2}\hss}
\raise5 \Einheit\hbox to0pt{\hskip16 \Einheit \thicklines
          \line(2,-1){2}\hss}
\raise4 \Einheit\hbox to0pt{\hskip18 \Einheit \thicklines
          \vector(-3,-2){3}\hss}
\raise2 \Einheit\hbox to0pt{\hskip15 \Einheit \thicklines
          \line(-3,-2){3}\hss}
\raise6 \Einheit\hbox to0pt{\hskip2 \Einheit \thicklines
          \line(-1,2){2}\hss}
\raise10 \Einheit\hbox to0pt{\hskip6 \Einheit \thicklines
          \line(2,-1){8}\hss}
\raise14 \Einheit\hbox to0pt{\hskip2 \Einheit \thicklines
          \line(2,1){4}\hss}
\hbox{\hskip8cm}%
\DickPunkt(4,0)
\DickPunkt(2,6)
\DickPunkt(8,4)
\DickPunkt(0,10)
\DickPunkt(6,10)
\DickPunkt(2,14)
\DickPunkt(6,16)
\DickPunkt(12,14)
\DickPunkt(10,18)
\DickPunkt(12,0)
\DickPunkt(18,4)
\DickPunkt(14,6)
\DickPunkt(18,10)
\DickPunkt(14,12)
\DickPunkt(16,16)
\Pfad(18,4),222222\endPfad
\unskip\leavevmode
\raise0 \Einheit\hbox to0pt{\hskip4 \Einheit \thicklines
          \vector(-1,3){1}\hss}
\raise3 \Einheit\hbox to0pt{\hskip3 \Einheit \thicklines
          \line(-1,3){1}\hss}
\raise6 \Einheit\hbox to0pt{\hskip2 \Einheit \thicklines
          \vector(3,-1){3}\hss}
\raise5 \Einheit\hbox to0pt{\hskip5 \Einheit \thicklines
          \line(3,-1){3}\hss}
\raise4 \Einheit\hbox to0pt{\hskip8 \Einheit \thicklines
          \vector(-1,-1){2}\hss}
\raise2 \Einheit\hbox to0pt{\hskip6 \Einheit \thicklines
          \line(-1,-1){2}\hss}
\raise10 \Einheit\hbox to0pt{\hskip0 \Einheit \thicklines
          \vector(1,2){1}\hss}
\raise12 \Einheit\hbox to0pt{\hskip1 \Einheit \thicklines
          \line(1,2){1}\hss}
\raise14 \Einheit\hbox to0pt{\hskip2 \Einheit \thicklines
          \vector(1,-1){2}\hss}
\raise12 \Einheit\hbox to0pt{\hskip4 \Einheit \thicklines
          \line(1,-1){2}\hss}
\raise10 \Einheit\hbox to0pt{\hskip6 \Einheit \thicklines
          \vector(-1,0){3}\hss}
\raise10 \Einheit\hbox to0pt{\hskip3 \Einheit \thicklines
          \line(-1,0){3}\hss}
\raise16 \Einheit\hbox to0pt{\hskip6 \Einheit \thicklines
          \vector(2,1){2}\hss}
\raise17 \Einheit\hbox to0pt{\hskip8 \Einheit \thicklines
          \line(2,1){2}\hss}
\raise18 \Einheit\hbox to0pt{\hskip10 \Einheit \thicklines
          \vector(1,-2){1}\hss}
\raise16 \Einheit\hbox to0pt{\hskip11 \Einheit \thicklines
          \line(1,-2){1}\hss}
\raise14 \Einheit\hbox to0pt{\hskip12 \Einheit \thicklines
          \vector(-3,1){3}\hss}
\raise15 \Einheit\hbox to0pt{\hskip9 \Einheit \thicklines
          \line(-3,1){3}\hss}
\raise12 \Einheit\hbox to0pt{\hskip14 \Einheit \thicklines
          \vector(1,2){1}\hss}
\raise14 \Einheit\hbox to0pt{\hskip15 \Einheit \thicklines
          \line(1,2){1}\hss}
\raise16 \Einheit\hbox to0pt{\hskip16 \Einheit \thicklines
          \vector(1,-3){1}\hss}
\raise13 \Einheit\hbox to0pt{\hskip17 \Einheit \thicklines
          \line(1,-3){1}\hss}
\raise10 \Einheit\hbox to0pt{\hskip18 \Einheit \thicklines
          \vector(-2,1){2}\hss}
\raise11 \Einheit\hbox to0pt{\hskip16 \Einheit \thicklines
          \line(-2,1){2}\hss}
\raise0 \Einheit\hbox to0pt{\hskip12 \Einheit \thicklines
          \vector(1,3){1}\hss}
\raise3 \Einheit\hbox to0pt{\hskip13 \Einheit \thicklines
          \line(1,3){1}\hss}
\raise6 \Einheit\hbox to0pt{\hskip14 \Einheit \thicklines
          \vector(2,-1){2}\hss}
\raise5 \Einheit\hbox to0pt{\hskip16 \Einheit \thicklines
          \line(2,-1){2}\hss}
\raise4 \Einheit\hbox to0pt{\hskip18 \Einheit \thicklines
          \vector(-3,-2){3}\hss}
\raise2 \Einheit\hbox to0pt{\hskip15 \Einheit \thicklines
          \line(-3,-2){3}\hss}
\raise6 \Einheit\hbox to0pt{\hskip2 \Einheit \thicklines
          \line(-1,2){2}\hss}
\raise10 \Einheit\hbox to0pt{\hskip6 \Einheit \thicklines
          \line(2,-1){8}\hss}
\raise14 \Einheit\hbox to0pt{\hskip2 \Einheit \thicklines
          \line(2,1){4}\hss}
\Label\u{3}(4,0)
\Label\lu{\hbox{4\hskip5pt}}(2,6)
\Label\ru{\hbox{\hskip5pt 2}}(8,4)
\Label\l{1}(0,10)
\Label\u{9}(6,10)
\Label\lo{\hbox{5\hskip5pt}}(2,14)
\Label\u{8}(6,16)
\Label\ro{\hbox{\hskip7pt 7}}(12,14)
\Label\o{6}(10,18)
\Label\u{14}(12,0)
\Label\u{10}(18,4)
\Label\ro{\hbox{\hskip5pt 15}}(14,6)
\Label\r{13}(18,10)
\Label\lu{\hbox{11\hskip5pt}}(14,12)
\Label\o{12}(16,16)
\hskip7cm
$$
\vskip10pt
\centerline{\Small
a. An element of $\widetilde{\mathcal M}(3,5)$\hskip3cm
b. The labelling of vertices}
\caption{}
\label{fig:1}
\end{figure}

To make the link with \cite{GoJaAS}, given an element of $\widetilde{\mathcal
M}(q,k)$, we translate it into a factorization 
\begin{equation} \label{eq:al1al2} 
(1,2,\dots,qk)=\pi_1\circ \pi_2
\end{equation}
of the ``long" cycle $(1,2,\dots,qk)$ into the product of two
permutations in $S_{qk}$, $\pi_1$ consisting of $k$ cycles of length
$q$, and $\pi_2$ consisting of $k-1$ cycles of length $2$ and fixed
points otherwise. To explain this translation, consider 
Figure~\ref{fig:1}, which illustrates an example for $q=3$ and $k=5$. 
Given an element of $\widetilde{\mathcal M}(q,k)$, 
we determine labels for all vertices in the
following way: the marked vertex is labelled $1$. Now we suppose that we
already have labelled $i$ vertices by $1,2,\dots,i$. Placing ourselves
in the vertex labelled $i$, $v_i$ say, 
there are two possibilities: either this
vertex is incident to a red edge or not. In the first case, we move
from $v_i$ along the red edge, arriving in the vertex $u$, say, and
then continue along the blue edge emanating from $u$. The vertex which
we reach at the other end of this blue edge is labelled $i+1$. In the
second case, we simply move along the blue edge emanating from $v_i$,
and we label the vertex which we reach at the other end of this blue
edge by $i+1$. Figure~\ref{fig:1}.b shows the
resulting labelling in our example. From the labelling we can read off
a factorization \eqref{eq:al1al2} by interpreting a $q$-gon
with vertices labelled $j_1,j_2,\dots,j_q$ in clockwise order as the
cycle $(j_1,j_2,\dots,j_q)$ and letting $\pi_1$ be the product of all
these cycles, and by interpreting a red edge with end vertices
$j_1,j_2$ as the transposition $(j_1,j_2)$ and letting $\pi_2$ be the product
of all these transpositions. In this way, our example in
Figure~\ref{fig:1} corresponds to the factorization
$$(1,2,\dots,15)=\pi_1\circ \pi_2,$$
where 
$$\pi_1=(1,5,9)(2,3,4)(6,7,8)(10,14,15)(11,12,13)$$
and
$$\pi_2=(1,4)(5,8)(9,15)(10,13).$$
It is not difficult to see that this translation defines a bijection
between elements of $\widetilde{\mathcal M}(q,k)$ and factorizations
\eqref{eq:al1al2} where the disjoint cycle factorization of 
$\pi_1$ consists of $k$ cycles of length $q$, and 
where the disjoint cycle factorization of 
$\pi_2$ consists of $k-1$ cycles of length $2$ and fixed points otherwise.
The solution of the enumeration 
problem for these objects is then found in \cite[Theorem~3.2]{GoJaAS}
by specializing $n=qk$, $m=2$, $\al_1=(q^k)$, and
$\al_2=(2^{k-1},1^{qk-2k+2})$. 
\end{proof}

\subsection{The parity of the numbers $M_q(n)$}
\label{Subsec:NqnParity}
Let $I_n$ denote the number of solutions of the equation $x^2=1$ in
the symmetric group $S_n$. It is well known that 
\begin{equation} \label{Eq:In}
I_n=\sum _{\al\geq0}\frac {n!} {2^\al \al!\,(n-2\al)!},
\end{equation}
if we set $\frac{1}{n!}=0$ for integers $n<0$ in accordance with the
behaviour of the gamma function; cf. \cite[Equation~(4)]{CHM}. 
The exact value of the $2$-adic valuation of $I_n$ has been determined by
Ochiai \cite[Sec.~3.2]{Ochiai}. The result is that
$$
v_2(I_n)=\begin{cases} 
n/4,&\text {if }n\equiv 0\ (\mathrm{mod}\ 4),\\[2mm]
(n-1)/4,&\text {if }n\equiv 1\ (\mathrm{mod}\ 4),\\[2mm]
(n+2)/4,&\text {if }n\equiv 2\ (\mathrm{mod}\ 4),\\[2mm]
(n+5)/4,&\text {if }n\equiv 3\ (\mathrm{mod}\ 4),
\end{cases}
$$
where, as usual, $v_2(\al)$ stands for the $2$-adic valuation of
$\al$. 
For our purposes, the weaker estimate
\begin{equation} \label{Eq:2In}
v_2(I_n)\ge \begin{cases} 
{n}/ {4},&\text {if $n\not\equiv 1\ (\mathrm{mod}\ 4)$},\\[2mm] 
({n-1})/ {4},&\text {if $n\equiv 1\ (\mathrm{mod}\ 4)$}.
\end{cases}
\end{equation} 
suffices, which already follows from \cite[Theorem~10]{CHM}.

The following auxiliary result (Lemma~\ref{lem:multi}), 
whose proof will be given in Section~\ref{Subsec:Lem12} in the appendix, 
is needed to bound the $2$-adic valuation 
of the numbers $M_q(qk)$ in the case when $q\geq5$; see
Proposition~\ref{Prop:Nq2Part}.

\begin{lemma} \label{lem:multi}
If $\rh_1\equiv \rh_2\equiv\dots\equiv \rh_\al\equiv1\ (\mathrm{mod}\ 4),$ then
\begin{equation} \label{eq:multi1}
v_2\left(\binom{\rho_1+\rh_2+\cdots+\rho_\al}
{\rho_1,\rho_2,\ldots,\rho_\al}\right)\ge
v_2(\al!).
\end{equation}
If $\rh_1\equiv \rh_2\equiv\dots\equiv \rh_\al\equiv3\ (\mathrm{mod}\ 4),$ then
\begin{equation} \label{eq:multi2}
v_2\left(\binom{\rho_1+\rh_2+\cdots+
\rho_\al}{\rho_1,\rh_2,\ldots,\rho_\al}\right)\ge
v_2\big((3\al)!\big)-\al.
\end{equation}
\end{lemma}
\begin{proposition} \label{Prop:Nq2Part}
For a prime number $q\geq5$ and an integer $k\geq1,$ the number
$M_q(qk)$ of subgroups in $\mathfrak H(q)$ 
of index $qk$ and with $\mu(\De)=0$ satisfies
\begin{equation} \label{Eq:v2Nqk}
v_2(M_q(qk))\ge
\frac {qk-1} {4}-v_2\big((k-1)!\big)-\fl{\log_2k}.
\end{equation}
In particular, $M_q(qk)$ is even for $q\geq5$ and $k\geq1$.
\end{proposition}
\begin{proof}
By Part (ii) of Corollary~\ref{Cor:SubgpMu=0Enum},
Proposition~\ref{Prop:NqComp}, and 
\eqref{Eq:In}, we have 
\begin{align}
\notag
M_q(qk)&=
\frac {1} {q^{k-1} (k-1)!}\sum_{0\leq e\leq\frac{qk}{2}}
\,\sum_{\gamma=1}^{k}\,\frac{(-1)^{\gamma-1}}{\gamma}\\
\notag
&\kern4cm
\cdot
\underset{\alpha_1+\cdots+\alpha_\gamma=e}
{\sum_{\alpha_1,\ldots,\alpha_\gamma\geq0}}\,
\,\underset{\rho_1+\cdots+\rho_\gamma=k}
{\sum_{\rho_1,\dots,\rho_\gamma\geq1}}
\binom{k}{\rho_1,\ldots,\rho_\gamma}
\prod_{i=1}^\gamma
\frac {(q\rho_i)!}{2^{\al_i}\alpha_i!\,(q\rh_i-2\al_i)!}\\
&=
\sum_{\gamma=1}^{k}\,
\underset{\rho_1+\cdots+\rho_\gamma=k}
{\sum_{\rho_1,\dots,\rho_\gamma\geq1}}
\,\frac {1} {q^{k-1} (k-1)!}\,
\frac{(-1)^{\gamma-1}}{\gamma}
\,\binom{k}{\rho_1,\ldots,\rho_\gamma}
\prod_{i=1}^\gamma
I_{q\rh_i}.
\label{Eq:v2A}
\end{align}
For fixed $\ga$ and $\rh_1,\rh_2,\dots,\rh_\ga$, we bound the
$2$-adic valuation of the
corresponding summand in the sum in the last line, 
$S(\ga;\rh_1,\dots,\rh_\ga)$ say. Namely, without loss
of generality, let $\rh_1,\rh_2,\dots,\rh_\al$ be the $\rh_j$'s for which
$q\rh_j\equiv1\ (\mathrm{mod}\ 4)$. Then we have
$$
\binom{k}{\rho_1,\ldots,\rho_\gamma}
=
\binom{\rh_1+\rh_2+\dots+\rh_\al}{\rho_1,\ldots,\rho_\al}
\binom{k}{\rho_1+\dots+\rh_\al,\rh_{\al+1},\ldots,\rho_\gamma},
$$
and hence, by a combination
of \eqref{Eq:2In} and Lemma~\ref{lem:multi},
\begin{align*}
v_2\big(S(\ga;\rh_1,&\dots,\rh_\ga)\big)\\
&\ge -v_2\big((k-1)!\big)-\fl{\log_2\ga}+
v_2\left(
\binom{\rh_1+\rh_2+\dots+\rh_\al}{\rho_1,\ldots,\rho_\al}
\right)
+\sum_{i=1}^\gamma v_2(I_{q\rh_i})\\
&\ge -v_2\big((k-1)!\big)-\fl{\log_2k}+\frac {\max\{\al-1,0\}} {2}
+\frac {q(\rh_1+\rh_2+\dots+\rh_\ga)-\al} {4}\\
&\ge -v_2\big((k-1)!\big)-\fl{\log_2k}
+\frac{\max\{\al-2,-\al\}} {4}+\frac {qk} 4.
\end{align*}
The $2$-adic valuation of $M_q(qk)$ is at least the minimum of the 
expression displayed in the last line taken over all possible choices of $\al$. This minimum is
exactly the expression on the right-hand side of \eqref{Eq:v2Nqk}.

To see that the right-hand side of \eqref{Eq:v2Nqk} is always
positive for $q\geq5$, one observes that\break $v_2\big((k-1)!\big)\le k-2$
as long as $k\ge2$, and thus
$$v_2(M_q(qk))\ge 
\frac {(q-4)k+7} {4}-\fl{\log_2k}\ge
\frac {k+7} {4}-\fl{\log_2k}>0,
$$
provided that $k\ge2$. For $k=1$, it can be verified directly that
the right-hand side of \eqref{Eq:v2Nqk} is positive.
\end{proof}
Proposition~\ref{Prop:Nq2Part} leaves open the case when $q=3$, which
is settled in Proposition~\ref{Prop:N3Parity} below, making use of
results in \cite{Stothers}. The following auxiliary result 
(Lemma~\ref{Lem:BinomParity}) will be
used in the proof of Proposition~\ref{Prop:N3Parity}; it also bears on
the parity of the numbers $M_q(qk;k-1,k)/(k-1)!$ in the case
when $q$ is a Fermat prime; see Corollary~\ref{Cor:NqTreeCaseParity}. 
The (straightforward but somewhat lengthy) proof of 
Lemma~\ref{Lem:BinomParity} is recorded in Section~\ref{Subsec:Lem14}
in the appendix.

\begin{lemma}
\label{Lem:BinomParity}
Let $\lambda,k\geq1$ be integers. Then
\begin{multline}
\label{Eq:BinomParityCond}
\binom{2^\lambda k+1}{k-1} \equiv 1\ (\mathrm{mod}\ 2)\
\text { if, and only if, }\\
k=\frac{2^{\lambda\sigma}-1}{2^\lambda-1}\,\mbox{ or
}\,k=\frac{2^{\lambda\sigma+1}-2}{2^\lambda-1}\,\mbox{ for some }\,
\sigma\geq1. 
\end{multline}
\end{lemma}
\begin{corollary}
\label{Cor:NqTreeCaseParity}
Let $q$ be a Fermat prime, and let $k$ be a positive integer. Then 
\begin{multline*}
\frac{M_q(qk;k-1,k)}{(k-1)!}\equiv1\ (\mathrm{mod}\ 2)\\
\text { if, and only if, }\
k=\frac{(q-1)^\sigma-1}{q-2}\mbox{
or } k=2 \frac{(q-1)^\sigma-1}{q-2}\mbox{ for some }\sigma\geq1. 
\end{multline*}
\end{corollary}
\begin{proof}
For $k\geq2$, Equation~(\ref{Eq:GJ}) gives
\[
M_q(qk;k-1,k) = q^k\,((q-1)k+1)
((q-1)k)\cdots((q-1)k-k+3)\,\frac{1}{(q-1)k+1}
\]
so that, modulo~$2$,
\[
\frac{M_q(qk;k-1,k)}{(k-1)!} = q^k \binom{(q-1)k+1}{k-1} \frac{1}{(q-1)k+1}\, \equiv\ \binom{(q-1)k+1}{k-1},
\]
a congruence which is also seen to hold for $k=1$. Since $q$ is a Fermat prime, we
have $q-1=2^\lambda$ for some $\lambda\geq1$, and our claim follows
from Lemma~\ref{Lem:BinomParity}. 
\end{proof}
\begin{proposition}
\label{Prop:N3Parity}
For $k\geq1,$ the number $M_3(3k)$ of subgroups $\Delta$ of index $3k$
in the modular group $\mathfrak{H}(3)=\mathrm{PSL}_2(\Z)$ with the
property that $\mu(\Delta)=0,$ is even. 
\end{proposition}
\begin{proof}
Our starting point is again Corollary~\ref{Cor:SubgpMu=0Enum} (ii),
more precisely, the fact that 
\begin{equation}
\label{Eq:N3SumDecomp}
M_3(3k) = 3^{-(k-1)}\sum_{k-1\leq e\leq \frac{3k}{2}}
\frac{M_3(3k;e,k)}{(k-1)!},\quad k\geq1. 
\end{equation}
Concerning the summands occurring on the right-hand side of
(\ref{Eq:N3SumDecomp}), Stothers shows the following: 
\begin{align}
\frac{M_3(3k;k-1,k)}{(k-1)!} \,&=\, \frac{3^k\, (2k)!}{(k-1)!\,
(k+2)!},\quad k\geq1,\label{Eq:N3k,k-1}\\[2mm] 
\frac{M_3(3k;k,k)}{(k-1)!}\, &=\, 2^{2k-2}\,3^k,\quad
k\geq1,\label{Eq:N3k,k}\\[2mm] 
\frac{M_3(6\ell;3\ell,2\ell)}{(2\ell-1)!}\,
&=\,3^{2\ell-1}\,f_3(6\ell),\quad\ell\geq1, 
\end{align}
and
\begin{equation}
\label{Eq:N3kGeneralCase}
\frac{M_3(3k;e,k)}{(k-1)!}\, =\,
2^{3k-2e-1}\,3^k\,k\,
\frac{\prod_{\ell=0}^{3k-2e-2}(3k-e-2\ell-2)}{(3k-2e)!}\,f_3(6(e-k)),\quad
k< e< \frac{3k}{2}, 
\end{equation}
where $f_3(n)$ is the number of free subgroups of index $n$
in $\mathfrak H(3)$, and the product in (\ref{Eq:N3kGeneralCase}) has to be evaluated as $1$ if $2e=3k-1$;
cf.\ Propositions~1.7 and 1.8 and Formula (3) in \cite{Stothers}. Of
course, Equation~(\ref{Eq:N3k,k-1}) also follows from our
Proposition~\ref{Prop:GJ}. By (\ref{Eq:N3kGeneralCase}) and Legendre's
formula for the $p$-adic valuation of factorials, we have, 
\begin{multline}
\label{Eq:N3General2Part}
v_2\Big(\frac{M_3(3k;e,k)}{(k-1)!}\Big) =
\big(\mathfrak{s}_2(3k-2e)-1\big)\,+\,v_2(k)\,+\,v_2(f_3(6(e-k)))\\
+\,
\sum_{\ell=0}^{3k-2e-2}
v_2(3k-e-2\ell-2),\quad \quad \quad \quad  k<e<\frac{3k}{2}, 
\end{multline}
where $\mathfrak{s}_2(x)$ is the sum of digits in the binary expansion
of $x$. Further, it is known that $f_3(6\lambda)$ is odd if, and only
if, $\lambda+1$ is a non-trivial $2$-power;
cf. \cite[Cor.~1.11]{Stothers} or
\cite[Proposition~6]{MuCAFGVFG}. Using this fact together with
(\ref{Eq:N3General2Part}), one finds that, for $k<e<\frac{3k}{2}$, 
\begin{multline}
\label{Eq:N3kGeneralParity}
\frac{M_3(3k;e,k)}{(k-1)!}\equiv
1\ (\mathrm{mod}\ 2)\
\text { if, and only if, }\\
e=2^{\sigma+1} + 2^\sigma -2\mbox{ and
}k=2^{\sigma+1}-1\mbox{ for some }\sigma\geq1. 
\end{multline}
Indeed, since $2e<3k$, we have $\mathfrak{s}_2(3k-2e)\geq1$; hence,
all summands on the right-hand side of (\ref{Eq:N3General2Part}) are
non-negative. For this right-hand side to vanish, it is thus necessary
and sufficient that (i) $k$ is odd, (ii) $e-k+1=2^\sigma$ for some
$\sigma\geq1$, and (iii) $3k-2e=2^\lambda$ for some $\lambda\geq0$. It
follows from (ii) and (iii) that $k=2^{\sigma+1} + 2^\lambda-2$, which
is odd only for $\lambda=0$, whence (\ref{Eq:N3kGeneralParity}).

The parity behaviour of $f_3(6\lambda)$ also gives that
\begin{equation}
\label{Eq:N3FreeCaseParity}
\frac{M_3(6\ell;3\ell,2\ell)}{(2\ell-1)!}\equiv
1\ (\mathrm{mod}\ 2)\
\text { if, and only if, }\
\ell+1=2^\sigma \mbox{ for some }
\sigma\geq1. 
\end{equation}
From (\ref{Eq:N3SumDecomp})--(\ref{Eq:N3k,k}),
(\ref{Eq:N3kGeneralParity}), and (\ref{Eq:N3FreeCaseParity}), we
deduce that, modulo~$2$, 
\begin{align*}
M_3(3k) &= \frac{3(2k)!}{(k-1)!\, (k+2)!}\,+\,3\cdot
2^{2k-2}\\
&\kern3cm
+\,
\sum_{k<e<\frac{3k}{2}}\frac{M_3(3k;e,k)}{3^{k-1}\,(k-1)!}\,+\,
\left.\begin{cases}
                \frac{\displaystyle
M_3\Big(3k;\frac{3k}{2},k\Big)}{\displaystyle
3^{k-1}\,(k-1)!},&2\mid k\\[3mm]
                0,& 2\nmid  k
                \end{cases}\right\}\\[3mm]
&\equiv \binom{2k+1}{k-1}\,\,+\,\,\left.\begin{cases}
                                 1;& k=2^\tau-1,\,\tau\geq2\\[1mm]
                                 0;& \mbox{otherwise}
                                 \end{cases}\right\}\,\,+\,\,
\left.\begin{cases}
   1;& k=2(2^\sigma-1),\, \sigma\geq1\\[1mm]
   0;&\mbox{otherwise,}
   \end{cases}\right\}
\end{align*}
provided that $k\geq2$. By Lemma~\ref{Lem:BinomParity} with
$\lambda=1$, we have 
\[
\binom{2k+1}{k-1}\equiv 1\ (\mathrm{mod}\ 2)\
\text { if, and only if, }\
  k=2^\sigma-1
\mbox{ or } k=2(2^\sigma-1)\mbox{ for some } \sigma\geq1. 
\]
Combining the last two assertions, we find that 
\[
M_3(3k)\equiv 0\ \mathrm{mod}\ 2,\quad k\geq2.
\]
Finally,
\[
M_3(3) = M_3(3;0,1) + M_3(3;1,1) =1 + 3 \equiv 0\ \mathrm{mod}\ 2,
\]
and the proof is complete.
\end{proof}
Putting together the observation that $M_q(n)=0$ whenever $q\nmid n$
(see the start of this section) with Propositions~\ref{Prop:Nq2Part} and
\ref{Prop:N3Parity}, we obtain the main result of this section.

\begin{theorem}
\label{Thm:NqnParity}
For each odd prime $q$ and every integer $n\geq1,$ the number
$M_q(n)$ of index $n$ subgroups $\Delta$ in $\mathfrak{H}(q)$ with the
property that $\mu(\Delta)=0,$ is even. 
\end{theorem}

\section{Generalized parity patterns of Hecke groups}
\label{Sec:GenParHecke}

The main result of this section computes the parity behaviour of the
generalized subgroup numbers of $\mathfrak H(q)$, with $q$ an odd
prime, in the case when $H$ is of even order; see
Theorem~\ref{Thm:SolvSmod2}. Our interest in this
computation stems from the fact that, in this setting, 
$s_q^H(n):=s_{\mathfrak H(q)}^H(n)\equiv
N_q(n)\ (\mathrm{mod}\ 2)$, where the $N_q(n)$'s 
enter into the mod~$2$ calculation of
$s_{\Ga_m(q)}^H(n)$ via Proposition~\ref{Prop:Reduction1}. Indeed, writing
\[
\mathfrak{U} \cong C_2^{\ast\lambda(\mathfrak{U})} \ast
C_q^{\ast\mu(\mathfrak{U})} \ast F_{\nu(\mathfrak{U})} 
\]
for a subgroup $\mathfrak{U}$ in $\mathfrak{H}(q)$,
in accordance with Kurosh's subgroup theorem, we have 
$$
s_q^H(n) = \sum_{(\mathfrak{H}(q):\mathfrak{U})=n}
\vert\Hom(C_2,H)\vert^{\lambda(\mathfrak{U})}\cdot
\vert\Hom(C_q,H)\vert^{\mu(\mathfrak{U})}\cdot \vert
H\vert^{\nu(\mathfrak{U})}.
$$
By Frobenius' theorem (cf.\ \cite{Frob2} or
\cite[Theorem~9.1.2]{MHall}) concerning the equation $x^n=1$ in a
finite group, the assumption 
$\vert H\vert\equiv0\ (\mathrm{mod}\ 2)$ implies that
$\vert\Hom(C_2,H)\vert$ is even. Furthermore, we have
\begin{equation} \label{eq:Cqung}
\vert\Hom(C_q,H)\vert
  = 1 + (q-1)\times\mbox{(number of subgroups 
$U\cong C_q$ in $H$)}\equiv 1\ \mathrm{mod}\ 2.
\end{equation}
Hence, if $\vert H\vert$ is even, it follows that, modulo~2,
\begin{equation}
s^H_q(n)\equiv
\underset{\lambda(\mathfrak{U})=\nu(\mathfrak{U})=0}
{\underset{(\mathfrak{H}(q):\mathfrak{U})=n}{\sum_\mathfrak{U}}}
1 = N_q(n),
\label{eq:s=N}
\end{equation}
as claimed.

On the other hand,
it was shown in particular in \cite{MRepres} that --- just as subgroup
numbers are connected with the enumeration of permutation
representations --- the generalized subgroup numbers $s_\Gamma^H(n)$ of
a finitely generated group $\Gamma$ are related to the function
$\vert\Hom(\Gamma,H\wr S_n)\vert$ counting monomial representations of
$\Gamma$ via the identity 
\begin{equation}
\label{Eq:MonRepsGenFunc}
\sum_{n=0}^\infty 
\vert\Hom(\Gamma, H\wr S_n)\vert \frac{z^n}{\vert H\vert^n n!} =
\exp\bigg(\frac{1}{\vert H\vert}\sum_{n=1}^\infty s_\Gamma^H(n)
\frac{z^n}{n}\bigg) ,
\end{equation}
or, what comes to the same thing, via the recurrence relation
\begin{equation}
\label{Eq:MonRepsRec}
n\,\vert H\vert \,h_\Gamma^H(n) = \sum_{k\geq1} s_\Gamma^H(k)
h_\Gamma^H(n-k),\quad n\in\Z, 
\end{equation}
where
\[
h_\Gamma^H(n):= \begin{cases}
                    \vert\Hom(\Gamma,H\wr S_n)\vert/(\vert H\vert^n
n!),&n\geq0\\[2mm]
                    0,& n<0.
                    \end{cases}
\]
This follows from \cite[Cor.~1]{MRepres} by setting
$\Sigma=\emptyset$, $\Lambda=\N$, $M=\N_0$, and replacing the variable
$z$ by $z/\vert H\vert$. For a (relatively) untechnical account of the
theory of generalized permutation representations, which gives rise
(among other things) to Formula (\ref{Eq:MonRepsGenFunc}), the reader
may consult the survey papers \cite{MKorea} and
\cite{MBrasilia}. Introducing the series 
\[
\mathcal{H}_\Gamma^H(z):= \sum_{n=0}^\infty \vert\Hom(\Gamma, H\wr
S_n)\vert \frac{z^n}{\vert H\vert^n n!} 
\]
and
\[
\mathcal{S}_\Gamma^H(z):= \sum_{n=0}^\infty s_\Gamma^H(n+1) z^n,
\]
Identity (\ref{Eq:MonRepsGenFunc}) takes the form
\begin{equation}
\label{Eq:WreathTransform}
\mathcal{H}_\Gamma^H(z) = \exp\bigg(\frac{1}{\vert H\vert} \int
\mathcal{S}_\Gamma^H(z)\,dz\bigg). 
\end{equation}

In view of the preceding observations, our next goal will be to derive
a linear differential equation with polynomial coefficients for the
generating function $\mathcal H_{\mathfrak H(q)}^H(z)$. Using 
Relation \eqref{Eq:WreathTransform} together with the Fa\`a di Bruno
formula (see \eqref{Eq:Bell} below), this linear differential equation will
give rise to a Riccati-type differental equation for the generating
function $\mathcal S_{\mathfrak H(q)}^H(z)$. 
It is the latter differential equation that we are actually
interested in; see \eqref{Eq:PrincEq}.

\subsection{A system of linear differential equations for generating
functions related to $\mathcal H_{\mathfrak H(q)}^H(z)$}
\label{sec:5.1}

The starting point for our computations concerning the series
$\mathcal H_{\mathfrak H(q)}^H(z)$ is the observation that, by the
universal mapping property of free products,
\begin{equation} \label{eq:Hom=HomHom}
\vert\Hom(\mathfrak H(q),H\wr S_n)\vert=
\vert\Hom(C_2,H\wr S_n)\vert\cdot \vert\Hom(C_q,H\wr S_n)\vert.
\end{equation}
Setting $h:=\vert H\vert$,
$a:=\vert\Hom(C_2,H)\vert$, $b:= \vert\Hom(C_q,H)\vert$, 
\[
\alpha_n:=\begin{cases}
              \vert\Hom(C_2,H\wr S_n)\vert,&n\geq0\\[2mm]
              0,& n<0,
              \end{cases}
\]
and
\[
\beta_n:=\begin{cases}
              \vert\Hom(C_q,H\wr S_n)\vert,&n\geq0\\[2mm]
              0,& n<0.
              \end{cases}
\]
Equation (\ref{Eq:MonRepsRec}) specializes to the relations
\begin{equation}
\label{Eq:2WreathRec}
\alpha_{n+1} = a\alpha_n + hn\alpha_{n-1},\quad n\neq-1
\end{equation}
and
\begin{equation}
\label{Eq:qWreathRec}
\beta_{n+1} = b\beta_n + h^{q-1}
\frac {n!} {(n-q+1)!}
\beta_{n-q+1},\quad n\neq-1, 
\end{equation}
where we use the same convention concerning $\frac{1}{n!}$ as in 
Formula~(\ref{Eq:In}).
For $n,k\in\Z$, let 
\[
A_k(n):= \frac{\alpha_n\,\beta_{n-k}}{h^{n-k}\, (n-k)!},            
\]
still making use of the same convention concerning $\frac{1}{n!}$, 
and set
\[
F_k(z):= \sum_{n\in\Z} A_k(n) z^n.
\]
The reader should note that, in view of \eqref{eq:Hom=HomHom},
we have $F_0(z)=\mathcal H_{\mathfrak H(q)}^H(z)$.

Multiplying (\ref{Eq:2WreathRec}) by $\beta_{n-k}/(h^{n-k}(n-k)!)$ and
(\ref{Eq:qWreathRec}) by $\alpha_{n+k}/(h^nn!)$, we obtain the
relations 
\begin{equation}
\label{Eq:ARelI}
A_{k+1}(n+1) = a A_k(n) + hnA_{k-1}(n-1),\quad n,k\in\Z,\,k\geq0,
\end{equation}
respectively
\begin{equation}
\label{Eq:ARelII}
h(n+1)A_{k-1}(n+k) = b A_k(n+k) + A_{k+q-1}(n+k),\quad n,k\in\Z.
\end{equation}
Multiplication by $z^{n+1}$ and summation over $n\in\Z$, transforms Equation
(\ref{Eq:ARelI}) into the relation 
\begin{equation}
\label{Eq:FRelI}
F_{k+1}(z) = azF_k(z) + hz^2F_{k-1}(z) + hz^3F_{k-1}'(z),\quad k\geq0,
\end{equation}
while multiplication by $z^{n+k}$ and summation over $n\in\Z$ turns
(\ref{Eq:ARelII}) into 
\begin{equation}
\label{Eq:FRelII}
hzF_{k-1}'(z) = h(k-1)F_{k-1}(z) + bF_k(z) + F_{k+q-1}(z),\quad k\in\Z. 
\end{equation}
Clearly, by iterating Relation (\ref{Eq:FRelI}), we can express every
function $F_k$ with $k\geq0$ in terms of derivatives of $F_0$ and $F_1$ alone; more
precisely, define integral coefficient systems $\big(c_k^{(\mu)}\big)$
and $\big(d_k^{(\nu)}\big)$ for $k\geq0$ and
$0\leq\mu\leq\big\lfloor\frac{k}{2}\big\rfloor$, respectively $k\geq1$
and $0\leq\nu\leq\big\lfloor\frac{k-1}{2}\big\rfloor$, via 
\begin{align*}
c_0^{(0)} &= 1,\, c_1^{(0)} = 0,\, c_2^{(0)} = h,\, c_2^{(1)} = 1,\\
d_1^{(0)} &= 1,\, d_2^{(0)} = a,
\end{align*}
\[
c_{k+1}^{(\mu)} = \left.\begin{cases}
                     a c_k^{(0)} + hkc_{k-1}^{(0)},&\mu=0\\[2mm]
                     ac_k^{(\mu)} + h(k+\mu)c_{k-1}^{(\mu)} +
c_{k-1}^{(\mu-1)},&
1\leq\mu\leq\big\lfloor\frac{k-1}{2}\big\rfloor\\[2mm]
                     c_{k-1}^{(\frac{k-1}{2})},&
\mu=\big\lfloor\frac{k+1}{2}\big\rfloor, k\mbox{ odd}\\[2mm]
                     ac_k^{(k/2)} + c_{k-1}^{(\frac{k-2}{2})},&
\mu=\big\lfloor\frac{k+1}{2}\big\rfloor, k\mbox{ even}
                     \end{cases}\right\}\quad(k\geq2),
\]
and
\[
d_{k+1}^{(\nu)} = \left.\begin{cases}
                     a d_k^{(0)} + h(k-1)d_{k-1}^{(0)},&\nu=0\\[2mm]
                     ad_k^{(\nu)} + h(k+\nu-1)d_{k-1}^{(\nu)} +
d_{k-1}^{(\nu-1)},&
1\leq\nu\leq\big\lfloor\frac{k-2}{2}\big\rfloor\\[2mm]
                     ad_k^{(\frac{k-1}{2})} +
d_{k-1}^{(\frac{k-3}{2})},&\nu=\big\lfloor\frac{k}{2}\big\rfloor,
k\mbox{ odd}\\[2mm]
                     d_{k-1}^{(\frac{k-2}{2})},&
\nu=\big\lfloor\frac{k}{2}\big\rfloor, k\mbox{ even}
                     \end{cases}\right\}\quad(k\geq2).
\]
Then we have the following.
\begin{lemma}
\label{Lem:FkRed}
With $\big(c_k^{(\mu)}\big)$ and $\big(d_k^{(\nu)}\big)$ as above,
\begin{equation}
\label{Eq:FkRed}
F_k(z) = \sum_{\mu=0}^{\lfloor\frac{k}{2}\rfloor} c_k^{(\mu)} h^\mu
z^{k+\mu} F_0^{(\mu)}(z)\, +
\,\sum_{\nu=0}^{\lfloor\frac{k-1}{2}\rfloor} d_k^{(\nu)} h^\nu
z^{k+\nu-1} F_1^{(\nu)}(z),\quad k\geq0. 
\end{equation}
\end{lemma}

The proof of Lemma~\ref{Lem:FkRed}, which, again, is straightforward 
but somewhat technical, will be given in Section~\ref{Subsec:Lem18} in
the appendix. 

For later usage, we record evaluations of $c_k^{(\mu)}$
and $d_k^{(\nu)}$ modulo~$2$ as given by the following two lemmas.

\begin{lemma}
\label{Lem:CoeffMod2}
For $h\equiv0\ (\mathrm{mod}\ 2),$ we have
\begin{align}
c_k^{(\mu)}&\equiv\delta_{2\mu,k}\ \mathrm{mod}\ 2,\quad
\mu\geq0,\,k\geq2\mu,\label{Eq:cmod2}\\[2mm] 
d_k^{(\nu)}&\equiv\delta_{2\nu+1,k}\ \mathrm{mod}\ 2,\quad\nu\geq0,\,k\geq2\nu+1,
\label{Eq:dmod2}
\end{align}
where $\delta_{s,t}$ is the Kronecker delta.
\end{lemma}
\begin{proof}
As we observed at the start of this section, 
the assumption $h\equiv0\bmod{2}$ implies that
$a\equiv0\bmod{2}$.
Hence, the definition of $c_k^{(\mu)}$ simplifies
modulo~$2$ to 
\begin{align*}
c_0^{(0)}&\equiv c_2^{(1)}\equiv 1\\[2mm]
c_1^{(0)} &\equiv c_2^{(0)}\equiv 0\\[2mm]
c_k^{(\mu)} &\equiv \left.\begin{cases}
                        0,& \mu=0\\[2mm]
                        c_{k-2}^{(\mu-1)},&\mu\geq1
                        \end{cases}\right\}\quad(k\geq3).
\end{align*} 
We now argue by induction on $k$. For $k\leq2$, Formula
(\ref{Eq:cmod2}) clearly holds. Supposing that (\ref{Eq:cmod2}) holds
true for $k<K$ with some $K\geq3$, the inductive hypothesis gives 
\[
c_K^{(\mu)} \equiv\left\{ \begin{matrix}
                       0,\hfill&\mu=0\hfill\\[2mm]
                       \delta_{2(\mu-1),K-2},\hfill&\mu\geq1\hfill
                       \end{matrix}\right\}\, =\, \delta_{2\mu,K}
\ \mathrm{mod}\ 2,\quad\mu\leq\Big\lfloor\frac{K}{2}\Big\rfloor, 
\]
as required. The proof of Formula~(\ref{Eq:dmod2}) is similar.
\end{proof}

\begin{lemma}
\label{Lem:CoeffMod2hunger}
For $h\equiv1\ (\mathrm{mod}\ 2),$ we have
\begin{align}
c_k^{(\mu)}&\equiv
\left\{\begin{matrix}
\delta_{2\mu,k}+\delta_{2\mu+1,k};\hfill&\text{if }\mu
\text{ is odd}\hfill\\
\delta_{2\mu,k}+\delta_{2\mu+2,k}+\delta_{2\mu+3,k};\hfill&
\text{if }\mu\text{ is even}\hfill
\end{matrix}\right\}\ \mathrm{mod}\ 2,\quad
\mu\geq0,\,k\geq2\mu,\label{Eq:cmod2hunger}\\[2mm]
d_k^{(\nu)}&\equiv
\left\{\begin{matrix}
\delta_{2\mu+1,k};\hfill&\text{if }\mu\text{ is odd}\hfill\\
\delta_{2\mu+1,k}+\delta_{2\mu+2,k};\hfill&
\text{if }\mu\text{ is even}\hfill
\end{matrix}\right\}\ \mathrm{mod}\ 2,\quad\nu\geq0,\,k\geq2\nu+1,
\label{Eq:dmod2hunger}
\end{align}
where, again, $\delta_{s,t}$ is the Kronecker delta.
\end{lemma}

\begin{proof}
The assumption $h\equiv1\ (\mathrm{mod}\ 2)$ implies that
$a=1$. The rest of the 
proof is similar to the proof of the preceding lemma,
and is omitted.
\end{proof}

Multiplying (\ref{Eq:FRelII}) for $k=0$ by $z^2$, substituting
(\ref{Eq:FRelI}) with $k=0$, and computing $F_{q-1}(z)$ by means of
Lemma~\ref{Lem:FkRed}, we find that 
\begin{multline}
\label{Eq:LSE0}
(d_{q-1}^{(0)}z^q-1) F_1(z)\,+\,\sum_{\nu=1}^{\frac{q-3}{2}}
d_{q-1}^{(\nu)} h^\nu z^{q+\nu} F_1^{(\nu)}(z) =
-(az+bz^2+c_{q-1}^{(0)}z^{q+1})f(z)\\[2mm] 
-\,\sum_{\mu=1}^{\frac{q-1}{2}} c_{q-1}^{(\mu)} h^\mu z^{q+\mu+1}
f^{(\mu)}(z) 
\end{multline}
with $f(z):=F_0(z)$. Next, taking $k=1$ in (\ref{Eq:FRelII}) and
computing $F_q(z)$ by means of Lemma~\ref{Lem:FkRed}, we obtain the
equation 
\begin{multline}
\label{Eq:LSE1}
(d_q^{(0)}z^{q-1}+b)F_1(z)\,+\,\sum_{\nu=1}^{\frac{q-1}{2}}
d_q^{(\nu)} h^\nu z^{q+\nu-1} F_1^{(\nu)}(z) = -c_q^{(0)}z^qf(z) -
h(c_q^{(1)}z^{q+1}-z)f'(z)\\[2mm] - \sum_{\mu=2}^{\frac{q-1}{2}}
c_q^{(\mu)} h^\mu z^{q+\mu} f^{(\mu)}(z). 
\end{multline}
Similarly, combining Equation~(\ref{Eq:FRelII}) for $k=2,3,\ldots,q-2$
with Lemma~\ref{Lem:FkRed}, we obtain further equations, which
together with (\ref{Eq:LSE0}) and (\ref{Eq:LSE1}) form a system of
$q-1$ linear equations over the field $\Q((z))$ in the variables
$F_1^{(0)}(z)$, $F_1^{(1)}(z)$, \ldots, $F_1^{(q-2)}(z)$. More
specifically, we find that 
\begin{equation}
\label{Eq:LSE}
\bar{\Delta}_q \begin{pmatrix}
         F_1^{(0)}(z)\\
         \vdots\\
         F_1^{(q-2)}(z)
         \end{pmatrix} = \begin{pmatrix}
                           b_0(z)\\
                           \vdots\\
                           b_{q-2}(z)
                           \end{pmatrix},
\end{equation}
where
$\bar{\Delta}_q=(h^\lambda\omega_{\kappa,\lambda})_{0\leq\kappa,\lambda\leq
q-2}$, 
%Christian:
with $\omega_{\kappa,\lambda}$ and $b_{\kappa}(z)$ given explicitly
in terms of $b,h$, the $c^{(\nu)}_k$'s, the $d^{(\nu)}_k$'s, 
and the derivatives of $f(z)$, as described in Section~\ref{app:db}
in the appendix.

\subsection{The determinant of the system, and its minors}
\label{sec:5.2}
Denote by ${}_0\bar{\Delta}_q$ the matrix obtained from
$\bar{\Delta}_q$ by replacing the $0$-th column with the right-hand
side of (\ref{Eq:LSE}), and by ${}_0\Delta_q$ the corresponding matrix
obtained from $\Delta_q:=(\omega_{\kappa,\lambda})$ via the same
operation. Then
$\det\bar{\Delta}_q=h^{\binom{q-1}{2}}\det\Delta_q$,
$\det{}_0\bar{\Delta}_q=h^{\binom{q-1}{2}}\det{}_0\Delta_q$,
and, according to Cramer's Rule and Laplace's expansion theorem, 
\[
F_1(z) = \frac{\det{}_0\Delta_q}{\det\Delta_q} =
(\det\Delta_q)^{-1}\sum_{\kappa=0}^{q-2}(-1)^\kappa
(\det\Delta_{\kappa,0}) b_\kappa(z), 
\] 
where $\Delta_{\kappa,0}$ is the matrix obtained from $\Delta_q$ by
deleting the $\kappa$-th row and $0$-th column. We are assuming here,
of course, that $\det\Delta_q\neq0$, a fact which follows, among
other things, from our next two auxiliary results, 
whose proofs will be given in Sections~\ref{Subsec:Lem21} and 
\ref{Subsec:Lem22} in the appendix, respectively. 

\begin{lemma}
\label{Lem:DetMod2}
For $h\equiv0\ (\mathrm{mod}\ 2),$ we have
\begin{enumerate}
\item[(i)] $\det\Delta_q \equiv \det\Delta_{0,0} \equiv
z^{\frac{3q^2-11q+12}{2}}\ \mathrm{mod}\ 2,$
\vskip2mm

\item[(ii)] $\det\Delta_{\kappa,0} \equiv
0\ \mathrm{mod}\ 2,\quad1\leq\kappa\leq q-2.$ 
\end{enumerate}
\end{lemma}
\begin{lemma}
\label{Lem:DetMod2hunger}
For $h\equiv1\ (\mathrm{mod}\ 2),$ we have
\begin{enumerate}
\item[(i)] 
$\displaystyle
\det\Delta_q \equiv\left\{\begin{matrix}
z^{\frac{3q^2-11q+12}{2}};\hfill&\text{if }q\equiv
1\ (\mathrm{mod}\ 4)\\[2mm]
z^{\frac{3q^2-11q+12}{2}}+z^{\frac{3q^2-8q+9}{2}};\hfill&\text{if }q\equiv
3\ (\mathrm{mod}\ 4)
\end{matrix}\right\}
\ \mathrm{mod}\ 2;$
\vskip2mm

\item[(ii)] 
$\det\Delta_{0,0} \equiv
z^{\frac{3q^2-11q+12}{2}}\ (\mathrm{mod}\ 2),$
$\det\Delta_{1,0} \equiv
0\ (\mathrm{mod}\ 2),$\newline
and for $2\le \kappa\le q-2$,
$$\det\Delta_{\kappa,0} \equiv
\left\{\begin{matrix}
0;\hfill&\text{if }\kappa\ge2\text{ and }q\equiv
1\ (\mathrm{mod}\ 4)\hfill\\[2mm]
0;\hfill&\text{if }\kappa\equiv1\ (\mathrm{mod}\ 4)\text{ and }q\equiv
3\ (\mathrm{mod}\ 4)\hfill\\[2mm]
z^{\frac{3q^2-8q+5-12\kappa'}{2}};\hfill&
\text{if }\kappa=4\kappa'+2\text{ and }q\equiv
3\ (\mathrm{mod}\ 4)\hfill\\[2mm]
z^{\frac{3q^2-8q+3-12\kappa'}{2}};\hfill&\text{if }\kappa=4\kappa'+3\text{
and }q\equiv
3\ (\mathrm{mod}\ 4)\hfill\\[2mm]
z^{\frac{3q^2-8q+1-12\kappa'}{2}};\hfill&\text{if }\kappa=4\kappa'+4\text{
and }q\equiv
3\ (\mathrm{mod}\ 4)\hfill\\[2mm]
\end{matrix}\right\}\ \mathrm{mod}\ 2.$$
\end{enumerate}
\end{lemma}

\subsection{A functional equation for $\mathcal S_{\mathfrak H(q)}^H(z)$}
\label{sec:5.3}

The following auxiliary result, 
whose proof is recorded in Section~\ref{Subsec:Lem23} in the appendix, 
is concerned with computing
the expressions 
\begin{equation}
\label{Eq:GenLogDer}
h^\nu \Big(\frac{d^\nu}{dz^\nu}
\mathcal{H}_\Gamma^H(z)\Big)\Big/\mathcal{H}_\Gamma^H(z),\quad\nu\in\N_0
\end{equation}
over $\Z$ and modulo~$2$ in terms of the series $\mathcal{S}_\Gamma^H(z)$.
(The reader should recall \eqref{Eq:WreathTransform}.)
\begin{lemma}
\label{Lem:GenLogDerComp}
Let $\Gamma$ be a finitely generated group, $H$ a finite group,
and let $h=\vert H\vert$. 
\begin{enumerate}
\item For each $\nu\in\N_0,$ we have
\begin{multline}
\label{Eq:HigherDerHComp}
h^\nu \Big(\frac{d^\nu}{dz^\nu}
\mathcal{H}_\Gamma^H(z)\Big)\Big/\mathcal{H}_\Gamma^H(z) \\=
\underset{{\pi_1+2\pi_2+\dots+\nu \pi_\nu=\nu}}
{\sum_{\pi_1,\dots,\pi_\nu\ge0}} h^{\nu-\pi_1-\dots-\pi_\nu}
\,\frac{\nu!}{\prod_{j\geq1}(j!)^{\pi_j} \pi_j!}
\,\prod_{j\geq1}\Big(\big(\mathcal{S}_\Gamma^H(z)\big)^{(j-1)}\Big)^{\pi_j},
\end{multline}
and the coefficients ${\nu!}\big/\big({\prod_{j\geq1}(j!)^{\pi_j} \pi_j!}\big)$
are integers.
\vskip2mm

\item For each $\nu\in\N_0,$ the series {\em (\ref{Eq:GenLogDer})} is
an integral power series in $z,$ and satisfies the congruence 
\begin{equation}
\label{Eq:HigherDerHCompMod2}
h^\nu \Big(\frac{d^\nu}{dz^\nu}
\mathcal{H}_\Gamma^H(z)\Big)\Big/\mathcal{H}_\Gamma^H(z) \equiv
\sum_{\mu=0}^{\lfloor\frac{\nu}{2}\rfloor} h^\mu \binom{\nu}{2\mu}
\Big(\big(\mathcal{S}_\Gamma^H(z)\big)'\Big)^\mu
\big(\mathcal{S}_\Gamma^H(z)\big)^{\nu-2\mu}\bmod{2}. 
\end{equation}
In particular, if $h\equiv0\ (\mathrm{mod}\ 2),$ then
\begin{equation}
\label{Eq:HigherDerHCompMod2Heven}
h^\nu \Big(\frac{d^\nu}{dz^\nu}
\mathcal{H}_\Gamma^H(z)\Big)\Big/\mathcal{H}_\Gamma^H(z) \equiv
\big(\mathcal{S}_\Gamma^H(z)\big)^\nu\ \mathrm{mod}\ 2,\quad\nu\geq0. 
\end{equation}
\end{enumerate}
\end{lemma}

Write
\[
\big((\det\Delta_q)^{-1}\big)^{(\nu)} =
\frac{\Theta_\nu(z)}{(\det\Delta_q)^{\nu+1}},\quad\nu\geq0, 
\]
where $\{\Theta_\nu(z)\}_{\nu\geq0}$ is the family of polynomials in
$z$ determined recursively via 
\begin{align*}
\Theta_{\nu+1}(z) &= (\det\Delta_q) \Theta_\nu'(z) - (\nu+1)
(\det\Delta_q)'
\Theta_\nu(z),\quad\nu\geq0\\[2mm] 
\Theta_0(z) &= 1.
\end{align*}
Then, by Leibniz's formula for the derivatives of a product function, we have
\begin{equation}
\label{Eq:F1Ders}
F_1^{(\nu)}(z) = \sum_{\lambda=0}^\nu \binom{\nu}{\lambda}\,
\frac{\Theta_{\nu-\lambda}(z)}{(\det\Delta_q)^{\nu-\lambda+1}}
\,(\det{}_0\Delta_q)^{(\lambda)},\quad\nu\geq0. 
\end{equation}
Inserting (\ref{Eq:F1Ders}) into (\ref{Eq:LSE1}), multiplying
throughout by $(\det\Delta_q)^{{\frac{q+1}{2}}}$, and dividing
by the series $f(z)$, we obtain the relation 
\begin{multline}
\label{Eq:PrincEq}
c_q^{(0)} z^q (\det\Delta_q)^{{\frac{q+1}{2}}} +
z(c_q^{(1)}z^q-1)(\det\Delta_q)^{{\frac{q+1}{2}}}hf'(z)/f(z) \\
+
\sum_{\mu=2}^{\frac{q-1}{2}} c_q^{(\mu)} z^{q+\mu}
(\det\Delta_q)^{{\frac{q+1}{2}}} h^\mu f^{(\mu)}(z)/f(z)\\
+
(d_q^{(0)}z^{q-1}+b)(\det\Delta_q)^{{\frac{q-1}{2}}}
(\det{}_0\Delta_q)/f(z)\\
+ \sum_{\nu=1}^{\frac{q-1}{2}} \sum_{\lambda=0}^\nu
\binom{\nu}{\lambda} d_q^{(\nu)} z^{q+\nu-1}
\Theta_{\nu-\lambda}(z)(\det\Delta_q)^{{\frac{q+1}{2}}-(\nu-\lambda+1)}
h^\nu (\det{}_0\Delta_q)^{(\lambda)}/f(z) = 0, 
\end{multline}
where the reader should recall that $f(z)=F_0(z)=\mathcal
H_{\mathfrak H(q)}^H(z)$.
Two kinds of results may be derived from the last identity: first, an
integral recurrence relation for the function
$s_{\mathfrak{H}(q)}^H(n)$ for each given $q$ in terms of the
parameters $h=\vert H\vert$, $a=\vert\Hom(C_2,H)\vert$, and
$b=\vert\Hom(C_q,H)\vert$; second, a description of the parity
behaviour of $s_{\mathfrak{H}(q)}^H(n)$ for all $q$. 
We illustrate the first type of result below with the case of the
modular group $\mathfrak{H}(3)$. As concerns the second kind of result, 
we note that $s_{\mathfrak{H}(q)}^H(n) \equiv s_q(n)\ (\mathrm{mod}\ 2)$ 
for $\vert H\vert$ odd. Since the parity behaviour of $s_q(n)$ is known 
(see Theorem~\ref{Thm:sqn}), we shall concentrate here on the case where 
$\vert H\vert$ is even, which yields the parity of the numbers $N_q(n)$.

\subsection{The case of $\mathfrak H(3)$}
\label{sec:5.4}

For $q=3$, Eq.~(\ref{Eq:PrincEq}) reduces to 
\begin{multline}
\label{Eq:PrincEQq=3}
c_3^{(0)} z^3 (\det\Delta_3)^2 +
z(c_3^{(1)}z^3-1)(\det\Delta_3)^2 h
f'(z)/f(z)\\
+\big\{(d_3^{(0)}z^2+b)(\det\Delta_3)+d_3^{(1)}h
z^3\Theta_1(z)\big\}(\det{}_0\Delta_3)/f(z)\\[2mm] +
d_3^{(1)}z^3(\det\Delta_3) h(\det{}_0\Delta_3)'/f(z) = 0. 
\end{multline}
We have
\begin{align*}
c_3^{(0)} &= ac_2^{(0)}+2hc_1^{(0)} = ah,\\[2mm]
c_3^{(1)} &= ac_2^{(1)}+c_1^{(0)} = a,\\[2mm]
d_3^{(0)} &= ad_2^{(0)}+hd_1^{(0)} = a^2+h,\\[2mm]
d_3^{(1)} &= d_1^{(0)} = 1,\\[2mm]
\det\Delta_3 &= \det\begin{pmatrix}
                         az^3-1 & 0\\
                         (a^2+h)z^2+b & z^3
                         \end{pmatrix} = z^3(az^3-1),\\[2mm]
\Theta_1(z) &= -(\det\Delta_3)' = -3z^2(2az^3-1), 
\end{align*}
and
\begin{align*}
\det{}_0\Delta_3& = \det\begin{pmatrix}
                            -z(a+bz+hz^3)f(z)-hz^5f'(z) & 0\\
                            -ahz^3f(z)-hz(az^3-1)f'(z) & z^3
                            \end{pmatrix}\\
& = -z^4(a+bz+hz^3)f(z)-hz^8f'(z). 
\end{align*}
Inserting these values into (\ref{Eq:PrincEQq=3}), dividing both sides
by $z^7$, and collecting terms, we find that 
\begin{multline}
\label{Eq:ModularHomDiffEq}
(z^7-az^{10}) h^2f''(z)/f(z) + \big(-1+4az^3+2bz^4+(7h-3a^2)z^6 -
2abz^7-4ahz^9\big) h f'(z)/f(z)\\[2mm] +
ab+b^2z+a(a^2+3h)z^2+4bhz^3-ab^2z^4+(5h^2-a^4)z^5-ab(a^2+h)z^6-2ah^2z^8
= 0. 
\end{multline}
Substituting
\[
S_3^H(z) = hf'(z)/f(z)
\]
and 
\[
h(S_3^H(z))' + (S_3^H(z))^2 = h^2f''(z)/f(z)
\]
in (\ref{Eq:ModularHomDiffEq}) according to
Lemma~\ref{Lem:GenLogDerComp}(i), we find the Riccati-type
differential equation 
\begin{multline}
\label{Eq:ModularGenSubDiffEq}
\big(z^7-az^{10}\big)\big(h(\mathcal{S}_3^H(z))'+(\mathcal{S}_3^H(z))^2\big)
\\[2mm]
+ \big(-1+4az^3+2bz^4+(7h-3a^2)z^6 - 2abz^7-4ahz^9\big)
\mathcal{S}_3^H(z)\\[2mm] +
ab+b^2z+a(a^2+3h)z^2+4bhz^3-ab^2z^4+(5h^2-a^4)z^5-ab(a^2+h)z^6-2ah^2z^8
= 0 
\end{multline}
for the generating function
$\mathcal{S}_3^H(z):=\mathcal{S}_{\mathfrak{H}(3)}^H(z)$. Comparing
coefficients in (\ref{Eq:ModularGenSubDiffEq}), we finally obtain the
following result, which generalizes Theorem~1 in \cite{GIR}. 
\begin{proposition}
\label{Prop:ModularHRec}
The function $s_3^H(n):=s_{\mathfrak{H}(3)}^H(n)$ satisfies the
recurrence relation 
\begin{multline}
\label{Eq:s3HRec}
s_3^H(n) = 4a s_3^H(n-3)\,+\,2b
s_3^H(n-4)\,+\,(hn-3a^2)s_3^H(n-6)\,-\,2ab s_3^H(n-7)\\[2mm]
-\,ah(n-6)s_3^H(n-9)\,+\,\sum_{\mu=1}^{n-7} s_3^H(\mu)
s_3^H(n-\mu-6)\,-\,a\sum_{\mu=1}^{n-10} s_3^H(\mu)
s_3^H(n-\mu-9),\quad n\geq10, 
\end{multline}
with initial values
\begin{align*}
s_3^H(1) &= ab,\\
s_3^H(2) &= b^2,\\
s_3^H(3) &= a(a^2+3h),\\
s_3^H(4) &= 4b(a^2+h),\\
s_3^H(5) &= 5ab^2,\\
s_3^H(6) &= 3a^4+2b^3+5h^2+12a^2h,\\
s_3^H(7) &= 14ab(a^2+2h),\\
s_3^H(8) &= 8b^2(3a^2+2h),\\
s_3^H(9) &= 3a(3a^4+6b^3+15h^2+16a^2h).
\end{align*}
\end{proposition}

\subsection{The parity behaviour of $N_q(n)$}
\label{sec:5.5}

We now use
Lemmas~\ref{Lem:CoeffMod2}, \ref{Lem:DetMod2}, and
\ref{Lem:GenLogDerComp}(ii) to simplify identity~(\ref{Eq:PrincEq})
under the assumption that $h=\vert H\vert\equiv0\ (\mathrm{mod}\ 2)$. 
Throughout, we set
$\mathcal{S}_{\mathfrak{H}(q)}^H(z)=:\mathcal{S}_q^H(z)$ and
$\alpha_0:=3q^2-11q+12$.

First, $c_q^{(0)}\equiv c_q^{(1)}\equiv0\ (\mathrm{mod}\ 2)$, since $q>2$; hence,
\[
c_q^{(0)}z^q(\det\Delta_q)^{{\frac{q+1}{2}}} \equiv 0\ \mathrm{mod}\ 2
\]
and
\[
z(c_q^{(1)}z^q-1)(\det\Delta_q)^{{\frac{q+1}{2}}} h f'(z)/f(z)
\equiv z^{{\frac{q+1}{4}}\alpha_0+1} \mathcal{S}_q^H(z)\ \mathrm{mod}\ 2. 
\]
Similarly, we have $c_q^{(\mu)}\equiv0\ (\mathrm{mod}\ 2)$ for $\mu<q/2$; thus
\[
\sum_{\mu=2}^{\frac{q-1}{2}} c_q^{(\mu)} z^{q+\mu}
(\det\Delta_q)^{{\frac{q+1}{2}}} h^\mu f^{(\mu)}(z)/f(z) \equiv
0\ \mathrm{mod}\ 2. 
\]
Next, inspection shows that, in view of
Lemma~\ref{Lem:GenLogDerComp}(ii), $b_\kappa(z)/f(z)$ is always an
integral power series; hence, $(\det{}_0\Delta_q)/f(z)$ is an
integral power series, and modulo~$2$ we have 
\begin{align*}
(\det{}_0\Delta_q)/f(z) &= \sum_{\kappa=0}^{q-2} (-1)^\kappa
(\det\Delta_{\kappa,0}) b_\kappa(z)/f(z)\\[2mm] 
&\equiv (\det\Delta_{0,0}) b_0(z)/f(z)\\[2mm]
&\equiv z^{\alpha_0/2}\big\{z^2 + z^{\frac{3q+1}{2}} h^{\frac{q-1}{2}}
f^{(\frac{q-1}{2})}(z)/f(z)\big\}\\[2mm] 
&\equiv z^{2+\alpha_0/2} + z^{\frac{\alpha_0+3q+1}{2}}
\big(\mathcal{S}_q^H(z)\big)^{\frac{q-1}{2}}, 
\end{align*}
and so, modulo~$2$,
\begin{align*}
(d_q^{(0)}z^{q-1}+b) (\det\Delta_q)^{{\frac{q-1}{2}}}
(\det{}_0\Delta_q)/f(z) &\equiv
(\det\Delta_q)^{{\frac{q-1}{2}}}
(\det{}_0\Delta_q)/f(z)\\[2mm] 
&\kern-10pt
\equiv
z^{\frac{\alpha_0}{2}({\frac{q+1}{2}}-1)}\Big\{z^{2+\alpha_0/2} +
z^{\frac{\alpha_0+3q+1}{2}}
\big(\mathcal{S}_q^H(z)\big)^{\frac{q-1}{2}}\Big\}\\[2mm] 
&\kern-10pt
= z^{2+{\frac{q+1}{4}}\alpha_0} +
z^{({\frac{q+1}{2}}\alpha_0+3q+1)/2}
\big(\mathcal{S}_q^H(z)\big)^{\frac{q-1}{2}}. 
\end{align*}
Here we have used Lemma~\ref{Lem:DetMod2} to evaluate
$\det\Delta_q$ and $\det\Delta_{\kappa,0}$ modulo~$2$, the
facts, following from Lemma~\ref{Lem:CoeffMod2}, that
$d_q^{(0)}\equiv0\ (\mathrm{mod}\ 2)$ and that 
\[
c_{q-1}^{(\mu)} \equiv
1\ (\mathrm{mod}\ 2)\
\text { if, and only if, }\
\mu=\frac{q-1}{2},\quad0\leq\mu\leq\frac{q-1}{2},
\]
plus part~(ii) of Lemma~\ref{Lem:GenLogDerComp} to rewrite the term
$h^{\frac{q-1}{2}} f^{(\frac{q-1}{2})}(z)/f(z)$. This leaves the
double sum $\Sigma$ on the right-hand side of
(\ref{Eq:PrincEq}). Clearly, for $0\leq\kappa\leq q-2$ and
$0\leq\eta\leq\nu$, the term $h^\nu b_\kappa^{(\eta)}(z)/f(z)$ is an
integral power series, and we have 
\[
h^\nu b_\kappa^{(\eta)}(z)/f(z) \equiv 0\ \mathrm{mod}\ 2,\quad \eta<\nu.
\]
From this observation it follows that, for $0\leq\lambda\leq\nu$, the
expression 
\[
h^\nu (\det{}_0\Delta_q)^{(\lambda)}/f(z) = \sum_{\kappa=0}^{q-2}
\sum_{\eta=0}^\lambda (-1)^\kappa\binom{\lambda}{\eta}
(\det\Delta_{\kappa,0})^{(\lambda-\eta)} h^\nu
b_\kappa^{(\eta)}(z)/f(z) 
\]
is an integral power series, and that
\begin{equation}
\label{Eq:SigmaEval1}
h^\nu (\det{}_0\Delta_q)^{(\lambda)}/f(z) \equiv\left\{ \begin{matrix}
z^{\alpha_0/2} h^\nu b_0^{(\nu)}(z)/f(z),\hfill&\lambda=\nu\hfill\\[2mm]
0,\hfill&\lambda<\nu\hfill
\end{matrix}\right\}\ \mathrm{mod}\ 2.
\end{equation}
From (\ref{Eq:SigmaEval1}) in turn, we conclude that each summand of
$\Sigma$ (and hence $\Sigma$ itself) is an integral power series, and
that 
\begin{equation}
\label{Eq:SigmaEval2}
\Sigma \equiv \sum_{\nu=1}^{\frac{q-1}{2}} d_q^{(\nu)} z^{q+\nu-1}
(\det\Delta_q)^{{\frac{q+1}{2}}-1} h^\nu
(\det{}_0\Delta_q)^{(\nu)}/f(z)\ \mathrm{mod}\ 2. 
\end{equation}
Next, using the fact that, by Lemma~\ref{Lem:CoeffMod2},
\[
d_q^{(\nu)} \equiv 1\ (\mathrm{mod}\ 2)\
\text { if, and only if, }\
\nu=\frac{q-1}{2},\quad0\leq\nu\leq\frac{q-1}{2}, 
\]
the right-hand side of (\ref{Eq:SigmaEval2}) in its turn is congruent
modulo~$2$ to 
\[
z^{\frac{3q-3}{2}} (\det\Delta_q)^{{\frac{q-1}{2}}}
h^{\frac{q-1}{2}} (\det{}_0\Delta_q)^{(\frac{q-1}{2})}/f(z), 
\]
which, in view of (\ref{Eq:SigmaEval1}) and
Lemmas~\ref{Lem:DetMod2}(i) and ~\ref{Lem:GenLogDerComp}(ii),
simplifies further to give 
\begin{align*}
\Sigma &\equiv z^{\frac{\alpha_0+3q-3}{2}}
(\det\Delta_q)^{{\frac{q-1}{2}}} h^{\frac{q-1}{2}}
b_0^{(\frac{q-1}{2})}(z)/f(z)\\[2mm] 
&\equiv z^{\frac{{\frac{q+1}{2}}\alpha_0+3q-3}{2}}\Big\{z^2
h^{\frac{q-1}{2}} f^{(\frac{q-1}{2})}(z)/f(z) + z^{\frac{3q+1}{2}}
h^{q-1} f^{(q-1)}(z)/f(z)\Big\}\\[2mm] 
&\equiv
z^{\frac{{\frac{q+1}{2}}\alpha_0+3q+1}{2}}
\big(\mathcal{S}_q^H(z)\big)^{\frac{q-1}{2}}
+ z^{\frac{{\frac{q+1}{2}}\alpha_0+6q-2}{2}}
\big(\mathcal{S}_q^H(z)\big)^{q-1}. 
\end{align*}
Putting all the pieces together and dividing by
$z^{{\frac{q+1}{4}}\alpha_0+1}$, we obtain the surprisingly elegant
congruence 
\begin{equation}
\label{Eq:Smod2}
z + \mathcal{S}_q^H(z) + z^{3q-2}\big(\mathcal{S}_q^H(z)\big)^{q-1}
\equiv 0\ \mathrm{mod}\ 2,\quad \vert H\vert\equiv0\ (\mathrm{mod}\ 2). 
\end{equation}
We are now in a position to establish our second main result.
\begin{theorem}
\label{Thm:SolvSmod2}
Let $q$ be an odd prime number. Then the number $N_q(n)$ of index $n$
subgroups in the Hecke group $\mathfrak{H}(q)$ isomorphic to a free
product of cyclic groups of order $q$ is odd if, and only if,
$n=2+4(q-1)\eta,$ where $\eta$ is a non-negative integer 
satisfying the condition that  
\begin{equation}
\label{Eq:etaCond}
  \mathfrak{s}_2((q-1)\eta+1) = \mathfrak{s}_2(\eta)
+ \mathfrak{s}_2((q-2)\eta+1). 
\end{equation}
Here, as before, 
$\mathfrak{s}_2(x)$ is the sum
of digits in the binary expansion of the positive integer $x$. 
\end{theorem}
\begin{proof}
Recall that, by \eqref{eq:s=N}, the number $N_q(n)$ is the same
modulo~2 as the number $s_q^H(n)$ for a group $H$ of even order. 
On the other hand, Equation~\eqref{Eq:Smod2} is a functional equation
for the generating function of the $s_q^H(n)$'s, taken modulo~2.
We now use this functional equation to determine $s_q^H(n)$ modulo~$2$, and
thereby $N_q(n)$ modulo~$2$ as well.

Define an integral power series $\hat{\mathcal{S}}_q^H(z)$ by means of
the equation 
\[
\hat{\mathcal{S}}_q^H(z) = z + z^{3q-2}
\big(\hat{\mathcal{S}}_q^H(z)\big)^{q-1}. 
\]
Then, in view of (\ref{Eq:Smod2}), we have
\[
\hat{\mathcal{S}}_q^H(z) \equiv S_q^H(z)\ \mathrm{mod}\ 2.
\]
Setting
\[
\hat{\mathcal{T}}_q^H(z):= z^{-1} \hat{\mathcal{S}}_q^H(z) = 1 +
\sum_{n\geq1} \hat{t}_q^H(n) z^n, 
\]
we obtain an integral power series with functional equation
\begin{equation}
\label{Eq:TqHFuncEq}
\hat{\mathcal{T}}_q^H(z) = 1 +
z^{4(q-1)}\big(\hat{\mathcal{T}}_q^H(z)\big)^{q-1}. 
\end{equation}
From the functional equation it is obvious that $\hat T_q^H(z)$ must
be a power series in $z^{4(q-1)}$; that is, 
the coefficients $\hat{t}_q^H(n)$ must satisfy
\begin{equation}
\label{Eq:tqHnon-trivial}
n\not\equiv0\bmod{4(q-1)}\,\Longrightarrow\, \hat{t}_q^H(n)=0,\quad n\geq1.
\end{equation}
Hence, introducing a series $\hat{\mathcal{U}}_q^H(v)$ via
\[
\hat{\mathcal{U}}_q^H(v) = \sum_{\eta\geq1} \hat{t}_q^H(4(q-1)\eta) v^\eta,
\]
we have
\[
\hat{\mathcal{T}}_q^H(z) - 1 = \hat{\mathcal{U}}_q^H(z^{4(q-1)}),
\]
and, setting $v=z^{4(q-1)}$, we find from (\ref{Eq:TqHFuncEq}) that
$\hat{\mathcal{U}}_q^H(v)$ satisfies the functional equation 
\[
\frac {\hat{\mathcal{U}}_q^H(v)}
{(\hat{\mathcal{U}}_q^H(v)+1)^{q-1}} = v.
\]
%with
%\[
%\Phi(\zeta):= (1 + \zeta)^{q-1}.
%\]
Using the notation $\big\langle v^\eta\big\rangle f(v)$ for the
coefficient of $v^\eta$ in the power series $f(v)$,
Lagrange inversion 
(cf.\ \cite[Theorem~5.4.2]{StanleyI}) 
implies that, for $\eta\geq1$, we have
\begin{align*}
\big\langle v^\eta\big\rangle \hat{\mathcal{U}}_q^H(v) &=
\eta^{-1}\big\langle\zeta^{\eta-1}\big\rangle
(1+\zeta)^{(q-1)\eta}\\[2mm]
&= \eta^{-1}\binom{(q-1)\eta}{\eta-1}\\[2mm]
&= \frac{1}{(q-1)\eta+1} \binom{(q-1)\eta+1}{\eta}.
\end{align*}
Substituting back, $\hat{\mathcal{S}}_q^H(z)$ is found to be given
explicitly by 
\[
\hat{\mathcal{S}}_q^H(z) = \sum_{\eta\geq0}\, \frac{1}{(q-1)\eta+1}\,
\binom{(q-1)\eta+1}{\eta}\, z^{4(q-1)\eta+1}, 
\]
which in turn allows us to deduce that
\begin{equation}
\label{Eq:sqHmod2expl}
s_q^H(n) \equiv \left\{\begin{matrix}
                 0;\hfill& n\not\equiv 2\ (\mathrm{mod }\ 4q-4)\hfill\\[2mm]
                 \frac{1}{(q-1)\eta+1} \binom{(q-1)\eta+1}{\eta};\hfill&
n=2+4(q-1)\eta,\,\eta\geq0 \hfill
                 \end{matrix}\right\}\ \mathrm{mod}\ 2,\quad n\geq1.
\end{equation}
Since $q$ is an odd prime, and, hence, $(q-1)\eta+1$ is odd,
we have shown that, for $\vert H\vert$ even, $s_q^H(n)$ is odd if, and only if,
$n=2+4(q-1)\eta$ for some $\eta\geq0$ satisfying 
\[
v_2\left(\binom{(q-1)\eta+1}{\eta}\right)=0.
\]
Applying Kummer's formula
(cf.\ \cite[pp.~115--116]{Kummer})
\[
v_2\left(\binom{a}{b}\right) = \mathfrak{s}_2(b) + \mathfrak{s}_2(a-b) -
\mathfrak{s}_2(a)
\]
for the $2$-adic valuation of a binomial coefficient, we find that
\[
v_2\left(\binom{(q-1)\eta+1}{\eta}\right) = \mathfrak{s}_2(\eta) +
\mathfrak{s}_2((q-2)\eta+1) - \mathfrak{s}_2((q-1)\eta+1), 
\]
and our statement concerning $s_q^H(n)$ can be rewritten as
\begin{equation*}
s_q^H(n)\equiv1\ (\mathrm{mod}\ 2)\
\text { if, and only if, }\
n=2+4(q-1)\eta\mbox{ for
some }\eta\geq0\mbox{ satisfying }(\ref{Eq:etaCond}). 
\end{equation*}
In view of \eqref{eq:s=N}, this establishes the theorem.
\end{proof}
In the case where $q$ is a Fermat prime we can be more explicit. 
\begin{corollary}
\label{Cor:GenParHecke}
If $q$ is a Fermat prime, then $N_q(n)$ is odd if, and only if,
$n=\frac{4(q-1)^{\sigma+1} - 2q}{q-2}$ for some $\sigma\geq0$. 
\end{corollary}
\begin{proof}
If $q$ is a Fermat prime, that is, $q-1=2^\lambda$ is a non-trivial
$2$-power, then 
\[
\mathfrak{s}_2((q-1)\eta+1)=1+\mathfrak{s}_2(\eta), 
\]
and Equation (\ref{Eq:etaCond}) simplifies to
\begin{equation}
\label{Eq:etaCondFermat}
\mathfrak{s}_2((q-2)\eta+1) =1.
\end{equation}
Therefore $(q-2)\eta+1$ must be a 2-power, $(q-2)\eta+1=2^\tau$, say.
Equivalently, this is
$$(2^\la-1) \eta=2^\tau-1,$$
from which we infer that $\tau$ must be a multiple of $\la$,
$\tau=\la\si$ say. 
Hence, we find that $\eta$ must be of the form 
\[
\eta = \frac{2^{\lambda\sigma}-1}{2^\lambda-1},\quad\sigma\geq0;
\]
and, conversely, for these values of $\eta$, Equation
(\ref{Eq:etaCondFermat}) indeed holds true. 
We conclude that, in the case when $q$ is a Fermat prime, $N_q(n)$ is odd if, and only if, 
\[
n = 2+4(q-1)\frac{(q-1)^\sigma-1}{q-2} =
\frac{4(q-1)^{\sigma+1}-2q}{q-2},\quad\sigma\geq0, 
\]
and the corollary is proven.
\end{proof}

\begin{remark}
A detailed analysis of \eqref{Eq:PrincEq} in the case when $h=\vert
H\vert$ is odd, making use of Lemma~\ref{Lem:DetMod2hunger},
eventually leads to a new proof of Theorem~\ref{Thm:sqn}, the case 
where $q\equiv3\ (\mathrm{mod}\ 4)$ being particularly involved. 
\end{remark}

\section{The main result}
\label{sec:main}

Here we summarize our findings concerning the parity behaviour of the
generalized subgroup numbers of the groups $\Ga_m(q)$ from the
previous sections. If we combine Proposition~\ref{Prop:Reduction1} with 
Theorems~\ref{Thm:sqn}, \ref{Thm:NqnParity}, and \ref{Thm:SolvSmod2}, 
then we obtain the following result.

\begin{theorem} \label{thm:main}
For $m,n\geq1,$ an odd prime $q,$ and a finite group $H$ with
$\gcd(m,\vert H\vert)=1$ we have 
$$s_{\Gamma_m(q)}^H(n) = \sum_{d\mid m}\,s_{\Gamma_d(q)}^H(1,n),$$
where
$s_{\Gamma_d(q)}^H(1,n) \equiv 1\ (\mathrm{mod}\ 2)$ if, and only
if,
\begin{equation*} %\label{}
n=2d(1+2(q-1)\eta)\text{ with }
  \mathfrak{s}_2((q-1)\eta+1) = \mathfrak{s}_2(\eta)
+ \mathfrak{s}_2((q-2)\eta+1)\text{ and $q\nmid d$,}\\
\end{equation*}
or
\begin{equation*}
n=d(1+2(q-1)\eta)\text{ with }
  \mathfrak{s}_2((q-1)\eta+1) = \mathfrak{s}_2(\eta)
+ \mathfrak{s}_2((q-2)\eta+1)\text{ and $2\nmid d$, $q\nmid d$,
$2\nmid\vert H\vert$}.
\end{equation*}
\end{theorem}

In the case where $q$ is a Fermat prime, we can be more explicit.
Namely, if we combine Proposition~\ref{Prop:Reduction1} with the
particular statement in Theorem~\ref{Thm:sqn},
Theorem~\ref{Thm:NqnParity}, and Corollary~\ref{Cor:GenParHecke},
we arrive at the following conclusion.

\begin{corollary} \label{cor:main}
For $m,n\geq1,$ a Fermat prime $q,$ and a finite group $H$ with\break
$\gcd(m,\vert H\vert)=1$ we have 
$$s_{\Gamma_m(q)}^H(n) = \sum_{d\mid m}\,s_{\Gamma_d(q)}^H(1,n),$$
where
$s_{\Gamma_d(q)}^H(1,n) \equiv 1\ (\mathrm{mod}\ 2)$ if, and only
if,
\begin{equation*} %\label{}
n=2d\,\frac {2(q-1)^\si-q} {q-2}\text{ with $\si\ge1$
  and $q\nmid d,$}
\end{equation*}
or
\begin{equation*}
n=d\,\frac {2(q-1)^\si-q} {q-2}\text{ with $\si\ge1,$
  $2\nmid d,$ $q\nmid d,$ and 
$2\nmid\vert H\vert$}.
\end{equation*}
\end{corollary}

Let us fix a Fermat prime $q=2^{\la}+1$.
Suppose that $p$ is an odd prime with the property that 
%Christian: Ich habe statt C(q) eine normale Gleichungsnummer
%   eingefuehrt. Das erscheint mir natuerlicher.
\begin{equation} \label{eq:Cq}
\hspace{5mm}\mbox{\it the multiplicative group generated by $q-1$ does not 
contain $2^{-1}q$ modulo $p$.}
\end{equation}
Then the equation
\begin{equation}
\label{Eq:Overlap}
t_1(2(q-1)^{\si_1}-q)=t_2(2(q-1)^{\si_2}-q)
\end{equation}
has no solutions $(t_1, t_2, \si_1, \si_2)$ where $p\mid t_1$, but $p\nmid t_2$,
since otherwise we would have 
\[
2(q-1)^{\sigma_2}\equiv q\ \mathrm{mod}\ p,
\] 
contrary to our assumption on $p$. Hence, if $q$ is a Fermat prime, 
and if the prime divisors of $m$
are among the set consisting of 2, $q$, and primes $p$ satisfying 
Condition~\eqref{eq:Cq}, then  
we obtain for all $H$ with $\gcd(m, |H|)=1$ that
\begin{multline} 
s_{\Gamma_m(q)}^H(n) \equiv 1\ (\mathrm{mod}\ 2) \mbox{ if, and only if, } 
n = t\,\frac{2(q-1)^\sigma-q}{q-2}\mbox{ with }\sigma\geq1,\, t\mid 2m,\,
q\nmid t,\\
\mbox{and $t$ even for $\vert H\vert$ even.}
\label{Eq:FermatBed}
\end{multline}
Statement (\ref{Eq:FermatBed}) raises the question which primes satisfy 
Condition~\eqref{eq:Cq} for a given Fermat prime $q$. This problem is 
addressed in our next result, for $q=3,5,17$. For larger Fermat primes the 
calculations involved, though essentially trivial, become unwieldy, and are 
omitted.
\begin{proposition}\leavevmode
\begin{enumerate}
\item[(i)] All prime numbers $p\equiv 7,17\ (\mathrm{mod}\ 24)$
satisfy Condition~\eqref{eq:Cq} with $q=3$.
\item[(ii)] All prime numbers 
$p\equiv 7, 11, 17, 19, 21, 23, 29, 33\ (\mathrm{mod}\ 40)$ 
satisfy Condit\-ion~\eqref{eq:Cq} with $q=5$.
\item[(iii)] All prime numbers 
\begin{multline*}
p\equiv 7, 13, 19, 21, 23, 31, 35, 39,  41, 43, 53, 57, 59, 63, 65, 67, 69, 71, 73, 77, 79, 83,\\ 93,  95, 97, 105, 113, 115, 117,  123, 125, 129 \ 
(\mathrm{mod}\ 136)
\end{multline*}
satisfy Condition~\eqref{eq:Cq} with $q=17$.
\end{enumerate}
\end{proposition}
\begin{proof}
This is a simple application of quadratic reciprocity. Let us first consider 
the case when $q=3$. If $p\neq3$ is an odd prime not satisfying 
Condition~\eqref{eq:Cq} with $q=3$, then we have
\[
2^\alpha \equiv 3\ \mathrm{mod}\ p
\]
for some $\alpha\geq1$, implying $\big(\frac{3}{p}\big)=+1$ or 
$\big(\frac{6}{p}\big)=+1$, according to whether or not $\alpha$ is even. 
Hence, every prime $p\equiv\pm1\ (\mathrm{mod}\ 8)$ not satisfying 
\eqref{eq:Cq} with $q=3$ has the property that $\big(\frac{3}{p}\big)=+1$. By 
quadratic reciprocity,
\[
\Big(\frac{p}{3}\Big) = \begin{cases}
                           +1,&p\equiv 1\ (\mathrm{mod}\ 4),\\[2mm]
                           -1,&p\equiv3\ (\mathrm{mod}\ 4),
                           \end{cases}
\]
for each prime $p$ not satisfying \eqref{eq:Cq} with $q=3$ and such that 
$p\equiv\pm1\ (\mathrm{mod}\ 8)$. It follows from this that every prime $p$ 
such that
\[
p\equiv\pm1\ (\mathrm{mod}\ 8)\,\mbox{ and }\,\Big(\frac{p}{3}\Big)=\begin{cases}
                            -1,& p\equiv 1\ (\mathrm{mod}\ 4),\\[2mm]
                            +1,& p\equiv 3\ (\mathrm{mod}\ 4),
                            \end{cases}
\]
will satisfy Condition~\eqref{eq:Cq} with $q=3$. Thus, both
\[
p\equiv 1\ (\mathrm{mod}\ 8)\,\mbox{ and }\,p\equiv 2\ (\mathrm{mod}\ 3)
\]
as well as
\[
p\equiv 7\ (\mathrm{mod}\ 8)\,\mbox{ and }\,p\equiv 1\ (\mathrm{mod}\ 3)
\]
are hypotheses sufficient to ensure that $p$ meets 
Condition~\eqref{eq:Cq} with $q=3$. The first pair of congruences is 
equivalent to $p\equiv 17\ (\mathrm{mod}\ 24)$, while the second one is 
equivalent to $p\equiv 7\ (\mathrm{mod}\ 24)$, whence Assertion~(i).

For 
%Christian: Ergaenzung
a Fermat prime $q$ with $q>3$, we have $q\equiv 1\ (\mathrm{mod}\ 4)$, so
\[
\Big(\frac{p}{q}\Big) = \Big(\frac{q}{p}\Big)
\]
for any odd prime $p\neq q$ by quadratic reciprocity. Consequently, if $p$ is 
such that $p\equiv\pm1\ (\mathrm{mod}\ 8)$ and \eqref{eq:Cq} does not hold 
for $p$, then we find that
\[
1 = \Big(\frac{2q}{p}\Big) = \Big(\frac{2}{p}\Big)\Big(\frac{q}{p}\Big) = 
\Big(\frac{q}{p}\Big) =  \Big(\frac{p}{q}\Big),
\]
while for $p\equiv\pm3\ (\mathrm{mod}\ 8)$ such that \eqref{eq:Cq} does 
not hold, we obtain $\big(\frac{p}{q}\big)=-1$. Hence, every prime $p$ 
satisfying either
\[
p\equiv\pm 1\ (\mathrm{mod}\ 8)\,\mbox{ and }\,\Big(\frac{p}{q}\Big)=-1
\]
or
\[
p\equiv\pm 3\ (\mathrm{mod}\ 8)\,\mbox{ and }\,\Big(\frac{p}{q}\Big)=+1
\]
will meet Condition~\eqref{eq:Cq}. Thus, it only remains to translate the 
statement concerning $\big(\frac{p}{q}\big)$ into a congruence condition 
modulo~$q$, and to solve the resulting pairs of congruences, each pair leading 
to an equivalent congruence conditon for $p$ modulo~$8q$. 

For $q=5$, we have
\[
\Big(\frac{p}{5}\Big)=+1\,\mbox{ if, and only if, }\, p\equiv\pm1\ 
(\mathrm{mod}\ 5),
\]
so both
\[
p\equiv\pm1\ (\mathrm{mod}\ 8)\,\mbox{ and }\,p\equiv 2,3\ (\mathrm{mod}\ 5)
\]
as well as
\[
p\equiv\pm3\ (\mathrm{mod}\ 8)\,\mbox{ and }\,p\equiv1,4\ (\mathrm{mod}\ 5)
\]
ensure that $p$ meets Condition~\eqref{eq:Cq} with $q=5$. 
This yields Part~(ii), 
while a corresponding calculation for $q=17$ gives Assertion~(iii).
\end{proof}
For $q$ a Fermat prime and $m$ subject only to the condition that 
$\gcd(m,\vert H\vert)=1$, we claim that (\ref{Eq:FermatBed})
%Christian: Was soll das "still" hier bedeuten?
%still 
holds true for $n\geq\frac{4m^2}{q-2}$. 
To prove this, it suffices to give an upper bound for non-diagonal
solutions of Equation~(\ref{Eq:Overlap}). 
Assume that $t_2>t_1$. Then
\begin{equation}
\label{Eq:FermatBound}
%Christian: modifiziert
0 < \sigma_1 - \sigma_2 < \frac{\log t_2/t_1}{\log (q-1)},
\end{equation}
and we compute: 
\begin{align*}
t_2\big(2(q-1)^{\si_2}-q\big) & = 
\gcd\Big(t_1\big(2(q-1)^{\si_1}-q\big), t_2\big(2(q-1)^{\si_2}-q\big)\Big)\\[2mm]
  & \leq  \gcd\Big(t_1t_2\big(2(q-1)^{\si_1}-q\big)-t_1t_2\big(2(q-1)^{\si_1}-
q (q-1)^{\si_1-\si_2}\big),\\[1mm]
&\kern8.5cm
t_2\big(2(q-1)^{\si_2}-q\big)\Big)\\[2mm]
%Christian: Ueberall \leq
  & \leq  \gcd\Big(qt_1t_2\big((q-1)^{{\si_1}-{\si_2}}-1\big),
  t_2\big(2(q-1)^{\si_2}-q\big)\Big)\\[2mm]
  & \leq t_2 \gcd\Big(t_1\big((q-1)^{\sigma_1-\sigma_2}-1\big), 2(q-1)^{\sigma_2}-q\Big)\\[2mm]
  & < t_2^2,
\end{align*}
where we have used (\ref{Eq:FermatBound}) in the last step. 
Our claim follows now from Corollary~\ref{cor:main} 
and the fact that $t_2\leq 2m$. 

In particular, we
see from the last observation and (\ref{Eq:FermatBed}), 
that the parity behaviour of $s_{\Gamma_m}^H(n)$
as a function in $n$ determines $m$ up to a power of
$q$, provided that $\gcd(m,|H|)=1$.

As a final example, we compute the parity behaviour of 
$s_{\Gamma_m(q)}^H(n)$ for $q=3$ and $m=625$. We have to determine the 
non-diagonal solutions of the equation
\begin{equation}
\label{Eq:Ueberlapp2}
t_1(2^{\sigma_1+1}-3) = t_2(2^{\sigma_2+1}-3)
\end{equation}
for $\sigma_1,\sigma_2\geq1$ and $t_1,t_2\mid 1250$. We have $t_1\equiv 
t_2\ (\mathrm{mod}\ 2)$; thus, we may concentrate on the case when both $t_1$ 
and $t_2$ are odd, bearing in mind that each non-diagonal solution of 
(\ref{Eq:Ueberlapp2}) with $t_1,t_2$ odd gives rise to two non-diagonal 
solutions, namely $(t_1,t_2,\sigma_1,\sigma_2)$ and $(2t_1,2t_2,\sigma_1,
\sigma_2)$.

Suppose without loss of generality that $t_2>t_1$, and let 
$\delta:=v_5(t_2)-v_5(t_1)$. For $\delta=1$, we deduce that $2^{\sigma_1+1} 
\equiv 3\ (\mathrm{mod}\ 5)$, which is equivalent to $\sigma_1 \equiv 2\ 
(\mathrm{mod}\ 4)$. The case where $\sigma_1=2$ leads to the solutions
\begin{equation}
\label{Eq:Ueberlapp3}
(t_1,t_2,\sigma_1,\sigma_2) = (5^a,5^{a+1},2,1),\quad a=0,1,2,3.
\end{equation}
If $\sigma_1=6,10$, then $2^{\sigma_2+1}$ would have to equal $28$ 
respectively $412$, which is impossible; while for $\sigma_1\geq14$, we 
deduce that
\[
2^{\sigma_2+1} - 3 \geq \frac{2^{15}-3}{5} = 6553,
\]
contradicting the estimate
\[
2^{\sigma_2+1}-3 < t_2\leq 625
\]
obtained above. Hence, the only non-diagonal solutions of 
(\ref{Eq:Ueberlapp2}) with $t_1$ and $t_2$ odd and $\delta=1$ are given by 
(\ref{Eq:Ueberlapp3}).

For $\delta=2$, we get that $2^{\sigma_1+1} \equiv 3\ (\mathrm{mod}\ 25)$, 
which is equivalent to  $\sigma_1 \equiv 6\ (\mathrm{mod}\ 20)$. The case 
where $\sigma_1=6$ leads to the solutions
\begin{equation}
\label{Eq:Ueberlapp4}
(t_1,t_2,\sigma_1,\sigma_2) = (5^a,5^{a+2}, 6,2),\quad a=0,1,2.
\end{equation}
If $\sigma_1\geq 26$, then we deduce that
\[
2^{\sigma_2+1}-3 \geq \frac{2^{27}-3}{5^2} = 5368709,
\]
again contradicting the estimate $2^{\sigma_2+1}-3 < t_2$. Consequently 
(\ref{Eq:Ueberlapp4}) describes the only non-diagonal solutions of 
(\ref{Eq:Ueberlapp2}) with $t_1,t_2$ odd and $\delta=2$.

For $\delta=3$, we have $2^{\sigma_1+1}\equiv 3\ (\mathrm{mod}\ 125)$, 
which is equivalent to $\sigma_1\equiv 6\ (\mathrm{mod}\ 100)$, leading to 
the solutions
\begin{equation}
\label{Eq:Ueberlapp5}
(t_1,t_2,\sigma_1,\sigma_2) = (5^a, 5^{a+3},6,1),\quad a=0,1.
\end{equation}
Finally, for $\delta=4$, we have $2^{\sigma_1+1}\equiv 3\ 
(\mathrm{mod}\ 625)$, which is equivalent to $\sigma_1\equiv\break 106\ 
(\mathrm{mod}\ 500)$. But $106$ is already well above the range allowed for 
$\sigma_1$. Hence, there is precisely one non-diagonal solution of 
(\ref{Eq:Ueberlapp2}) for $n=5,10,25,50,3125,6250$; there are exactly three 
non-diagonal solutions for $n=125,250,625,1250$; and there are no non-diagonal 
solutions for other values of $n$. Consequently, for $\gcd(5,\vert H\vert)=1$, 
Assertion (\ref{Eq:FermatBed}) holds true for $n>6250$, while for $n\leq 6250$ 
we have
\begin{multline*}
s_{\Gamma_{625}(3)}^H(n) \equiv 1\ \mathrm{mod}\ 2\,\mbox{ if, 
and only if, }\\
n= 1,2,13,26,29,58,61,65,122,125,130,145,250,253,290,305,325,506,\\
509,610,625,650,725,1018,1021,1250,1265,1450,1525,1625,2042,2045,\\
2530,2545,3050,3250,3625,4090,4093,5090,5105
\end{multline*}
if $\vert H\vert$ is odd, and
\begin{multline*}
s_{\Gamma_{625}(3)}^H(n) \equiv 1\ \mathrm{mod}\ 2\,\mbox{ if, 
and only if, }\\
n=2,26,58,122,130,250,290,506,610,650,1018,1250,1450,2042,2530,\\
3050,3250,4090,5090
\end{multline*}
for $\vert H\vert$ even.

%Christian: modifiziert
\section*{Acknowledgements}
The authors are indebted to the anonymous referee for a careful 
reading of the manuscript, and for helpful 
suggestions concerning the presentation of the material.

\bigskip

\appendix

\global\def\theTheorem{\mbox{A%\arabic{Theorem}
}}
\setcounter{theorem}{0}

\def\thesection{\!\!\!}
\section{Proofs of some auxiliary results}

\def\thesection{A}

\subsection{Proof of Lemma~\ref{lem:multi}}
\label{Subsec:Lem12}
Suppose first that $\rh_1\equiv \rh_2\equiv\dots\equiv
\rh_\al\equiv1\ (\mathrm{mod}\ 4)$. It is well known that the 2-adic valuation
of the multinomial coefficient 
$\binom{\rho_1+\rh_2+\ldots+\rho_\al}
{\rho_1,\rh_2\ldots,\rho_\al}$ is equal to the number of carries, $C$
say, occurring during addition of the numbers
${\rho_1,\rh_2\ldots,\rho_\al}$. If we assume that, in total, we find
$E_\ell$ 1's as the $\ell$-th digit in the binary representations of
${\rho_1,\rh_2\ldots,\rho_\al}$, then 
\begin{align*}
C&=\fl{\frac {E_0}2}+\fl{\frac {E_1+\fl{\frac {E_0}2}}2}
+\fl{\frac {E_2+\fl{\frac{E_1+\fl{\frac {E_0}2}}2}}2}+\cdots\\
&\ge
\fl{\frac {E_0}2}+\fl{\frac {\fl{\frac {E_0}2}}2}
+\fl{\frac {\fl{\frac {\fl{\frac {E_0}2}}2}}2}+\cdots
=\sum _{\ell\ge1} ^{}\fl{\frac {E_0}{2^\ell}}.
\end{align*}
According to our assumption, we have $E_0=\al$, whence
$$v_2\left(\binom{\rho_1+\rh_2+\ldots+\rho_\al}
{\rho_1,\rh_2\ldots,\rho_\al}\right) = C\,\ge\,
\sum _{\ell\ge1}\fl{\frac {\al}{2^\ell}}=
v_2(\al!).
$$
The second assertion is established in a similar manner.\quad \quad \qed

\subsection{Proof of Lemma~\ref{Lem:BinomParity}}
\label{Subsec:Lem14}
Let $k=\sum_{j\geq0}k_j2^j$ be the $2$-adic expansion of $k$. Then
\[
2^\lambda k+1 = 1\cdot2^0 + 0\cdot2^1 +\cdots+0\cdot 2^{\lambda-1} +
\sum_{j\geq0} k_j2^{j+\lambda} 
\]
is the $2$-adic expansion of $2^\lambda k+1$. Suppose first that $k$
is odd. Then $k_0=1$, and so 
\[
k-1 = 0\cdot2^0 + \sum_{j\geq1} k_j 2^j.
\]
By Lucas' congruence
(cf., for instance,
\cite[Theorem~3.4.1]{Cameron}), $\binom{2^\lambda
k+1}{k-1}\equiv1$~(mod~2) is equivalent to the conjunction of 
\begin{equation}
\label{Eq:BinomParity1}
k_1=k_2=\cdots=k_{\lambda-1}=0
\end{equation}
and 
\begin{equation}
\label{Eq:BinomParity2}
k_{\mu+\lambda} \leq k_\mu,\quad \mu\geq0.
\end{equation}
The conjunction of (\ref{Eq:BinomParity1}) and (\ref{Eq:BinomParity2})
in turn is equivalent to 
\begin{equation}
\label{Eq:BinomParity3}
k_j=0\,\,(j\not\equiv0\ (\mathrm{mod}\ \lambda))\,\,\mbox{ and
}\,\,k_{\mu\lambda}\leq k_{(\mu-1)\lambda}\,\,(\mu\geq1). 
\end{equation}
Let $\sigma\geq1$ be smallest with the property that $k_{\sigma\lambda}=0$. Then
$k_{\mu\lambda}=1$ for $\mu=0,1,\ldots,\sigma-\nolinebreak1$ and $k_{\mu\lambda}=0$
for $\mu\geq\sigma$, so, by (\ref{Eq:BinomParity3}),
\[
k = \sum_{0\leq \mu\leq\sigma-1} 2^{\lambda \mu} =
\frac{2^{\lambda\sigma}-1}{2^\lambda-1},\quad\sigma\geq1. 
\]
Hence, 
\begin{equation}
\label{Eq:BinomParityCond1}
\binom{2^\lambda k+1}{k-1} \equiv 1\ (\mathrm{mod}\ 2)\
\text { if, and only if, }\
k=\frac{2^{\lambda\sigma}-1}{2^\lambda-1}\,\mbox{ for some }\,
\sigma\geq1,\quad k\mbox{ odd.} 
\end{equation}
Now suppose that $k$ is even, and write
\[
k=0\cdot2^0 +\cdots+0\cdot2^{r-1} + 1\cdot2^r + \sum_{j>r} k_j2^j
\]
with some $r\geq1$ (such $r$ must exist since $k>0$). Then
\[
k-1 = 1\cdot2^0+\cdots+1\cdot2^{r-1} + 0\cdot2^r + \sum_{j>r} k_j2^j
\]
is the $2$-adic expansion of $k-1$ in this case. We now distinguish
two cases, according to whether or not $\lambda=1$.

(i) If $\lambda>1$, then the coefficient of $2^1$ in $2^\lambda k+1$
vanishes, while, for $r\geq2$, the corresponding coefficient of $k-1$
equals $1$. Hence, $\binom{2^\lambda k+1}{k-1}\equiv1$~(mod~2)
together with $\lambda>1$ forces $r=1$, so that the expansion of $k-1$
reads 
\[
k-1 = 1\cdot2^0 + 0\cdot2^1 + \sum_{j\geq2} k_j2^j.
\]
Now, again by Lucas' theorem, $\binom{2^\lambda
k+1}{k-1}\equiv1$~(mod~2) is equivalent to the conjunction of 
\begin{equation}
\label{Eq:BinomParity4}
k_0=k_2=\cdots=k_{\lambda-1}=0,\, k_1=1
\end{equation}
and
\begin{equation}
\label{Eq:BinomParity5}
k_j\leq k_{j-\lambda},\quad j\geq\lambda.
\end{equation}
The conjunction of (\ref{Eq:BinomParity4}) and (\ref{Eq:BinomParity5})
in turn is equivalent to 
\begin{equation}
\label{Eq:BinomParity6}
k_j=0\,\,(j\not\equiv1\ (\mathrm{mod}\ \lambda))\,\mbox{ and }\,k_{\mu\lambda+1}
\leq k_{(\mu-1)\lambda+1}\, (\mu\geq1). 
\end{equation}
Let $\sigma\geq1$ be smallest with the property that
$k_{\sigma\lambda+1}=0$, so that $k_{\mu\lambda+1}=1$ for
$0\leq\mu\leq\sigma-1$ and $k_{\mu\lambda+1}=0$ for
$\mu\geq\sigma$. Thus, by (\ref{Eq:BinomParity6}), 
\[
k = \sum_{0\leq\mu\leq\sigma-1} 2^{\lambda\mu+1}
= 2\bigg(\sum_{0\leq\mu\leq\sigma-1} 2^{\lambda\mu}\bigg)
= \frac{2^{\lambda\sigma+1}-2}{2^\lambda-1}.
\]
Hence, 
\begin{multline}
\label{Eq:BinomParityCond2}
\binom{2^\lambda k+1}{k-1} \equiv 1\ (\mathrm{mod}\ 2)\
\text { if, and only if, }\\
k=\frac{2^{\lambda\sigma+1}-2}{2^\lambda-1}\,\mbox{ for some }\,
\sigma\geq1,\quad(2\mid k,\,\lambda\geq2). 
\end{multline}
(ii) Now let $\lambda=1$, so that
\[
2^\lambda k+1 = 1\cdot2^0 + \sum_{j\geq0} k_j 2^{j+1}.
\]
Application of Lucas' congruence shows that, in this case,
$\binom{2^\lambda k+1}{k-1}\equiv1\ (\mathrm{mod}\ 2)$ is equivalent to the
conjunction of 
\begin{equation}
\label{Eq:BinomParity7}
k_0=k_1=\cdots=k_{r-2}=1,\,k_{r-1}=0,\,k_r=1, 
\end{equation}
and
\begin{equation}
\label{Eq:BinomParity8}
k_j\leq k_{j-1},\quad j>r.
\end{equation}
Since, by definition of $r$, we have $k_j=0$ for $0\leq j<r$, we again
conclude that $r=1$, so that the conjunction of
(\ref{Eq:BinomParity7}) and (\ref{Eq:BinomParity8}) is now equivalent
to 
\begin{equation}
\label{Eq:BinomParity9}
k_0=0,\, k_1=1,\, k_j\leq k_{j-1}\, (j\geq2).
\end{equation}
Let $\sigma\geq1$ be smallest with the property that
$k_{\sigma+1}=0$. Then $k_\mu=1$ for $1\leq \mu\leq\sigma$, and
$k_\mu=0$ for $\mu>\sigma$, and so, by (\ref{Eq:BinomParity9}), 
\[
k = \sum_{1\leq\mu\leq\sigma} 2^\mu = 2^{\sigma+1}-2. 
\]
Thus, 
\begin{equation}
\label{Eq:BinomParityCond3}
\binom{2^\lambda
k+1}{k-1}\equiv1\ (\mathrm{mod}\ 2)\
\text { if, and only if, }\
k=2^{\sigma+1}-2\,\mbox{
for some }\sigma\geq1,\quad(2\mid k,\,\lambda=1). 
\end{equation}
Our claim (\ref{Eq:BinomParityCond}) follows now from
(\ref{Eq:BinomParityCond1}), (\ref{Eq:BinomParityCond2}), and
(\ref{Eq:BinomParityCond3}).\quad \quad \qed 

\subsection{Proof of Lemma~\ref{Lem:FkRed}}
\label{Subsec:Lem18}
We use induction on $k$. For $k=0,1$, Equation (\ref{Eq:FkRed}) is
trivial, while for $k=2$ it follows from (\ref{Eq:FRelI}) with
$k=1$. Suppose that (\ref{Eq:FkRed}) holds for $k\leq K$ with some
integer $K\geq2$. Then, by (\ref{Eq:FRelI}) with $k=K$, the inductive hypothesis,
and the definition of the coefficient systems $\big(c_k^{(\mu)}\big)$
and $\big(d_k^{(\nu)}\big)$, we have  
\begin{align*}
F_{K+1} &= azF_K(z) + hz^2F_{K-1}(z) + hz^3F_{K-1}'(z)\\[2mm]
&= az\bigg\{\sum_{\mu=0}^{\lfloor\frac{K}{2}\rfloor} c_K^{(\mu)} h^\mu
z^{K+\mu} F_0^{(\mu)}\, +\, \sum_{\nu=0}^{\lfloor\frac{K-1}{2}\rfloor}
d_K^{(\nu)} h^\nu z^{K+\nu-1} F_1^{(\nu)}\bigg\}\\ 
&\hspace{1.5cm}+\,hz^2\bigg\{\sum_{\mu=0}^{\lfloor\frac{K-1}{2}\rfloor}
c_{K-1}^{(\mu)} h^\mu z^{K+\mu-1} F_0^{(\mu)}\, +\,
\sum_{\nu=0}^{\lfloor\frac{K-2}{2}\rfloor} d_{K-1}^{(\nu)} h^\nu
z^{K+\nu-2} F_1^{(\nu)}\bigg\}\\ 
&\hspace{1.5cm}+\,hz^3\bigg\{\sum_{\mu=0}^{\lfloor\frac{K-1}{2}\rfloor}
c_{K-1}^{(\mu)} h^\mu z^{K+\mu-1} F_0^{(\mu)}\, +\,
\sum_{\nu=0}^{\lfloor\frac{K-2}{2}\rfloor} d_{K-1}^{(\nu)} h^\nu
z^{K+\nu-2} F_1^{(\nu)}\bigg\}'\\[2mm] 
&=(ac_K^{(0)}+hKc_{K-1}^{(0)})z^{K+1}F_0(z)\\
&\hspace{1.5cm}+\,\sum_{\mu=1}^{\lfloor\frac{K-1}{2}\rfloor}
(ac_K^{(\mu)}+h(K+\mu)c_{K-1}^{(\mu)}+c_{K-1}^{(\mu-1)})h^\mu
z^{K+\mu+1}F_0^{(\mu)}(z)\\ 
&\hspace{1.5cm}+\,h^{\lfloor\frac{K+1}{2}\rfloor}
z^{K+\lfloor\frac{K+1}{2}\rfloor+1}
F_0^{(\lfloor\frac{K+1}{2}\rfloor)}(z)\times\begin{cases} 
c_{K-1}^{(\frac{K-1}{2})},& K\mbox{ odd}\\[2mm]
ac_K^{(K/2)}+c_{K-1}^{(\frac{K-2}{2})},& K\mbox{ even}
\end{cases}\\
&\hspace{1.5cm}+\,(ad_K^{(0)}+h(K-1)d_{K-1}^{(0)})z^K F_1(z)\\
&\hspace{1.5cm}+\,\sum_{\nu=1}^{\lfloor\frac{K-2}{2}\rfloor}
(ad_K^{(\nu)}+h(K+\nu-1)d_{K-1}^{(\nu)}+d_{K-1}^{(\nu-1)})z^{K+\nu}
h^\nu F_1^{(\nu)}(z)\\ 
&\hspace{1.5cm}+\,h^{\lfloor\frac{K}{2}\rfloor}
z^{K+\lfloor\frac{K}{2}\rfloor}
F_1^{(\lfloor\frac{K}{2}\rfloor)}(z)\times\begin{cases} 
ad_K^{(\frac{K-1}{2})} + d_{K-1}^{(\frac{K-3}{2})},&K\mbox{ odd}\\[2mm]
d_{K-1}^{(\frac{K-2}{2})},&K\mbox{ even}
\end{cases}\\[2mm]
&= c_{K+1}^{(0)} z^{K+1}
F_0(z)\,+\,\sum_{\mu=1}^{\lfloor\frac{K-1}{2}\rfloor} c_{K+1}^{(\mu)}
h^\mu z^{K+\mu+1} F_0^{(\mu)}(z)\\ 
&\hspace{1.5cm}+\,c_{K+1}^{(\lfloor\frac{K+1}{2}\rfloor)}
h^{\lfloor\frac{K+1}{2}\rfloor} z^{K+\lfloor\frac{K+1}{2}\rfloor+1}
F_0^{(\lfloor\frac{K+1}{2}\rfloor)}(z)\\ 
&\hspace{1.5cm}+\,d_{K+1}^{(0)} z^K
F_1(z)\,+\,\sum_{\nu=1}^{\lfloor\frac{K-2}{2}\rfloor} d_{K+1}^{(\nu)}
h^\nu z^{K+\nu} F_1^{(\nu)}(z)\\ 
&\hspace{1.5cm}+\,d_{K+1}^{(\lfloor\frac{K}{2}\rfloor)}
h^{\lfloor\frac{K}{2}\rfloor} z^{K+\lfloor\frac{K}{2}\rfloor}
F_1^{(\lfloor\frac{K}{2}\rfloor)}(z)\\[2mm] 
&= \sum_{\mu=0}^{\lfloor\frac{K+1}{2}\rfloor} c_{K+1}^{(\mu)} h^\mu
z^{K+\mu+1}
F_0^{(\mu)}(z)\,+\,\sum_{\nu=0}^{\lfloor\frac{K}{2}\rfloor}
d_{K+1}^{(\nu)} h^\nu z^{K+\nu} F_1^{(\nu)}(z), 
\end{align*}
as required.\quad \quad \qed

\subsection{Explicit description of the numbers 
$\omega_{\kappa,\lambda}$ and $b_{\kappa}(z)$}
\label{app:db}

We have
\begin{multline*}
\omega_{\kappa,\lambda} =\\[3mm] 
\begin{cases}
                             d_{q-1}^{(0)}z^q-1;&\kappa=\lambda=0\\[2mm]
                             d_{q-1}^{(\lambda)}
z^{\lambda+q};&\kappa=0,\,\lambda=1,\ldots,\frac{q-3}{2}\\[2mm]
                             d_q^{(0)}z^{q-1}+b;&\kappa=1,\,\lambda=0\\[2mm]
                             d_q^{(\lambda)}z^{\lambda+q-1};&
\kappa=1,\,\lambda=1,\ldots,\frac{q-1}{2}\\[2mm]
                             d_{\kappa-1}^{(0)}hz^{\kappa-2}+
bd_\kappa^{(0)}z^{\kappa-1}\\
\kern1cm +d_{\kappa+q-1}^{(0)}z^{\kappa+q-2};&3\leq\kappa\leq q-2,\,
\kappa\mbox{ odd},\,\lambda=0\\[2mm]
                             -((\lambda-1)d_{\kappa-1}^{(\lambda)}h+
d_{\kappa-1}^{(\lambda-1)})z^{\kappa+\lambda-2}\\
\kern1cm +bd_\kappa^{(\lambda)}z^{\kappa+\lambda-1}+
d_{\kappa+q-1}^{(\lambda)}z^{\kappa+\lambda+q-2};&
3\leq\kappa\leq q-2,\,\kappa\mbox{ odd},\,\lambda=1,\ldots,
\frac{\kappa-3}{2}\\[2mm]
                             -d_{\kappa-1}^{(\frac{\kappa-3}{2})}
z^{\frac{3\kappa-5}{2}}+
bd_\kappa^{(\frac{\kappa-1}{2})}z^{\frac{3\kappa-3}{2}}\\
\kern1cm +d_{\kappa+q-1}^{(\frac{\kappa-1}{2})}z^{\frac{3\kappa+2q-5}{2}};&
3\leq\kappa\leq q-2,\,\kappa\mbox{ odd},\,\lambda=
\frac{\kappa-1}{2}\\[2mm]
                             d_{\kappa+q-1}^{(\lambda)}z^{\kappa+\lambda+q-2};&
3\leq\kappa\leq q-2,\,\kappa\mbox{
odd},\,\lambda=\frac{\kappa+1}{2},\ldots,
\frac{\kappa+q-2}{2}\\[2mm]
                             d_{\kappa-1}^{(0)}hz^{\kappa-2}+
bd_\kappa^{(0)}z^{\kappa-1}\\
\kern1cm +d_{\kappa+q-1}^{(0)}z^{\kappa+q-2};&
2\leq\kappa\leq q-3,\,\kappa\mbox{ even},\,\lambda=0\\[2mm]
                             -((\lambda-1)d_{\kappa-1}^{(\lambda)} 
h+d_{\kappa-1}^{(\lambda-1)})z^{\kappa+\lambda-2}\\
\kern1cm +bd_\kappa^{(\lambda)}z^{\kappa+\lambda-1}+
d_{\kappa+q-1}^{(\lambda)}z^{\kappa+\lambda+q-2};&
2\leq\kappa\leq q-3,\,\kappa\mbox{ even},\,\lambda=1,\ldots,
\frac{\kappa-2}{2}\\[2mm]
                             -d_{\kappa-1}^{(\frac{\kappa-2}{2})}
z^{\frac{3\kappa-4}{2}}
+d_{\kappa+q-1}^{(\kappa/2)}z^{\frac{3\kappa+2q-4}{2}};&
2\leq\kappa\leq q-3,\,\kappa\mbox{ even},\,\lambda
=\frac{\kappa}{2}\\[2mm]
                             d_{\kappa+q-1}^{(\lambda)}z^{\kappa+\lambda+q-2};&
2\leq\kappa\leq q-3,\,\kappa\mbox{ even},\,\lambda
=\frac{\kappa+2}{2},\ldots,\frac{\kappa+q-3}{2}\\[2mm]
                             0;&\mbox{otherwise,}
                             \end{cases}
\end{multline*}
and
\begin{multline*}
b_\kappa(z) =\\[3mm] 
\begin{cases}
                  -(az+bz^2+c_{q-1}^{(0)}z^{q+1})f(z)
-
\sum\limits_{\mu=1}^{\frac{q-1}{2}}c_{q-1}^{(\mu)}h^\mu
z^{\mu+q+1}f^{(\mu)}(z);&\kern-1cm \kappa=0\\[2mm]
                  -c_q^{(0)}z^qf(z) - h(c_q^{(1)}z^{q+1}-z)f'(z) -
\sum\limits_{\mu=2}^{\frac{q-1}{2}} c_q^{(\mu)} h^\mu
z^{\mu+q}f^{(\mu)}(z);&\kern-1cm \kappa=1\\[2mm]
                  -(bc_\kappa^{(0)}z^\kappa+
c_{\kappa+q-1}^{(0)}z^{\kappa+q-1})f(z)\\ 
\kern1cm - \sum\limits_{\mu=1}^{\frac{\kappa-1}{2}}
((c_{\kappa-1}^{(\mu-1)}-\mu h c_{\kappa-1}^{(\mu)}) 
z^{\kappa+\mu-1}\\ 
\hspace{2cm} + bc_\kappa^{(\mu)} z^{\kappa+\mu} + 
c_{\kappa+q-1}^{(\mu)} z^{\kappa+\mu+q-1})h^\mu f^{(\mu)}(z)\\ 
\kern1cm - (c_{\kappa-1}^{(\frac{\kappa-1}{2})}z^{\frac{3\kappa-1}{2}}
+c_{\kappa+q-1}^{(\frac{\kappa+1}{2})}z^{\frac{3\kappa+2q-1}{2}})
h^{\frac{\kappa+1}{2}} f^{(\frac{\kappa+1}{2})}(z)\\ 
\kern1cm - \sum\limits_{\mu=\frac{\kappa+3}{2}}^{\frac{\kappa+q-2}{2}} 
c_{\kappa+q-1}^{(\mu)} h^\mu z^{\kappa+\mu+q-1} 
f^{(\mu)}(z);&\kern-1cm  3\leq\kappa\leq q-2,\,\kappa\mbox{ odd}\\[2mm]
                  -(bc_\kappa^{(0)}z^\kappa+c_{\kappa+q-1}^{(0)}
z^{\kappa+q-1})f(z)\\ 
\kern1cm - \sum\limits_{\mu=1}^{\frac{\kappa-2}{2}}
((-\mu h c_{\kappa-1}^{(\mu)}-c_{\kappa-1}^{(\mu-1)})
z^{\kappa+\mu-1}\\ 
\hspace{2cm}+ bc_\kappa^{(\mu)} z^{\kappa+\mu} + c_{\kappa+q-1}^{(\mu)} 
z^{\kappa+\mu+q-1})h^\mu f^{(\mu)}(z)\\ 
\kern1cm - (-c_{\kappa-1}^{(\frac{\kappa-2}{2})}z^{\frac{3\kappa-2}{2}} 
+ bc_\kappa^{(\kappa/2)}z^{3\kappa/2} 
+ c_{\kappa+q-1}^{(\kappa/2)}z^{\frac{3\kappa+2q-2}{2}})h^{\kappa/2} 
f^{(\kappa/2)}(z)\\
\kern1cm -
\sum\limits_{\mu=\frac{\kappa+2}{2}}^{\frac{\kappa+q-1}{2}}
c_{\kappa+q-1}^{(\mu)} h^\mu z^{\kappa+\mu+q-1}
f^{(\mu)}(z);&\kern-1cm 2\leq\kappa\leq q-3,\,\kappa\mbox{ even.}
\end{cases} \kern-8pt
\end{multline*}

\subsection{Proof of Lemma~\ref{Lem:DetMod2}}
\label{Subsec:Lem21}
Using Lemma~\ref{Lem:CoeffMod2} plus the facts that $h$ is even and that 
$b$ is odd (recall \eqref{eq:Cqung}), 
we find that, for $0\leq\kappa,\lambda\leq q-2$,
\begin{equation}
\label{Eq:DeltaMod2}
\omega_{\kappa,\lambda} \equiv \left\{\begin{matrix}
                                 1;\hfill&\kappa=\lambda=0\hfill\\[2mm]
                                 z^{\frac{3\kappa-3}{2}};\hfill&
1\leq\kappa\leq q-2,\, \kappa\mbox{ odd},\,
\lambda=\frac{\kappa-1}{2}\hfill\\[2mm]
                                 z^{\frac{3\kappa+3q-6}{2}};\hfill&
1\leq\kappa\leq q-2,\, \kappa\mbox{ odd},\,
\lambda=\frac{\kappa+q-2}{2}\hfill\\[2mm]
                                 z^{\frac{3\kappa-4}{2}};\hfill&
2\leq\kappa\leq q-3,\,\kappa\mbox{
even},\,\lambda=\frac{\kappa}{2}\hfill\\[2mm]
                                 0;\hfill&\mbox{otherwise}\hfill
                                 \end{matrix}\right\}\ \mathrm{mod}\ 2.
\end{equation}
(i) Writing
\begin{equation} \label{eq:detdef}
\det\Delta_q \equiv \sum_{\sigma\in\mbox{\scriptsize
Sym}(\{0,1,\ldots,q-2\})} \omega_{0,\sigma(0)}
\omega_{1,\sigma(1)}\cdots \omega_{q-2,\sigma(q-2)}\ \mathrm{mod}\ 2, 
\end{equation}
we see that the only non-zero contribution on the right-hand side
occurs for the permutation $\sigma$ given by 
\begin{equation} \label{eq:perm} 
\sigma(\kappa) = \begin{cases}
                     \frac{\kappa+q-2}{2};& 1\leq\kappa\leq q-2,\,
\kappa\mbox{ odd}\\[2mm]
                     \frac{\kappa}{2};& 0\leq\kappa\leq q-3,\,
\kappa\mbox{ even.}
                     \end{cases}
\end{equation}
Indeed, we clearly must have $\sigma(0)=0$, which in turn forces
$\sigma(1)=\frac{q-1}{2}$; further, for $\kappa\geq2$ even, we must
have $\sigma(\kappa)=\kappa/2$, and this in turn forces
$\sigma(\kappa)=\frac{\kappa+q-2}{2}$ for $\kappa\geq3$ odd, since the
alternative value $\frac{\kappa-1}{2}$ already occurs as image of the
even number $\kappa-1$. It follows that, modulo~$2$, 
\[
\det\Delta_q \equiv
z^{\frac{3q-3}{2}}\cdot\underset{\kappa\mathrm{\,
odd}}{\prod_{3\leq\kappa\leq q-2}}z^{\frac{3\kappa+3q-6}{2}}\cdot
\underset{\kappa\,\mathrm{even}}{\prod_{2\leq\kappa\leq q-3}}
z^{\frac{3\kappa-4}{2}} \equiv z^e, 
\]
where
\begin{align*}
e &=
\frac{3(q-1)}{2}\,+\,\underset{\kappa\,\mathrm{odd}}{\sum_{3\leq\kappa\leq
q-2}}
\frac{3\kappa+3q-6}{2}\,+\,\underset{\kappa\,\mathrm{even}}
{\sum_{2\leq\kappa\leq
q-3}} \frac{3\kappa-4}{2}\\[2mm] 
&= \frac{3(q-1)}{2}\,+\,\frac{9(q-1)(q-3)}{8}\,+\,
\frac{3(q-1)(q-3)}{8}\,-\,(q-3)\\[2mm]
&= \frac{3q^2-11q+12}{2}.
\end{align*}
The fact that
\[
\det\Delta_q \equiv \det\Delta_{0,0}\ \mathrm{mod}\ 2
\]
follows immediately by expanding $\det\Delta_q$ modulo~$2$ with
respect to the 0-th row.

(ii) For $1\leq\kappa\leq q-2$, we have
\[
\det\Delta_{\kappa,0} \equiv \sum_\sigma
\omega_{0,\sigma(0)}\cdots
\omega_{\kappa-1,\sigma(\kappa-1)}\omega_{\kappa+1,\sigma(\kappa+1)}\cdots
\omega_{q-2,\sigma(q-2)} \equiv 0\ \mathrm{mod}\ 2. 
\]
Here, $\sigma$ runs through all bijections of the set
$\{0,\ldots,\kappa-1,\kappa+1,\ldots,q-2\}$ onto the set
$\{1,2,\ldots,q-2\}$, and $\omega_{0,\sigma(0)}\equiv0\ (\mathrm{mod}\ 2)$ for
each such $\sigma$.\quad \quad \qed

\subsection{Proof of Lemma~\ref{Lem:DetMod2hunger} (sketch)}
\label{Subsec:Lem22}
Using Lemma~\ref{Lem:CoeffMod2} plus the facts that $h$ and $b$ are
odd, we find that, for $0\leq\kappa,\lambda\leq q-2$,
\begin{equation}
\label{Eq:DeltaMod2hunger}
\omega_{\kappa,\lambda} \equiv \left\{\begin{matrix}
                                 1;\hfill&\kappa=\lambda=0\hfill\\[2mm]
                                 z^{\frac{3q-3}{2}};\hfill&\kappa=0,\,
\lambda=\frac{q-3}{2}\text{ even}\hfill\\[2mm]
                                 z^{\frac{3\kappa-7}{2}};\hfill&
3\leq\kappa\leq q-2,\, \kappa\mbox{ odd},\,
\lambda=\frac{\kappa-3}{2}\text{ even}\hfill\\[2mm]
            z^{\frac{3\kappa-5}{2}}+z^{\frac{3\kappa-3}{2}};\hfill&
3\leq\kappa\leq q-2,\, \kappa\mbox{ odd},\,
\lambda=\frac{\kappa-1}{2}\text{ odd}\hfill\\[2mm]
            z^{\frac{3\kappa-3}{2}};\hfill&
1\leq\kappa\leq q-2,\, \kappa\mbox{ odd},\,
\lambda=\frac{\kappa-1}{2}\text{ even}\hfill\\[2mm]
                                 z^{\frac{3\kappa+3q-6}{2}};\hfill&
1\leq\kappa\leq q-2,\, \kappa\mbox{ odd},\,
\lambda=\frac{\kappa+q-2}{2}\hfill\\[2mm]
              z^{\frac{3\kappa-6}{2}}+z^{\frac{3\kappa-4}{2}};\hfill&
2\leq\kappa\leq q-3,\,\kappa\mbox{
even},\,\lambda=\frac{\kappa-2}{2}\text{ even}\hfill\\[2mm]
                                 z^{\frac{3\kappa-4}{2}};\hfill&
2\leq\kappa\leq q-3,\,\kappa\mbox{
even},\,\lambda=\frac{\kappa}{2}\hfill\\[2mm]
                                 z^{\frac{3\kappa+3q-7}{2}};\hfill&
2\leq\kappa\leq q-3,\,\kappa\mbox{
even},\,\lambda=\frac{\kappa+q-3}{2}\text{ even}\hfill\\[2mm]
                                 0;\hfill&\mbox{otherwise}\hfill
                                 \end{matrix}\right\}\ \mathrm{mod}\ 2.
\end{equation}

First suppose that $q\equiv1\ (\mathrm{mod}\ 4)$. For illustration, we display
$\Delta_{13}$ modulo~2:
$$
\setcounter{MaxMatrixCols}{15}
\left(\begin{matrix}
1& 0& 0& 0& 0& 0& 0& 0& 0& 0& 0& 0\\
1& 0& 0& 0& 0& 0& z^{18}& 0& 0& 0& 0& 0\\
1+z& z& 0& 0& 0& 0& z^{19}& 0& 0& 0& 0& 0\\
z& z^2+z^3& 0& 0& 0& 0& 0& z^{21}& 0& 0& 0& 0\\
0& 0& z^4& 0& 0& 0& 0& 0& 0& 0& 0& 0\\
0& 0& z^6& 0& 0& 0& 0& 0& z^{24}& 0& 0& 0\\
0& 0& z^6+z^7& z^7& 0& 0& 0& 0& z^{25}& 0& 0& 0\\
0& 0& z^7& z^8+z^9& 0& 0& 0& 0& 0& z^{27}& 0& 0\\
0& 0& 0& 0& z^{10}& 0& 0& 0& 0& 0& 0& 0\\
0& 0& 0& 0& z^{12}& 0& 0& 0& 0& 0& z^{30}& 0\\
0& 0& 0& 0& z^{12}+z^{13}& z^{13}& 0& 0& 0& 0& z^{31}& 0\\
0& 0& 0& 0& z^{13}& z^{14}+z^{15}& 0& 0& 0& 0& 0& z^{33}
\end{matrix}\right)$$
In the same way as in the proof of the preceding lemma, it can be seen
that there is exactly one permutation in $\mathrm{Sym}(\{0,1,\dots,q-2\})$ 
leading to a non-zero contribution in the expansion \eqref{eq:detdef}
of the determinant
of $\Delta_q$ modulo~2. This permutation is in fact the one in
\eqref{eq:perm}. The rest of the argument is completely analogous to
the one before.

Now let $q\equiv3\ (\mathrm{mod}\ 4)$. We illustrate this situation by displaying
$\Delta_{11}$ modulo~2: 
\begin{equation} \label{eq:det11}
\setcounter{MaxMatrixCols}{15}
\begin{pmatrix}
1& 0& 0& 0& z^{15}& 0& 0& 0& 0& 0\\
1& 0& 0& 0& 0& z^{15}& 0& 0& 0& 0\\
1+z& z& 0& 0& 0& 0& 0& 0& 0& 0\\
z& z^2+z^3& 0& 0& 0& 0& z^{18}& 0& 0& 0\\
0& 0& z^4& 0& 0& 0& z^{19}& 0& 0& 0\\
0& 0& z^6& 0& 0& 0& 0& z^{21}& 0& 0\\
0& 0& z^6+z^7& z^7& 0& 0& 0& 0& 0& 0\\
0& 0& z^7& z^8+z^9& 0& 0& 0& 0& z^{24}& 0\\
0& 0& 0& 0& z^{10}& 0& 0& 0& z^{25}& 0\\
0& 0& 0& 0& z^{12}& 0& 0& 0& 0& z^{27}
\end{pmatrix}.
\end{equation}
Here, there are more permutations leading to non-zero contributions in
the expansion \eqref{eq:detdef}
of $\det\Delta_q$ modulo~2. Namely, there are precisely
two non-zero entries modulo~2 in the $0$-th row, occurring in
columns~0 and $(q-3)/2$. If we decree that $\si(0)=0$, then there is precisely
one way to complete this into a permutation leading to a non-zero
contribution in the expansion of $\det\Delta_q$ modulo~2, namely the
permutation \eqref{eq:perm}. If, on the other hand, we set
$\si(0)=\frac {q-3} {2}$, then the sum of the contributions
corresponding to this subset of permutations is given by the minor of
$\Delta_q$ formed by deleting the $0$-th row and the $\frac {q-3}
{2}$-th column. This minor can be further reduced by observing that
every other column in the right half of the minor contains precisely
one non-zero entry modulo~2. 
(In \eqref{eq:det11}, these are the $5$-th, $7$-th, and $9$-th
column.)
In our running example, 
it is seen in this way that $\Delta_{11}$ modulo~2 equals the
determinant
$$
\setcounter{MaxMatrixCols}{15}
\det\begin{pmatrix}
1+z& z& 0& 0& 0& 0\\
z& z^2+z^3& 0& 0& z^{18}& 0\\
0& 0& z^4& 0& z^{19}& 0\\
0& 0& z^6+z^7& z^7& 0& 0\\
0& 0& z^7& z^8+z^9& 0& z^{24}\\
0& 0& 0& 0& 0& z^{25}\\
\end{pmatrix}.$$
From the form of the above matrix, it is obvious that its determinant is
equal to
$$\det\begin{pmatrix} 
1+z& z\\z& z^2+z^3
\end{pmatrix}\cdot z^{19}
\cdot\det\begin{pmatrix} 
z^6+z^7\\z^7& z^8+z^9
\end{pmatrix}\cdot z^{25}=\pm z^{60}(z^2+2z)^2\equiv z^{64}\ \mathrm{mod}\ 2.$$
In general, this pattern that the minor under consideration can be
expanded as a product of powers of $z$ and $2\times 2$-minors, each
of which contributes only a power of $z$, persists and, in the end,
leads to the result asserted in (i).

The argument for establishing the formulae for the minors is 
similar.\quad \quad \qed

\subsection{Proof of Lemma~\ref{Lem:GenLogDerComp}}
\label{Subsec:Lem23}
We shall use the formula
\begin{equation}
\label{Eq:Bell}
\frac{d^\nu}{dz^\nu} A(B(z)) = 
\underset{{\pi_1+2\pi_2+\dots+\nu \pi_\nu=\nu}}
{\sum_{\pi_1,\dots,\pi_\nu\ge0}}
\frac{\nu!}{\prod_{j\geq1} (j!)^{\pi_j} \pi_j!}
\,\bigg[\prod_{j\geq1}\big(B^{(j)}(z)\big)^{\pi_j}\bigg]
\,A^{(\pi_1+\dots+\pi_\nu)}(B(z))
\end{equation}
for the derivatives of a composite function, which is known as
the Fa\`a di Bruno formula (cf.\ \cite[Sec.~3.4]{ComtAA};
but see also \cite{CraiAA,JohWAE}). In view of the left-hand side, 
Formula~\eqref{Eq:Bell} implies in particular the integrality assertion 
in Item~(i). In fact, the coefficients occurring in (\ref{Eq:Bell}) have 
a natural combinatorial interpretation (see \cite[Theorem~13.2]{Andrews}). 

Applying (\ref{Eq:Bell})
with $A(t)=e^t$ and $B(z) = h^{-1}\int\mathcal{S}_\Gamma^H(z)\,dz$,
Equation~(\ref{Eq:HigherDerHComp}) follows immediately in view of
(\ref{Eq:WreathTransform}).

Finally, using the facts that
$\big(\mathcal{S}_\Gamma^H(z)\big)^{(j)}\equiv0\ (\mathrm{mod}\ 2)$ for $j\geq2$,
and that 
\[
\frac{\nu!}{(\nu-2\mu)! 2^\mu \mu!} =
\binom{\nu}{2\mu}\cdot\prod_{1\leq k\leq\mu}(2k-1) \equiv
\binom{\nu}{2\mu}\ \mathrm{mod}\ 2, 
\]
we find from (\ref{Eq:HigherDerHComp}) that, modulo~$2$,
\begin{align*}
h^\nu \Big(\frac{d^\nu}{dz^\nu}
\mathcal{H}_\Gamma^H(z)\Big)\Big/\mathcal{H}_\Gamma^H(z) &\equiv
\underset{\pi_1+2\pi_2=\nu}{\sum_{\pi_1,\pi_2\ge0}} h^{\nu-\pi_1-\pi_2}
\frac{\nu!}{\prod_{j=1}^2(j!)^{\pi_j} \pi_j!}
\prod_{j=1}^2\Big(\big(\mathcal{S}_\Gamma^H(z)\big)^{(j-1)}\Big)^{\pi_j}
\\[2mm]
&\equiv \sum_{\mu=0}^{\lfloor\frac{\nu}{2}\rfloor} h^\mu
\binom{\nu}{2\mu} \Big(\big(\mathcal{S}_\Gamma^H(z)\big)'\Big)^\mu
\big(\mathcal{S}_\Gamma^H(z)\big)^{\nu-2\mu}, 
\end{align*}
in accordance with (\ref{Eq:HigherDerHCompMod2}). Our last claim
(\ref{Eq:HigherDerHCompMod2Heven}) is an immediate consequence of
(\ref{Eq:HigherDerHCompMod2}).\quad \quad \qed

\end{document}